\documentclass[11pt]{article}
\usepackage{geometry}                
\usepackage{graphicx}
\usepackage{amssymb}
\usepackage{fullpage}
\usepackage{epstopdf}
\usepackage{multicol}
\usepackage{natbib}
 \bibpunct[, ]{[}{]}{,}{n}{}{,}%

\usepackage[breaklinks=true,colorlinks,citecolor=blue,linkcolor=blue]{hyperref}

\usepackage{amsthm}

\usepackage{graphicx}

\usepackage[english]{babel}
\usepackage{amssymb}
\usepackage{amsmath}
\usepackage{enumerate}
\usepackage{color}

\usepackage{epsfig}
\usepackage{subfigure}

\newcommand{\old}[1]{{}}

\newcommand{\bb}{\mathbb}

\newcommand{\R}{\bb R}

\newcommand{\N}{\bb N}

\DeclareMathOperator\lin{lin}
\DeclareMathOperator\rec{rec}
\DeclareMathOperator\supp{supp}
\DeclareMathOperator\intr{int}
\DeclareMathOperator\cl{cl}
\DeclareMathOperator\cone{cone}

\def\Item$#1${\item $\displaystyle#1$
   \hfill\refstepcounter{equation}(\theequation)}
\theoremstyle{definition}
\newtheorem{prop}{Proposition}[section]
\newtheorem{theorem}[prop]{Theorem}
\newtheorem{lemma}[prop]{Lemma}

\newtheorem{claim}[prop]{Claim}
\newtheorem{corollary}[prop]{Corollary}
\newtheorem{definition}[prop]{Definition}
\newtheorem{remark}[prop]{Remark}

\newtheorem{example}[prop]{Example}

\newtheoremstyle{TheoremNum}
        {\topsep}{\topsep}              
        {\itshape}                      
        {}                              
        {\bfseries}                     
        {.}                             
        { }                             
        {\thmname{#1}\thmnote{ \bfseries #3}}
    \theoremstyle{TheoremNum}
    \newtheorem{theorem-pre}{Theorem}

\numberwithin{equation}{section}

\def\st{\mid}

\title{Projection: A Unified Approach to Semi-Infinite Linear Programs and Duality in Convex Programming}
\author{Amitabh Basu \\
Kipp Martin \\
Christopher Thomas Ryan}

\begin{document}

\maketitle

\abstract{Fourier-Motzkin elimination is a projection algorithm for solving finite linear programs.  We extend Fourier-Motzkin elimination  to  semi-infinite linear programs  which are linear programs with finitely many variables and infinitely many constraints. Applying projection leads to new characterizations of important properties for primal-dual pairs of semi-infinite programs such as  zero duality gap, feasibility, boundedness, and solvability. Extending the Fourier-Motzkin elimination procedure to semi-infinite linear programs yields a new classification of variables that is used to determine the existence of duality gaps. In particular, the existence of what the authors term dirty variables can lead to duality gaps. Our approach has interesting applications in finite-dimensional convex optimization. For example, sufficient conditions for a zero duality gap, such as the Slater constraint qualification, are reduced to guaranteeing that there are no dirty variables. This leads to completely new proofs of such sufficient conditions for zero duality.}

\begingroup
\let\clearpage\relax
\section{Introduction}\label{s:introduction}

Duality is an important theoretical and practical topic in optimization.   In order to better understand the structure of an optimization problem (called the primal), and design solution algorithms, it is often useful to consider its dual (or duals). 
A key determinant of the usefulness of the dual is the \emph{duality gap} which is  the difference between the optimal value of a primal and the optimal value of the dual. Establishing that the primal and dual have \emph{zero duality gap} is particularly desirable and is a subject of intense study throughout the field of optimization.

Linear programming is a perfect example. Every linear program has a well-understood dual with the  simple property that when the primal is feasible with bounded optimal value, there is zero duality gap. Moreover, optimal solutions to both the primal and dual are guaranteed to exist. For more general problems, additional conditions are needed to establish zero duality gap and the existence of an optimal solution. 

Much research has focused on \emph{sufficient} conditions for zero duality gap. Possibly the most well-known sufficient condition for zero duality gap is the \emph{Slater constraint qualification} for convex programming. Slater's condition states that when there is feasible point that strictly satisfies all the convex constraints of the primal convex program (sometimes called a \emph{Slater point}) there is zero duality gap. ``Slater-like" conditions are also prevalent in conic programming, where the existence of interior points to the dual conic program guarantees a zero duality gap (see for instance, Gartner and Matous\'ek~\cite{gartner-matousek}). 
Less well-known is the duality theory of \emph{semi-infinite linear programs}. These are linear optimization problems with a finite number of variables and possibly infinitely many constraints. For surveys covering theory, applications and algorithms for semi-infinite linear programming see Goberna and L\'opez \cite{goberna2002linear}, Hettich and Kortanek \cite{hettich1993semi} and L\'opez~\cite{lopez2012stability}. In this paper we use  the duality theory of semi-infinite linear programs to understand the duality of both convex and conic programs. In semi-infinite linear programming, a variety of sufficient conditions for zero duality gap have been introduced (see for example, Anderson and Nash~\cite{anderson-nash}, Charnes, Cooper and Kortanek~\cite{charnes1963duality}, Duffin and Karlovitz~\cite{duffin-karlovitz65}, Goberna and L\'opez~\cite{goberna1998linear}, Karney~\cite{karney81}, Kortanek~\cite{kortanek1974classifying}, and Shapiro~\cite{shapiro2009semi}). We provide an alternate and unifying approach to duality in semi-infinite linear programs. 

We extend Fourier-Motzkin elimination (projection) \cite{fourier26,motzkin36} to semi-infinite systems of linear inequalities as a method to study duality. Taking the dictum expressed by Duffin and Karlovitz~\cite{duffin-karlovitz65} of ``the desirability of omitting topological considerations" to its logical conclusion, our approach does not rely on the theory of topological vector spaces or convex analysis. Instead, our approach combines simple aggregation of pairs of linear inequalities using nonnegative multipliers (an algebraic operation) with simple analysis on the reals, $\R$.  Applying Fourier-Motzkin elimination to  a semi-infinite linear  program reveals important properties about the semi-infinite linear program that can only be obtained through this elimination (or projection)  process.   In particular,  Fourier-Motzkin elimination  reveals the existence of what the authors term ``dirty'' variables.  Dirty variables  are necessary for the existence of a  duality gap. The dirty variable characterization also has important implications for finite dimensional problems.   For example, sufficient conditions for a zero duality gap in a finite-dimensional convex optimization problem, such as the Slater constraint qualification, are reduced to guaranteeing that there are no dirty variables in an appropriately defined semi-infinite linear program. To the author's knowledge, using Fourier-Motzkin elimination to study study semi-infinite linear programs is new. Blair~\cite{blair1974extension} employed a Fourier-Motzkin elimination technique to extend a result by Jersoslow and Kortanek~\cite{jeroslow1971semi} on the feasibility of semi-infinite linear systems. However, their systems were over the ordered field $\R(M)$ obtained by appending the reals $\R$ with a transcendental number $M$ larger than every real. By contrast, we consider systems over the reals $\R$. This allows us to study optimality and duality of semi-infinite linear programs. 

The extension of Fourier-Motzkin elimination to semi-infinite linear programs involves subtleties that do not arise in  standard Fourier-Motzkin theory where the number of inequalities is finite. Sections~\ref{s:fm-elim}~and~\ref{s:silp-classification}  provide a cogent framework for analyzing  semi-infinite linear programs.\old{This gives rise to new results on primal and dual feasibility, boundedness, solvability and duality that are summarized in Table~\ref{table:summary}.} Applying projection leads to new characterizations of important properties for primal-dual pairs of semi-infinite programs such as zero duality gap, feasibility and boundedness, and solvability.  These results can be leveraged to provide new proofs of some classical results in finite-dimensional conic linear programs and convex optimization.  See    Section \ref{s:application-conic}  and  Section \ref{s:application-convex}, respectively. Application of the results from Section~\ref{s:silp-classification} to the generalized Farkas' theorem for semi-infinite linear programs is given in Section~\ref{s-application-farkas-minkowski}.
Additional  sufficient conditions for zero duality gap in semi-infinite linear programs are in Section~\ref{s:app-countable} of the Electronic Companion. 

\paragraph{Notation.} Let $Y$ be a vector space. The \emph{algebraic dual} of $Y$, denoted $Y'$, is the set of linear functionals with domain $Y$. Let $\psi \in Y'$.   The evaluation of $\psi$ at $y$ is denoted by  $ \langle y, \psi \rangle$;  that is, $ \langle y, \psi \rangle = \psi(y)$.

Let $P$ be a convex cone in $Y$. A convex cone $P$ is {\em pointed} if and only if $P \cap -P = \left\{0\right\}$. A pointed convex cone $P$ in $Y$ defines a vector space ordering $\succeq_P$ of $Y$, with $y \succeq_P y'$ if  $y - y' \in P$. 
The \emph{dual cone} of $P$ is $P' = \left\{\psi \in Y : \langle y, \psi \rangle \ge 0 \text{ for all } y \in P\right\}$.
Elements of $P^{\prime}$ are called {\it positive linear functionals} in $Y$.
A cone $P$ is  {\em reflexive} if $P'' = P$ under the natural embedding of $Y \hookrightarrow Y''$. 

Let $A$ be a linear mapping  from vector space  $X$ to vector space  $Y$. The \emph{algebraic adjoint} $A' : Y' \to X'$ is defined by $A'(\psi) = \psi\circ A$ and satisfies $ \langle x, A'(\psi) \rangle = \langle A(x), \psi \rangle$ where $\psi \in Y'$ and $x \in X$. 

Given any set $I$, $2^I$ denotes the power set of $I$ and $\R^I$ denotes the vector space of real-valued functions $u$ with domain $I$, i.e., $u : I \to \R$. For $u \in \R^I$ the \emph{support} of $u$ is the set $\supp(u) = \left\{i \in I : u(i) \neq 0\right\}$. The subspace $\R^{(I)}$ are those functions in $\R^I$ with finite support. Let $\ge$ denote the standard vector space ordering on $\R^I$. That is, $u \ge v$ if and only if $u(i) \ge v(i)$ for all $i \in I$. The subspace $\R^{(I)}$ inherits this ordering. Let $\R_+^I$ (resp. $\R_+^{(I)}$) denote the pointed cone of $u \in \R^I$ (resp. $u \in \R_+^{(I)}$) with $u\ge 0$. 
Using the standard embedding of $\R^{(I)}$ into $(\R^I)'$ for $u \in \R^I$ and $v \in \R^{(I)}$,  write $\langle u, v \rangle = \sum_{i \in I} u(i)v(i)$. The latter sum is well-defined since $v$ has finite support.
For all $h \in I,$  define a function $e^h \in \R^I$ by $e^{h}(i) = 1$ if $i = h$, and $e^{h}(i) = 0$ if $i \neq h$. When $I = \{1,2, \ldots, n\}$,  $\R^I$  is  $\R^n$ and  $e^1, e^2, ... e^n$  correspond to the standard unit vectors of $\R^n$.

The optimal value of an optimization problem ($*$) is denoted by $v(*).$ If the objective of $(*)$ is a supremum and the problem is (i) unbounded then set $v(*) = \infty$ or (ii) infeasible then set $v(*) = -\infty$. Conversely, if the objective of $(*)$ is an infimum and the problem is (i) unbounded then set $v(*) = - \infty$ or (ii) infeasible then set $v(*) = \infty$.

\paragraph{Our results.} The main topic of study is the  semi-infinite program
\begin{align*}
\begin{array}{rl}
\qquad  \inf_{x\in \R^n} & c^\top x \\
 \textrm{s.t.} & \sum_{k = 1}^n a^k(i)x_k \geq b(i) \quad \text{ for } i\in I
\end{array}\tag{\text{SILP}}
\end{align*}
where $I$ is an arbitrary (potentially infinite) index set, $c \in \R^n$, and $b, a^k \in \R^I$ for $k = 1, \dots, n$, and its \emph{finite support dual}
\begin{align*}\label{eq:FDSILP}
\begin{array}{rrl}
\sup &\sum_{i \in I} b(i) v(i)&\\
   {\rm s.t.} & \sum_{i \in I} a^{k}(i) v(i)&=c_k  \quad \text{ for } k = 1, \ldots, n \\
& v&\in \R_+^{(I)}. 
\end{array}\tag{\text{FDSILP}}
\end{align*}
Our main results on this primal-dual pair are summarized in Table~\ref{table:summary} (see page \pageref{table1_page}).   These include a sufficient condition for primal solvability (Theorem~\ref{theorem:primal-solvability}) and characterizations of both dual solvability (Theorem~\ref{theorem:dual-solvability}) and zero duality gap (Theorem~\ref{theorem:zero-duality-gap}). Here, zero duality gap means $v(\ref{eq:SILP}) = v(\ref{eq:FDSILP})$ when \eqref{eq:SILP} is feasible.

We identify a special class of semi-infinite linear programs, termed \emph{tidy} semi-infinite linear programs, where zero duality gap is guaranteed to hold (Theorem~\ref{theorem:all-clean-system}). The name \emph{tidy} comes from the fact that the Fourier-Motzkin elimination procedure eliminates (or ``cleans up") all primal decision variables. In our terminology, there are no ``dirty" decision variables.  

\theoremstyle{definition}
\newtheorem*{theorem:all-clean-system}{Theorem~\ref{theorem:all-clean-system}}
\begin{theorem:all-clean-system}
If \eqref{eq:SILP} is feasible and tidy then
\begin{enumerate}[(i)]
\item \eqref{eq:SILP} is   solvable,
\item \eqref{eq:FDSILP} is feasible and bounded,
\item there is a zero duality gap for the primal-dual pair \eqref{eq:SILP} and \eqref{eq:FDSILP}.
\end{enumerate} 
\end{theorem:all-clean-system}


The theory of tidy semi-infinite linear programs is leveraged to obtain new proofs of important duality results in conic and convex programming. In the main text of this manuscript (Section~\ref{s:application-convex}), we discuss convex programming. Section~\ref{s:application-conic} in the appendix (electronic companion) contains the discussion for conic programs. We consider the following general convex program
\begin{align*}\label{eq:CP}
\begin{array}{rcl}
\max_{x\in \R^n} \qquad f(x) && \\
{\rm s.t.} \qquad g_i(x) &\ge& 0 \quad \text{ for } i = 1, \ldots, p \\
x  & \in  & \Omega
\end{array}\tag{\text{CP}}
\end{align*}
where $f(x)$ and $g_i(x)$ for $i = 1, \ldots, p$ are concave functions, and $\Omega$ is a closed, convex set. 
Define the Lagrangian function $L(\lambda) :=  \max \{ f(x) + \sum_{i=1}^p\lambda_i g_i(x)\, : \, x \in \Omega  \}.$ The Lagrangian dual is
\begin{align*}\label{eq:LD}
\inf_{\lambda \geq 0} L(\lambda). \tag{\text{LD}}
\end{align*}
The following is a well-known key result in finite-dimensional convex programming.
\theoremstyle{definition}
\newtheorem*{thm:slater-convex}{Theorem~\ref{thm:slater-convex}}
\begin{thm:slater-convex}[Slater's theorem for convex programs]
Suppose the convex program~\eqref{eq:CP} is feasible and bounded. Moreover, suppose there exists an $x^* \in \Omega$ such that $g_i(x^*) > 0$ for all $i=1, \ldots, p$. Then there is zero duality gap between the convex program~\eqref{eq:CP} and its Lagrangian dual~\eqref{eq:LD}. Moreover, there exists a $\lambda^* \geq 0$ such that $v(\ref{eq:LD}) = L(\lambda^*)$, i.e., the Lagrangian dual is solvable.
\end{thm:slater-convex}

We provide a completely new proof of this classical result in Section~\ref{s:application-convex}. Our proof uses the fact that the \emph{Slater point} $x^*$ corresponds to a useful constraint in the semi-infinite linear program representing the Lagrangian dual. The structure of this constraint  is used to show the tidiness of the system. By  Theorem~\ref{theorem:bounded-zero-duality-gap}, this  implies zero duality gap and dual solvability. 

Beyond these results in conic and convex programming, the method of projection is used to elegantly prove several foundational results for semi-infinite linear programs. 
These results  include the generalized Farkas' theorem for infinite systems of linear inequalities (see our Theorem~\ref{thm:gen-farkas-minkowski} and Theorem 3.1 in Goberna and L\'opez~\cite{goberna1998linear}). Our proof does not rely on the theory of separating hyperplanes and thus does not mimic known proofs. Goberna and L\'opez use this result as the main tool for deriving their own set of necessary and sufficient conditions for zero duality in semi-infinite linear programs. Thus, our methodology can, in principle, be used as an alternate starting point to derive their results. Other authors have also given characterizations of properties of primal-dual pairs of semi-infinite linear programs. A comparison is given in Section~\ref{ss:summary}.
Additional results on the finite approximability of semi-infinite linear programs are in Section~\ref{s:app-countable} of the Electronic Companion. 

\section{Fourier-Motzkin elimination}\label{s:fm-elim}

In this section we  extend  Fourier-Motzkin elimination to semi-infinite linear systems.  For background on Fourier-Motzkin elimination applied to finite linear systems see Fourier~\cite{fourier26},  Motzkin~\cite{motzkin36}, and Williams~\cite{williams86}. In this section, Fourier-Motzkin elimination is used to characterize the feasibility and boundedness of semi-infinite systems of linear inequalities. In addition,  useful properties are shown about the Fourier-Motzkin multipliers which appear while aggregating constraints.

Consider the semi-infinite linear system 
\begin{equation}\label{eq:initialSystem}
a^1(i)x_1 + a^2(i)x_2 + \cdots + a^n(i)x_n  \geq b(i) \quad\text{ for } i\in I
\end{equation}
where $I$ is an arbitrary index set. Denote the set of $(x_1,\dots,x_n) \in \R^n$ that satisfy these inequalities by $\Gamma$. 
The projection of  $\Gamma$  into the subspace of $\R^{n}$ spanned by $\{e^j\}_{j=2}^n$ is
\begin{eqnarray}
P(\Gamma; x_{1}) := \{(x_2, x_3, \ldots, x_n) \in \R^{n-1} \, : \,  \exists x_1 \in \R\; \textrm{ s.t. }\,  (x_1 , x_2, \dots, x_n) \in \Gamma\}.  \label{eq:defineProjection}
\end{eqnarray}
Under certain conditions, the projection $P(\Gamma ;x_{1})$ is characterized by aggregating inequalities in the original system.  
Define the sets 
\begin{equation}\label{eq:H+H-}
\begin{array}{c}
\mathcal{H}_+(k) := \{i \in I \st a^k(i) > 0\} \\
\mathcal{H}_-(k) := \{i \in I \st a^k(i) < 0\} \\
\mathcal{H}_0(k) := \{i \in I \st a^k(i) = 0\}
\end{array}
\end{equation}
based on the coefficients of variable $x_{k}$ in (\ref{eq:initialSystem}).

For now, assume $\mathcal{H}_+(1)$ and $\mathcal{H}_-(1)$ are both nonempty. As in the finite case, eliminate variable $x_{1}$ by adding all possible pairs of inequalities with one inequality in $\mathcal{H}_+(1)$ and the other from $\mathcal{H}_-(1)$. Since there are potentially infinitely many constraints this may involve aggregating an infinite number of pairs.  The resulting system is
\begin{align} 
\sum_{k=2}^n a^k(i) x_k & \geq b(i) & & \text{ for } i \in \mathcal{H}_0(1) \label{eq:new-system-1} \\ 
\sum_{k=2}^n \left(\dfrac{a^k(p)}{a^1(p)}- \dfrac{ a^k(q)}{a^1(q)}\right) x_k  &  \geq   \dfrac{b(p)}{a^1(p)} - \dfrac{b(q)}{ a^1(q)} & & \text{ for } p \in \mathcal{H}_+(1) \text{ and } q \in \mathcal{H}_-(1).
\label{eq:new-system-2}
\end{align}
Denote the set of $(x_2, \dots, x_n) \in \R^{n-1}$ that  satisfy the constraints in \eqref{eq:new-system-1}-\eqref{eq:new-system-2} by $FM(\Gamma; x_1)$.

\begin{remark}\label{rem:initial_remark}
One way to view the inequalities \eqref{eq:new-system-2} is the following : pick a pair $(p,q)$ of inequalities with $p \in \mathcal{H}_+(1)$ and $q \in \mathcal{H}_-(1)$.  Then form a new constraint by multiplying  the first constraint by $\frac{1}{a^1(p)}$, multiplying the second constraint by $-\frac{1}{a^1(q)},$ and adding them together. This ``eliminates'' $x_1$ from this pair of constraints. Of course, one can achieve this by choosing any common multiple of $\frac{1}{a^1(p)}$ and $-\frac{1}{a^1(q)}$ as the multipliers prior to adding them together, and achieve a ``scaled'' inequality describing the same halfspace (with $x_1$ ``eliminated'').\hfill $\triangleleft$
\end{remark}

A key result is that the inequalities in \eqref{eq:new-system-1}-\eqref{eq:new-system-2} describe the projected set $P(\Gamma; x_{1})$.

\begin{theorem}\label{theorem:FM-elim}
If  $\mathcal{H}_+(1)$ and $\mathcal{H}_-(1)$  are both nonempty, then  $P(\Gamma; x_{1})  = FM(\Gamma; x_1)$. 
\end{theorem}

\begin{proof}
Since  $\mathcal{H}_+(1)$ and $\mathcal{H}_-(1)$  are both nonempty,  
$$
\begin{array}{rl}
& (x_2, x_3, \ldots, x_n) \in  P(\Gamma; x_{1})\\
\Leftrightarrow & \exists x_1 \in \R\,\textrm{ such that }\,  a^1(i)x_1 + a^2(i)x_2 + \ldots + a^n(i)x_n \geq b(i) \,\textrm{ for }\, i \in I \\
\Leftrightarrow & \exists x_1 \in \R\,\textrm{ such that }\left\{ \begin{array}{l} \sum_{k=2}^n a^k(i) x_k \geq b(i) \quad \forall i \in \mathcal{H}_0 \textrm{ and } \\ x_1 \geq \frac{b(p)}{a^1(p)} - \sum_{k=2}^n \frac{a^k(p)}{a^1(p)}x_k, \; \forall p \in \mathcal{H}_+(1)\,\textrm{ and }\\ x_1 \leq \frac{b(q)}{a^1(q)} - \sum_{k=2}^n \frac{a^k(q)}{a^1(q)}x_k, \; \forall q \in \mathcal{H}_-(1) \end{array}\right\}\\
\Leftrightarrow & \left\{ \begin{array}{l} \sum_{k=2}^n a^k(i) x_k \geq b(i) \quad \forall i \in \mathcal{H}_0 \textrm{ and } \\ \frac{b(p)}{a^1(p)} - \sum_{k=2}^n \frac{a^k(p)}{a^1(p)}x_k \leq  \frac{b(q)}{a^1(q)} - \sum_{k=2}^n \frac{a^k(q)}{a^1(q)}x_k   \quad \forall p \in \mathcal{H}_+(1), \forall q \in \mathcal{H}_-(1) \end{array}\right\}\\
\Leftrightarrow & (x_2, x_3, \ldots, x_n) \in FM(\Gamma; x_1).
\end{array}
$$
Note that the second to last equivalence holds because both $\mathcal{H}_+(1)$ and $\mathcal{H}_-(1)$ are nonempty. 
\end{proof}

Equally as important to our theory is how ``dual information" is accrued during the process of elimination. The following result captures the essence of this idea.

\begin{corollary}\label{cor:FM-elim} If  $\mathcal{H}_+(1)$ and $\mathcal{H}_-(1)$  are both nonempty, then there exist an index set $\tilde{I}$ and $u^h \in \R_{+}^{(I)}$ for $h \in \tilde{I}$ such that the projection $P(\Gamma; x_{1})$ is 
$$
P(\Gamma ;x_{1}) = \{ (x_2, \ldots, x_n) \st \tilde{a}^2(h)x_2 + \cdots + \tilde{a}^n(h)x_n \geq \tilde{b}(h) \text{ for } h \in  \tilde{I}\}
$$ 
where $\tilde{b}, \tilde{a}^2, \ldots, \tilde{a}^n \in \R^{\tilde{I}}$ are given by
\begin{itemize}
\item[(i)] $\tilde{b}(h) = \langle b, u^h \rangle$ for all $h \in \tilde{I}$,

\item[(ii)] $\tilde{a}^k(h) = \langle a^k, u^h\rangle$ for all $k = 2, \ldots, n$ and $h \in \tilde{I}$,

\item[(iii)] $\langle a^1, u^h \rangle = 0$ for all $h \in \tilde{I}$.
\end{itemize}
\end{corollary}
\begin{proof}
By Theorem~\ref{theorem:FM-elim}, $P(\Gamma; x_{1})  = FM(\Gamma; x_1)$. We show ƒ$FM(\Gamma; x_1)$ has the required representation.  Since $\mathcal{H}_+(1)$ and $\mathcal{H}_-(1)$  are both nonempty,     take  $\tilde{I} = \mathcal{H}_0(1) \cup (\mathcal{H}_+(1) \times \mathcal{H}_-(1)) $. For each $h \in \mathcal{H}_0(1)$, take $u^h \in \R_{+}^{(I)}$ as the function with  value $1$ at $h$ and $0$ otherwise. For each $h = (p, q) \in \mathcal{H}_+(1) \times \mathcal{H}_-(1)$, take $u^h\in \R_{+}^{(I)}$ as the function $u^h : I \rightarrow \R$ defined by 
$$
u^h(i) = \left\{ \begin{array}{cr} \frac{1}{a^1(p)}, &  \text{ when } i =  p  \\  -\frac{1}{a^1(q)}, & \text{ when } i = q \\ 0, & \textrm{otherwise}. \end{array}\right.
$$ 
Now define $\tilde{b}, \tilde{a}^2, \ldots, \tilde{a}^n$ using the equations from (i) and (ii) in the statement of the corollary. The proof is then complete by observing that $FM(\Gamma; x_1) = \{ (x_2, \ldots, x_n) \st \tilde{a}^2(h)x_2 + \cdots + \tilde{a}^n(h)x_n \geq \tilde{b}(h) \text{ for } h \in \tilde{I}\}$ with these definitions.
\end{proof}

Below is a formal statement of Fourier-Motzkin elimination,   which  applies the above procedure sequentially for each variable.

\noindent \hrulefill
\begin{center}
\textsc{Fourier-Motzkin Elimination Procedure}
\end{center}

\noindent \textbf{Input}:   A semi-infinite linear inequality system
\begin{equation*}
a^1(i)x_1 + a^2(i)x_2 + \cdots + a^n(i)x_n  \geq b(i) \quad \text{ for } i\in I.
\end{equation*}

\noindent \textbf{Output}:   A  semi-infinite linear inequality system
\begin{eqnarray}\label{eq:output-system}
\tilde{a}^{\ell}(h) x_{\ell} + \tilde{a}^{\ell + 1}(h) x_{\ell + 1} + \cdots + \tilde{a}^{n}(h) x_{n} \ge \tilde{b}(h) \quad \text{ for } h \in \tilde{I}  \label{eq:dirty1}
\end{eqnarray}
with $\tilde I \subseteq 2^I$ and $\tilde a^k \in \R^{\tilde I}$.
\old{such that for every $k = \ell, \ldots, n$, $\tilde a^k \in \R^{\tilde I}$  is not the zero vector, i.e., $\tilde a^k(h) \neq 0$ for at least one $h \in \tilde I$. }The variables $x_\ell, \dots, x_n$\old{ are  dirty and} form a subset of the variables of the input system relabeled according to a permutation $\pi : \left\{1,\dots, n\right\} \rightarrow \left\{1, \dots, n\right\}$. 
We allow $\ell \in \left\{1, \dots, n, n+1\right\},$ interpreting $\ell = n+1$ to mean \old{there are no dirty variables and thus }that the left-hand side is zero.  
We also output a set of vectors $\{u^h \in \R^{(I)}_+ : h \in \tilde I\}$\old{, satisfying $\langle a^k, u^h\rangle = 0$ for all $h \in \tilde I$ and every $k = 1, \ldots, \ell-1$, and $\langle a^k, u^h \rangle = \tilde a^k(h)$ for all $h \in \tilde I$ and every $k = \ell, \ldots, n$}.
\vskip 5pt
\old{
\noindent \textbf{Procedure}:
\vskip 5pt
\begin{enumerate}[1.]
\item \textsc{Initialization}.
$M \leftarrow \left\{1,\dots, n\right\}$, $\tilde I \leftarrow I$, $\tilde a^k \leftarrow a^k$ for all $k \in M$, $\tilde b \leftarrow b$, and $j \leftarrow 1$. 
\vskip 5pt
\item \textsc{Elimination.} \label{step:elimination}
\begin{enumerate}[a.]
\item Determine the sets $\mathcal{H}_+(j)$, $\mathcal{H}_-(j)$ and $\mathcal{H}_0(j)$ given in \eqref{eq:H+H-}.
\item If $\mathcal{H}_+(j) \cup \mathcal{H}_-(j) \neq \emptyset$ but one of $\mathcal{H}_+(j)$ and $ \mathcal{H}_-(j)$ is empty 
go to step \ref{step:incremement-j}.
\item $\tilde I \leftarrow \mathcal{H}_0(j) \cup \left[\mathcal{H}_+(j) \times \mathcal{H}_-(j)\right]$ and $M \leftarrow M \setminus \{j\}$. If $M = \emptyset$ go to step~\ref{step:update-b}.
\item For each $k \in M$ define $\hat a^k : \tilde I \rightarrow \R$ by
\begin{align*}
\hat a^k(h) :=   
\left\{
\begin{array}{cl}
\tilde a^k(h) & \text{ for } h \in \mathcal H_0(j) \\
\frac{\tilde a^k(p)}{\tilde a^{j}(p)}+ \frac{\tilde a^k(q)}{\tilde a^{j}(q)} & \text{ for } h =(p,q) \in \mathcal{H}_+(j) \times \mathcal{H}_-(j)
\end{array} 
\right.
\end{align*}
and set $\tilde a^k \leftarrow \hat a^k$.
\item Define $\hat b : \tilde I \rightarrow \R$ by
\begin{align*}
\hat b(h) :=   
\left\{
\begin{array}{cl}
\tilde b(h) & \text{ for } h \in \mathcal H_0(j) \\
\frac{\tilde b(p)}{\tilde a^{j}(p)} + \frac{\tilde b(q)}{\tilde a^{j}(q)} & \text{ for } h =(p,q) \in \mathcal{H}_+(j) \times \mathcal{H}_-(\sigma(j))
\end{array} 
\right.
\end{align*}
and set $\tilde b \leftarrow \hat b$. \label{step:update-b}
\item $j \leftarrow j + 1$ \label{step:incremement-j}
\end{enumerate}
\vskip 5pt
\item \textsc{Termination:}
\vskip 3pt
\noindent If $j = n + 1$ then go to step~
\ref{step:output}. Else, go to step~\ref{step:elimination}.
\vskip 5pt
\item \textsc{Output formatting:} \label{step:output}
\vskip 3pt
\noindent
Upon termination $M$ is either empty or, for some $\ell \in \{1, \dots, L\}$, can be written $M = \left\{m_1, \dots, m_{\ell-1}\right\}$ where $m_i \in \left\{1,\dots, n\right\}$ with $m_i \le m_j$ for $i \le j$. 
\vskip 5pt
\begin{enumerate}[a.] 
\item If $M = \emptyset$, output the system
\begin{equation*}
0  \geq \tilde b(h) \quad \text{ for all } h\in \tilde I.
\end{equation*}
\item Else if $M \neq \emptyset$, relabel the variables and columns: $x_{m_i} \leftarrow x_{n - (\ell -1) + i}$  and $\tilde a^{m_i} \leftarrow \tilde a^{n - (\ell -1) + i}$  for $i = 1, \dots, \ell-1$. This defines the permutation $\pi$ described in the output. Now, construct the system
\begin{equation*}
\tilde a^\ell(h)x_\ell + \tilde a^{\ell + 1}(h)x_{\ell+1} + \ldots + \tilde a^n(i)x_n  \geq \tilde b(h) \quad \text{ for all } h\in \tilde I.
\end{equation*} 
\end{enumerate}
\end{enumerate}
\noindent \hrulefill
\vskip 10pt}
\noindent \textbf{Procedure}:
\vskip 5pt
\begin{enumerate}[1.]
\item \textsc{Initialization}:
$\mathcal D \leftarrow \left\{1,\dots, n\right\}$, $\tilde I \leftarrow \left\{ \left\{i\right\} \st i \in I \right\}$, $\tilde a^k(\{i\}) \leftarrow a^k(i)$ for all $i \in I$ and $k \in \mathcal D$, $\tilde b \leftarrow b$, and $j \leftarrow 1$.  For each $h \in \tilde I = I$, set $u^h  \leftarrow  e^{h}$. 

\vskip 5pt
\item \textsc{Elimination:} \label{step:elimination}
While  ($j \leq n$) do:
\begin{enumerate}[a.]
\item Define  the sets $\mathcal{H}_+(j)$, $\mathcal{H}_-(j)$ and $\mathcal{H}_0(j)$ as follows.
\[
\begin{array}{c}
\mathcal{H}_+(j) := \{h \in \tilde I \st \tilde{a}^j(h) > 0\} \\
\mathcal{H}_-(j) := \{h \in \tilde I \st \tilde{a}^j(h) < 0\} \\
\mathcal{H}_0(j) := \{h \in \tilde I \st \tilde{a}^j(h) = 0\}
\end{array}
\]
\item If $\mathcal{H}_+(j) \neq \emptyset$ and $\mathcal{H}_-(j) \neq \emptyset$ do:
 \begin{enumerate}[(i)] \item Set $\tilde I \leftarrow \mathcal{H}_0(j) \cup \left\{p \cup q \st p \in \mathcal{H}_+(j),\, q \in \mathcal{H}_-(j) \right\}$
 and $\mathcal D \leftarrow \mathcal D \setminus \{j\}$. 
\item For each $k \in \mathcal D$ define $\hat a^k : \tilde{I} \to \R$ by
\begin{align*}
\hat a^k(h) :=   
\left\{
\begin{array}{cl}
\tilde a^k(h) & \text{ for } h \in \mathcal H_0(j) \\
\frac{\tilde a^k(p)}{\tilde a^{j}(p)}- \frac{\tilde a^k(q)}{\tilde a^{j}(q)} & \text{ for } h = p \cup q \text{ where } p \in   \mathcal{H}_+(j),\, q \in \mathcal{H}_-(j)
\end{array} 
\right.
\end{align*}

\item For each $h \in \tilde I$, define $\hat u^h \in \R^{(I)}_+$ by
\begin{align*}
\hat u^h :=   
\left\{
\begin{array}{cl}
 u^h & \text{ for } h \in \mathcal H_0(j) \\
\frac{1}{\tilde a^{j}(p)} u^p- \frac{1}{\tilde a^{j}(q)} u^q & \text{ for } h = p \cup q \text{ where } p \in   \mathcal{H}_+(j),\, q \in \mathcal{H}_-(j)\end{array} 
\right.
\end{align*}
\item For each $k \in \mathcal D$, set $\tilde a^k \leftarrow \hat a^k$.  For each $h \in \tilde I$, set $u^h  \leftarrow \hat{u}^{h}$. 

\item Define $\hat b : \tilde I \rightarrow \R$ by
\begin{align*}
\hat b(h) :=   
\left\{
\begin{array}{cl}
\tilde b(h) & \text{ for } h \in \mathcal H_0(j) \\
\frac{\tilde b(p)}{\tilde a^{j}(p)} - \frac{\tilde b(q)}{\tilde a^{j}(q)} & \text{ for } h = p \cup q \text{ where } p \in   \mathcal{H}_+(j),\, q \in \mathcal{H}_-(j)
\end{array} 
\right.
\end{align*}
and set $\tilde b \leftarrow \hat b$. 
\label{step:update-b}
\end{enumerate}
end do. 

\item If $\mathcal{H}_+(j) \cup \mathcal{H}_-(j) = \emptyset$ then set $\mathcal D \leftarrow \mathcal D \setminus \{j\}$. 

\item $j \leftarrow j + 1$. \label{step:incremement-j}
\end{enumerate}
end do.
\vskip 5pt
\item \textsc{Output formatting:} \label{step:output}  Upon termination $\mathcal D$ is either empty or, for some $\ell \in \{1, \dots, n\}$, can be written $\mathcal D = \left\{d_1, \dots, d_{n-\ell+1}\right\}$ where $d_i \in \left\{1,\dots, n\right\}$ with $d_i \le d_j$ for $i \le j$. \old{These are the indices of the dirty variables, before relabeling. }
Let $\overline{\mathcal D} = \left\{1,\dots,n\right\} \setminus \mathcal{D} = \{ \bar d_1, \ldots, \bar d_{\ell-1} \}$ where $\bar d_i \in \left\{1,\dots, n\right\}$ and $\bar d_i \le \bar d_j$ for $i \le j$. \old{These are the indices of clean variables before relabeling. }
In other words, $\ell-1$ variables were eliminated and the remaining $n-\ell+1$ variables indexed by the indices in $\mathcal D$ are not eliminated.
\vskip 5pt
\begin{enumerate}[a.] 
\item If $\mathcal D = \emptyset$, output the system
\begin{equation*}
0  \geq \tilde b(h) \quad \text{ for } h\in \tilde I.
\end{equation*}
\item  Else if $\mathcal D \neq \emptyset$, reassign the indices in $\mathcal D$ by $d_i \leftarrow \ell -1 + i$ for $i = 1, \ldots, n - \ell + 1.$ If $\overline{\mathcal D}$ is nonempty, reassign the indices in $\overline{\mathcal D}$ by $\bar d_i \leftarrow i$ for $i = 1, \dots, \ell -1$. This defines the permutation $\pi$ described in the output. 
Now, construct the system
\begin{equation*}
\tilde a^\ell(h)x_\ell + \tilde a^{\ell + 1}(h)x_{\ell+1} + \ldots + \tilde a^n(h)x_n  \geq \tilde b(h) \quad \text{ for } h\in \tilde I.
\end{equation*} 
\end{enumerate}
\end{enumerate}
\noindent \hrulefill
\vskip 10pt
\begin{remark}
In the above procedure, $\tilde I$ is redefined in every iteration but remains a subset of $2^I$; in particular, a family of finite subsets of $I$. The domain $\tilde I$ of the functions $\tilde a^k$ are redefined correspondingly. In contrast, the domain $I$ of the functions $u^h$ for $h \in \tilde I$ is unchanged throughout. The superscript $h$ is the support of $u^h$.  \hfill $\triangleleft$
\end{remark}

Examples~\ref{ex:clean-unbounded}, \ref{example:not-primal-optimal}, \ref{example:primal-solvable}, \ref{example:primal-infeasible-dual-solvable} and \ref{ex:no-slater-duality} and Remark~\ref{rem:boundedness} below will illustrate various aspects of the Fourier-Motzkin elimination procedure.  

\begin{definition}[Clean and dirty variables]  At the end of the  Fourier-Motzkin procedure, the  variables  $x_1, \ldots, x_{\ell-1}$ are called {\em clean} variables and the variables $x_\ell, \ldots, x_n$ are called {\em dirty} variables.  Thus, a dirty variable is one that the Fourier-Motzkin procedure could not eliminate and a clean variable is one that the procedure could eliminate.
\end{definition}

\begin{definition}[Canonical form]\label{def:canonical} A semi-infinite linear system~\eqref{eq:initialSystem} is said to be in \emph{canonical form} if the permutation $\pi$ output by the Fourier-Motzkin elimination is the identity permutation. 
\end{definition}

\begin{lemma}\label{lemma:canonical} For every semi-infinite linear system, there exists a permutation of the variables that  puts it into canonical form. Moreover, if one applies the Fourier-Motzkin procedure to the original system and to the permuted system, they result in the same system of inequalities in the output. 
\end{lemma}
\begin{proof}
The permutation output by the Fourier-Motzkin procedure is one such desired permutation.
\end{proof}

\begin{remark}\label{remark:canonical}
In light of Lemma~\ref{lemma:canonical},  we may assume without loss of generality, that semi-infinite linear systems are always given in canonical form before applying  the Fourier-Motzkin elimination procedure. There may exist multiple permutations of the variables which put a given semi-infinite system into canonical form. Moreover, two different permutations may lead to systems in canonical form with a different number of clean and dirty variables. However, if a permutation reveals a dirty variable then at least one dirty variable will exist in every permutation. For details see Theorem~\ref{theorem:invariance-of-cleanliness} in the Electronic Companion. 
For our purposes, the permutation of variables does not effect any of our results. Any permutation that puts the semi-infinite linear system into a canonical form will suffice. \hfill $\triangleleft$
\end{remark}

\begin{definition}\label{def:finite-support-notation}
The finite support element,  $u^h$ for every $h\in \tilde I$,   that is generated by the Fourier-Motzkin elimination procedure    is called a   \emph{Fourier-Motzkin elimination multiplier}, or simply a \emph{multiplier}.  
\end{definition}

The key property of the Fourier-Motzkin elimination procedure is that it characterizes geometric projections. For $\ell \le n$ define 
\begin{align*}
P(\Gamma;x_1, \dots, x_{\ell-1}) := \{(x_\ell, \ldots, x_n) \in \R^{n-\ell + 1} \, : \,  \exists x_1, \dots, x_{\ell - 1}\; \textrm{ s.t. }\,  (x_1, \dots, x_{\ell-1}, x_\ell, \dots, x_n) \in \Gamma\}.
\end{align*}

\begin{theorem}\label{theorem:FM-elim-succ}
Apply the Fourier-Motzkin elimination procedure with input inequality system \eqref{eq:initialSystem} to produce output system \eqref{eq:output-system}. For all $h \in \tilde I$,   the finite-support multipliers $u^h \in \R_{+}^{(I)}$  generated by the  Fourier-Motzkin procedure satisfy
\begin{itemize}
\item[(i)] $\tilde{b}(h) = \langle b, u^h \rangle$,
\item[(ii)] $\tilde{a}^k(h) = \langle a^k, u^h\rangle$ for all $k = \ell, \ldots, n$, and
\item[(iii)] $\langle a^k, u^h \rangle = 0$ for all $k=1, \ldots, \ell-1$.
\end{itemize}
In addition, if  not  all variables are eliminated, and in the output system \eqref{eq:output-system}  $\ell \le n$, then 
\begin{align*}
P(\Gamma; x_1, \ldots, x_{\ell-1}) = \{ (x_\ell, \ldots, x_n) \st \eqref{eq:output-system} \text{ holds}\}.
\end{align*}
\end{theorem}

\begin{proof}
If $\ell = 1$, then only Step 2d. of the Fourier-Motzkin elimination procedure is executed and the original system remains unchanged so $\tilde{I} = I,$   $\tilde{a}^{k} = a^{k},$ $k = 1, \ldots, n$ and $\tilde{b} = b.$   Based on the initialization step,  $u^{h} = e^{h}$  for $h \in \tilde{I}$ and (i)-(iii) follow.  If $\ell \ge 2,$ since the system is in canonical form, the result follows from recursively applying Corollary~\ref{cor:FM-elim}.
\end{proof}

\begin{corollary}[Clean projection]\label{cor:clean-projection}
Let \eqref{eq:initialSystem} be a semi-infinite linear system and let $1 \le M < \min\{\ell, n\}$ where $\ell$ is the index of the first dirty variable in the output system \eqref{eq:output-system}. Suppose $M$ iterations of Step 2 of the Fourier-Motzkin elimination procedure yields the system (recall~\eqref{eq:initialSystem} is assumed to be in canonical form)
\begin{equation}\label{eq:residual-system}
\tilde a^{M+1}(h)x_{M+1} + \tilde a^{M + 2}(h)x_{M+2} + \cdots + \tilde a^n(i)x_n  \geq \tilde b(h) \quad \text{ for } h\in \tilde I.
\end{equation} 
Then $P(\Gamma; x_1, \ldots, x_{M}) = \{ (x_{M+1}, \ldots, x_n) \st \eqref{eq:residual-system} \text{ holds}\}$.
\end{corollary}
\begin{proof}
Follows from a finite number of applications of Corollary~\ref{cor:FM-elim}.
\end{proof}

Partition the index set $\tilde I$  in \eqref{eq:output-system},  into two sets $H_1 := \{ h \in \tilde I : \tilde a^k(h) = 0 \textrm{ for all } k \in \{\ell, \ldots, n\}\}$ and $H_2 := \tilde I \setminus H_1$. Rewrite~\eqref{eq:dirty1} as
\begin{eqnarray}
0 &\ge&  \tilde{b}(h) \quad \text{ for } h\in H_1\label{eq:defineI1} \\
\tilde{a}^{\ell}(h) x_{\ell} + \tilde{a}^{\ell+1}(h) x_{\ell+1} + \cdots + \tilde{a}^{n}(h) x_{n} &\ge& \tilde{b}(h)\quad \text{ for } h \in H_{2}.  \label{eq:defineI2}
\end{eqnarray}
If $H_2 = \emptyset$ (that is, $\ell = n+1$), then  system  \eqref{eq:defineI1}-\eqref{eq:defineI2} is a \emph{clean} system. Otherwise,  if $H_2 \neq \emptyset$, \eqref{eq:defineI1}-\eqref{eq:defineI2} is a \emph{dirty} system.   In a dirty  system, for any $k\in \{\ell, \ldots, n\}$, either $\tilde a^k(h) \ge 0$ for all $h\in H_2$, or $\tilde a^k(h) \le 0$ for all $h \in H_2$. Moreover, $\sum_{k=\ell}^n | \tilde a^k(h)| > 0$ for $h \in H_2$.  

\begin{definition}\label{def:x-delta}
Given a dirty system \eqref{eq:defineI1}-\eqref{eq:defineI2} 
and a real number $\delta \ge 0$, let $x(\delta;\ell)$ denote the tuple $(\bar x_\ell, \dots, \bar x_n)$ where for each $k \in \{ \ell, \dots, n\}$, $\bar x_k = \delta$ if $\tilde a^k(h) \ge 0$ for all $h \in H_2$ and $\bar x_k = - \delta$ otherwise. Let $x_k(\delta;\ell)$ denote the $k$th entry of $x(\delta;\ell)$.
\end{definition}

\begin{remark}\label{remark:finite_fm}
When $I$ is a finite set, the concept of a dirty variable is unnecessary.  In the finite case, there is always a value of $\delta$ such that $x(\delta, \ell)$ is a feasible solution to~\eqref{eq:defineI2}.  It is therefore legitimate to drop the constraints indexed by $H_{2}$ from further consideration.  Therefore, when implementing the Fourier-Motzkin procedure in the finite case, if variable $x_{k}$ is dirty, then one would drop all the constraints $h$ for which $\tilde{a}^{k}(h) > 0$ (or $\tilde{a}^{k}(h) < 0$).\hfill $\triangleleft$
\end{remark}

\begin{theorem}[Feasibility]\label{theorem:feasible}
Applying  Fourier-Motzkin elimination to  \eqref{eq:initialSystem}  results in system~\eqref{eq:defineI1}-\eqref{eq:defineI2}.  If $H_2 \neq \emptyset$ then the system is feasible (i.e. $\Gamma$ is nonempty) if and only if 
\begin{multicols}{2}
\begin{itemize}
\item[(i)] $\tilde{b}(h)  \le 0 \text{ for all } h \in H_{1}$, and
\item[(ii)] $ \sup_{h \in H_{2}} \tilde{b}(h) / \sum_{k=\ell}^{n} | \tilde{a}^{k}(h)| < \infty.$
\end{itemize}
\end{multicols}
\noindent Moreover, if $H_2 = \emptyset$ then $\Gamma$ is nonempty if and only if (i) holds. 
\end{theorem}

\begin{proof}
If $H_2 \neq \emptyset$, then $\Gamma$ is nonempty if and only if $P(\Gamma; x_1, \ldots, x_{\ell-1})$ is nonempty. By Theorem~\ref{theorem:FM-elim-succ},   $P(\Gamma; x_1, \ldots, x_{ \ell-1})$ is defined by~\eqref{eq:defineI1}-\eqref{eq:defineI2}. Therefore, it suffices to show \eqref{eq:defineI1}-\eqref{eq:defineI2} has a feasible solution if and only if conditions (i) and (ii) hold. Since (i) and \eqref{eq:defineI1} are equivalent it remains to show \eqref{eq:defineI2} holds if and only if (ii) holds. 

($\Longrightarrow$) 
For all $h \in H_2$, $\tilde b(h) \leq \sum_{k=\ell}^n \tilde a^k(h)\bar x_k \leq \sum_{k=\ell}^n |\tilde a^k(h)| |\bar x_k| \leq \delta (\sum_{k=\ell}^n |\tilde a^k(h)|)$.   This implies  for every $h \in H_2$,  $\tilde b(h)/\sum_{k=\ell}^n |\tilde a^k(h)| \leq \delta < \infty$ and this gives condition (ii).

($\Longleftarrow$) 
Assume (ii) holds. Thus, there exists a  
$\delta \geq \max\{0, \sup_{h \in H_2} \tilde{b}(h) / \sum_{k=\ell}^{n} | \tilde{a}^{k}(h)|\}.$
We show $x(\delta;\ell)$ satisfies \eqref{eq:defineI2}.  
For any $h\in H_2$, $\sum_{k =\ell}^n \tilde a^k(h) x_k(\delta;\ell) = \delta(\sum_{k=\ell}^n |\tilde a^k(h)|) \geq \tilde b(h)$, where the last inequality follows from the fact that $\delta \geq \sup_{h \in H_{2}} \tilde{b}(h) / \sum_{k=\ell}^{n} | \tilde{a}^{k}(h)|.$ 

Now consider the case $H_2 = \emptyset$. If the inequalities in the original system hold (that is, $\Gamma \neq \emptyset$) then the inequalities $0 \ge \tilde b(h)$ for $h \in H_1$ must also hold, since these inequalities are consequences of the original system. Thus, (i) holds.\old{This is the case only if $\tilde b(h) \le 0$.} Conversely, suppose $\tilde b(h) \le 0$ for all $h \in H_1$. Now, just before $x_n$ is eliminated in the Fourier-Motzkin elimination procedure ($x_n$ must be eliminated since $H_2 = \emptyset$) the system stored in the algorithm (after a scaling as stated in Remark~\ref{rem:initial_remark}) is
\begin{eqnarray} 
0 & \ge & \hat b(h) \text{ for } h \in \mathcal H_0(n) \label{this}\\
x_n &\ge&  \hat{b}(h') \text{ for } h' \in \mathcal H_+(n) \label{that}\\
-x_n &\ge&  \hat{b}(h'') \text{ for } h'' \in \mathcal H_-(n). \label{the-other}
\end{eqnarray}
When $x_n$ is eliminated, system \eqref{eq:defineI1}-\eqref{eq:defineI2} is derived with $\tilde b(h) =  \hat{b}(h') + \hat{b}(h'')$ where $h = (h',h'')$ for $h' \in  \mathcal H_+(n)$ and $h'' \in \mathcal H_-(n)$. By hypothesis, $\tilde b(h) \le 0$ for all $h \in H_1$ and  this implies  $\hat{b}(h') \le -\hat{b}(h'')$.   Then there exists an $x_n$ such that $\hat{b}(h') \le x_n \le - \hat{b}(h'')$ for all $h' \in \mathcal H_+(n)$ and $h'' \in \mathcal H_-(n)$ and  this  $x_n$ that satisfies \eqref{that} and \eqref{the-other}. Note that \eqref{this} holds by hypothesis since $\mathcal H_0(n) \subseteq H_1$. Thus, \eqref{this}-\eqref{the-other} is a feasible system. By Corollary~\ref{cor:clean-projection} this system is the projection $P(\Gamma; x_1, \dots, x_{n-1})$. Thus,  $P(\Gamma; x_1, \dots, x_{n-1})$ is nonempty and therefore $\Gamma$ is nonempty. 
\end{proof}

\begin{remark}\label{x-delta-feasible}
In the  proof of Theorem~\ref{theorem:feasible}  it was shown  that when $\Gamma$ is nonempty and 
$\delta \geq \max\{0, \sup_{h \in H_{2}} \tilde{b}(h) / \sum_{k=\ell}^{n} | \tilde{a}^{k}(h)|\},$ 
the tuple $x(\delta;\ell)$ as defined in Definition~\ref{def:x-delta} is feasible to \eqref{eq:defineI1}-\eqref{eq:defineI2} and thus can be extended to a feasible vector in $\Gamma$. This fact is used below. \hfill $\triangleleft$
\end{remark}

We next characterize the boundedness of the feasible set $\Gamma$.

\begin{theorem}[Boundedness]\label{theorem:region-boundedness} If \eqref{eq:initialSystem} defines a nonempty bounded set $\Gamma$ then, after applying the Fourier-Motzkin elimination procedure, the resulting system~\eqref{eq:defineI1}-\eqref{eq:defineI2} has  $H_2 = \emptyset$.
\end{theorem}

\begin{proof}
The result follows from Theorem~\ref{theorem:clean-lineality-equivalence} in the electronic companion, because in this case $\rec(\Gamma) = \lin(\Gamma) = \left\{0\right\}$.
\end{proof}

\begin{example}\label{ex:clean-unbounded}
The opposite implication in Theorem~\ref{theorem:region-boundedness} does not hold in general. For example, consider the linear system $-x_1 -x_2 \geq 0, x_1 + x_2 \geq 0$. The feasible region is the unbounded line $x_1 + x_2 = 0$; but $H_2$ is empty when applying the Fourier-Motzkin elimination procedure because the output is the degenerate system $0 \geq 0$. \hfill $\triangleleft$
\end{example}

Theorem \ref{theorem:all-rays}  below provides a very useful property about Fourier-Motzkin elimination multipliers that  plays a pivotal role in  establishing duality results in Section~\ref{s:dual-results}.  

\begin{theorem}\label{theorem:all-rays} 
Applying  Fourier-Motzkin elimination  to  \eqref{eq:initialSystem} gives  \eqref{eq:output-system}. Let $\bar u \in \R^{(I)}_+$ such that $\langle a^k, \bar u \rangle = 0$ for  $k = 1, \ldots, M$ with $\ell - 1 \le M \le n$.   Then, there exists a nonempty finite index set $\bar I \subseteq \tilde I$ such that for all $h \in \bar I$ the Fourier-Motzkin multipliers $u^h$ satisfy $\langle a^k , u^h \rangle = 0$ for $k = 1, \ldots, M$. Moreover, there exist scalars $\lambda_{h} \ge 0$ for $h \in \bar I$ so that $\overline{u} = \sum_{h \in \bar I}\lambda_h u^h$. 
\end{theorem}
\begin{proof}  Proceed by induction on $n$.   First prove the inductive step on $n$ and then the $n=1$ step.  We assume the result is true for an $n-1$ variable system and show that this implies the result is true for an $n$ variable system. Apply Fourier-Motzkin elimination to the $n-1$ variable  system 
\begin{equation}\label{eq:drop-a-variable}
a^1(i)x_1 + a^2(i)x_2 + \cdots + a^n(i)x_{n-1}  \geq b(i) \quad  \text{ for } i\in I,
\end{equation}
obtained by dropping the last column in system~\eqref{eq:initialSystem}.  The result is
\begin{eqnarray}
\hat{a}^{\ell_{n-1}}(h) x_{\ell_{n-1}} + \hat{a}^{\ell_{n-1} + 1}(h) x_{\ell_{n-1} + 1} + \cdots + \hat{a}^{n-1}(h) x_{n-1} \ge \hat{b}(h) \quad \text{ for } h \in \hat{I}
\end{eqnarray}
where $\ell_{n-1}$ denotes the first index of the dirty variables in the Fourier-Motzkin elimination output.  There are two cases to consider.

\vskip 7pt

\noindent\underline{\em Case 1: $M < n$.} 
Variable $\ell - 1$ is the last clean variable in~\eqref{eq:initialSystem}. The assumption that  $M < n$, together with the theorem hypothesis that $\ell - 1 \le M,$  implies  $\ell - 1 < n$ so  the last clean variable  in~\eqref{eq:initialSystem} is strictly less than variable $n$.   Then the last clean variable in~\eqref{eq:drop-a-variable} is the same as the last clean variable in~\eqref{eq:initialSystem}.     
This implies Fourier-Motzkin elimination applied to both systems yields identical multiplier vectors.  We invoke the induction hypothesis for the $n-1$ variable system~\eqref{eq:drop-a-variable}.  For this to be valid, all the hypotheses for the $n-1$ system must hold. Denote by $M_{n-1}$  the value of $M$  and $\ell_{n-1}$ the value of $\ell$ when the induction hypothesis is applied to~\eqref{eq:drop-a-variable}.  Since $\ell-1 \leq M < n$ and the index of the last clean variable for~\eqref{eq:initialSystem} is the same as the last clean variable for~\eqref{eq:drop-a-variable}, it is valid to set $M_{n-1} = M$ and  $\ell_{n -1} - 1 = \ell - 1$.  Because   Fourier-Motzkin  elimination applied to both systems yields identical multiplier vectors, the induction hypothesis implies that the Fourier-Motzkin multipliers also satisfy  the requirements of the theorem for the $n$ variable system.

\vskip 7pt

\noindent\underline{\em Case 2: $M = n.$}  In this case  $\langle a^{k}, \overline{u} \rangle = 0$ for $k = 1, \ldots, n.$  Therefore it is valid to apply the induction hypothesis to the $n-1$ variable system~\eqref{eq:drop-a-variable}   with $M_{n-1} = n -1$ and $\ell_{n-1} = \min\{\ell, n\}$. Then  there exist a finite index set $\left\{1,\dots, t\right\}  = \overline{I} \subseteq \hat I$ and multipliers $w^j$ such that $ \langle a^k, w^j \rangle = 0$ for all $k=1,\dots,  n-1$ and $j = 1, \dots, t$ and  scalars $\hat \alpha_j > 0$ such that 
\begin{align}\label{eq:express-bar-u}
\begin{array}{c}
\bar u = \sum_{j=1}^t \hat \alpha_j w^j.
\end{array}
\end{align} 
The multipliers  $w^j$, $j = 1, \ldots, t$, are used  to show that column $n$ is clean in~\eqref{eq:initialSystem} and that  $\overline{u}$ is a nonnegative combination of multipliers that result from eliminating this last column $n.$

  By Theorem~\ref{theorem:FM-elim-succ}, the scalars $ \langle a^n, w^j \rangle$ are among the coefficients on $x_n$ before that variable is processed  when Fourier-Motzkin elimination  is applied to  \eqref{eq:initialSystem}. We claim that either (i) $\langle a^n, w^j \rangle = 0$ for $j = 1, \dots, t$ or (ii) there exist $j^+,j^- \in \left\{1, \dots, t\right\}$ such that $ \langle a^n, w^{j^+} \rangle > 0$ and $ \langle a^n, w^{j^-} \rangle < 0$. This follows since conditions (i) and (ii) are exhaustive, indeed $0 = \langle a^n, \bar u \rangle = \sum_{j=1}^t \hat \alpha_j \langle a^n, \ w^j \rangle$ for $\hat \alpha_j > 0$ and so if $\langle a^n, w^j \rangle \ge 0$ for $j = 1,\dots, t$ (similiarly $\langle a^n, w^j \rangle \le 0$ for $j = 1,\dots, t$) then $\langle a^n, w^j \rangle = 0$ for $j = 1, \dots, t$. 

If  (i) holds, and $\langle a^n, w^j \rangle = 0$ for $j = 1, \dots, t,$ then $\langle a^k, w^j \rangle = 0$ for $j = 1, \dots, t$, $k = 1,\ldots, n$; thus $w^j$ for $j=1, \ldots, t$ are Fourier-Motzkin multipliers when Fourier-Motzkin is applied to \eqref{eq:initialSystem}, and $\bar u = \sum_{j=1}^t \hat \alpha_j w^j$ and Case 2  is proved.

If (ii) holds then  $x_n$ is a clean variable with respect to the system produced during the Fourier-Motzkin procedure before variable $x_n$ is processed: it has both a positive coefficient $ \langle a^n, w^{j^+} \rangle > 0$ and a negative coefficient $ \langle a^n, w^{j^-} \rangle < 0$.   

Define three sets $J^+$, $J^-$ and $J^0$ where $j \in J^+$ if $\langle a^n, w^j \rangle > 0$, $j \in J^-$ if $\langle a^n, w^j \rangle < 0$ and $j \in J^0$ if $\langle a^n, w^j \rangle = 0$.  In case  (ii) both $J^+$ and $J^-$ are nonempty. As discussed in case (i), for $j \in J^0$, $w^j$ is already a Fourier-Motzkin multiplier which satisfies $ \langle a^k, w^j  \rangle = 0$ for $k = 1, \dots, M$ and so they meet the specifications of the theorem.  Now consider the $w^j$ for $ j \in J^+$ and $j \in J^-$.  Each pair of $(j^+,j^-) \in J^+ \times J^-$  yields a final Fourier-Motzkin multiplier which is a conic combination of $w^{j^+}$ and $w^{j^-}$. In order to simplify the analysis, normalize the  $w^{j}$ so that $ \langle a^n, w^j \rangle = 1$ for $j \in J^+$ and $ \langle a^n, w^j \rangle = -1$ for $j \in J^-$. Let  $\alpha_{j}$ be the multipliers after the corresponding scaling of $\hat{\alpha}_j$ for $j \in J^+ \cup J^-$. With this scaling, from Step 2.b.(iii) of the Fourier-Motzkin procedure, the $u^{j^+ j^-} = w^{j^+} + w^{j^-}$ for all $(j^+,j^-) \in J^+ \times J^-$ are among the Fourier-Motzkin elimination multipliers for the full system.  It suffices to show  that there exist multipliers $\theta_{j^+j^-}$ such  that
\begin{align}\label{eq:sum-condition}
\bar u = \sum_{j \in J^0} \hat \alpha_j w_j + \sum_{j^+\in J^+} \sum_{j^-\in J^-} \theta_{j^+j^-} u^{j^+ j^-}
\end{align}
and 
\begin{align}\label{eq:zero-condition}
\langle a^k,     u^{j^+ j^-} \rangle  =  \langle a^k, w^{j^+} + w^{j^-} \rangle = 0 \text{ for } k = 1, \dots, M.
\end{align}  
 Condition \eqref{eq:zero-condition} follows since $\langle a^k, w^j \rangle = 0$ for $k = 1, \dots, M-1$ and $ \langle a^n, w^{j^+} \rangle = - \langle a^n, w^{j^-} \rangle = 1$ for all $j^+ \in J^+$ and $j^- \in J^-$. 

To establish \eqref{eq:sum-condition} consider a transportation linear program with supply  nodes indexed by $J^{+}$ and demand nodes indexed by $J^{-}.$ Each supply node $j \in J^{+}$  has supply $\alpha_{j}.$  Each demand node $j \in J^{-}$  has demand $-\alpha_{j}.$  Since
\begin{eqnarray*}
0 = \langle a^{n}, \bar u \rangle = \langle a^{n}, \sum_{j \in J^0}  \alpha_j w^j + \sum_{j \in J^+ \cup J^-} \alpha_j w^j \rangle = \sum_{j \in J^+ \cup J^-} \alpha_{j} \langle a^{n},  w^{j} \rangle      =  \sum_{j \in J^{+}} \alpha_{j}   - \sum_{j \in J^{-}} \alpha_{j}  
\end{eqnarray*}
total supply is equal to total demand.  Therefore the transportation problem has a feasible solution  $\theta_{j^{+}, j^{-}}$ which is the flow from supply node $j^{+}$ to demand node $j^{-}.$  This feasible flow satisfies $\sum_{j^{-} \in J^{-}} \theta_{j^{+}, j^{-}}  = \alpha_{j^{+}}$ for $j^{+} \in J^{+}$ and 
$\sum_{j^{+} \in J^{+}} \theta_{j^{+}, j^{-}}  = \alpha_{j^{-}}$ for $j^{-} \in J^{-}$. 
and so
\begin{eqnarray*}
\sum_{j^{+} \in J^{+}} \sum_{j^{-} \in J^{-}} \theta_{j^{+}, j^{-}} u^{j^{+},j^{-}} &=& \sum_{j^{+} \in J^{+}} \sum_{j^{-} \in J^{-}} \theta_{j^{+}, j^{-}} (w^{j^{+}} + w^{j^{-}}) \\
&=& \sum_{j^{+} \in J^{+}} \sum_{j^{-} \in J^{-}} \theta_{j^{+}, j^{-}} w^{j^{+}} +  \sum_{j^{+} \in J^{+}} \sum_{j^{-} \in J^{-}} \theta_{j^{+}, j^{-}} w^{j^{-}} \\
&=&  \sum_{j^{+} \in J^{+}} \alpha_{j^{+}}w^{j^{+}} +  \sum_{j^{-} \in J^{-}}  \alpha_{j^{-}}w^{j^{-}}.
\end{eqnarray*}
Combining this with \eqref{eq:express-bar-u} yields \eqref{eq:sum-condition}.

Next, consider the case $n = 1$. By hypothesis, this forces $M=1$, i.e., $\langle a^1, \overline{u}\rangle = 0$.  
If the coefficient of $x_1$ is zero for all constraints indexed by $\supp(\overline{u})$, then the Fourier-Motzkin procedure initialization step gives multiplers  $w^j  = e^{j}$, $j \in \supp(\overline{u})$.  Then  $\overline{u} = \sum_{j \in \supp(\overline{u})}\overline{u}(j) w^j$. Otherwise, if variable $x_{1}$ has nonzero coefficients in the system indexed by $\supp(\overline{u})$, it follows that variable $x_{1}$ has both positive and coefficients in this system, since $\overline{u}$ is nonegative and $\langle a^1, \overline{u} \rangle$. Define the usual multiplier vector for each pair of positive and negative coefficients. Again, assume without the loss, the rows are scaled such that the positive coefficients are 1 and the negative coefficients -1.  Create a transportation problem  as above where each node has supply $\overline{u}_{j}$ if $j$ corresponds to a row with +1, or demand $-\overline{u}_{j}$ corresponds to a row with a -1.  Solving this  transportation problem, and using the same logic as before, gives the coefficients $\theta_{j^{+}, j^{-1}}$ to be used on the multiplier vectors $u^{j^{+},j^{-}}$ in order to generate $\overline{u}.$ 
\end{proof}
\section{Solvability and duality theory using projection}\label{s:silp-classification}

\subsection{The projected system}\label{s:projected-system}

The  semi-infinite linear program
\begin{align*}\label{eq:SILP}
\begin{array}{rl}
  \qquad  \inf_{x\in \R^n} & c^\top x \\
 \textrm{s.t.} & a^1(i)x_1 + a^2(i)x_2 + \cdots + a^n(i)x_n \geq b(i) \quad \text{ for } i\in I
\end{array}\tag{\text{SILP}}
\end{align*}
is the primal problem.  Reformulate \eqref{eq:SILP} as
\begin{eqnarray}  \qquad \inf \phantom{-  c_{1} x_{1} -  c_{2} x_{2} - \cdots - c_{n} x_{n} + } z  && \label{eq:initial-system-obj}   \\
\textrm{s.t.}  \quad -  c_{1} x_{1} -  c_{2} x_{2} - \cdots - c_{n} x_{n} + z   &\ge&   0   \label{eq:initial-system-obj-con} \\
  a^{1}(i) x_{1} + a^{2}(i) x_{2} + \cdots + a^{n}(i) x_{n} \phantom{ + z} &\ge&  b(i)\quad \text{ for } i \in I.  \label{eq:initial-system-con}
\end{eqnarray}
Let $\Lambda \subseteq \R^{n+1}$ denote the set of $(x_1, \dots, x_n, z)$ that satisfy \eqref{eq:initial-system-obj-con}-\eqref{eq:initial-system-con}. Consider $z$ as the $(n+1)$st variable and  constraint \eqref{eq:initial-system-obj-con} as the $0$th constraint in the system. For this to make sense we assume without loss of generality that $0$ is not an element of $I$. 

Applying  Fourier-Motzkin elimination procedure  to  the input system \eqref{eq:initial-system-obj-con}-\eqref{eq:initial-system-con} gives the output system \eqref{eq:output-system}, rewritten as
\begin{equation}\label{eq:J_system} 
\begin{array}{rcl}
0 &\ge&  \tilde{b}(h), \quad  h \in I_{1} \\
\tilde{a}^{\ell}(h) x_{\ell} + \tilde{a}^{\ell+1}(h) x_{\ell+1} + \cdots + \tilde{a}^{n}(h) x_{n} \phantom{+z} &\ge& \tilde{b}(h), \quad h \in I_{2} \\
z    &\ge&   \tilde{b}(h), \quad  h \in I_{3} \\
\tilde{a}^{\ell}(h) x_{\ell} + \tilde{a}^{\ell+1}(h) x_{\ell+1} + \cdots + \tilde{a}_{n}(h) x_{n} +z &\ge& \tilde{b}(h), \quad h \in I_{4}
\end{array}
\end{equation}
where $I_1$, $I_2$, $I_3$ and $I_4$ are disjoint with $\tilde I = I_1 \cup \cdots \cup I_4$. Note that $z$ can never be eliminated, so system \eqref{eq:J_system} is always dirty and $I_3 \cup I_4 \neq \emptyset$. This formatting also assumes that every time a constraint involving $z$ was aggregated, a multiplier of $1$ is used. This can always be achieved by Remark~\ref{rem:initial_remark}. It is possible that all other variables can be eliminated when $I_2 = I_4 = \emptyset$ (that is, $\ell = n+1$).  By construction, $|\sum_{k = \ell}^n\tilde{a}^k(h)| > 0$ for all $h \in I_2 \cup I_4$.

It is worth noting that the $\tilde a^k(h)$ and $\tilde b(h)$ in this section are different from those in Section~\ref{s:fm-elim}. Indeed, by including the constraint \eqref{eq:initial-system-obj-con} and enforcing the rule that a coefficient of $1$ on $z$ is maintained, the resulting output system will be different than if the Fourier-Motzkin elimination procedure was undertaken on 
\eqref{eq:initial-system-con} alone.

By Theorem~\ref{theorem:FM-elim-succ}, system \eqref{eq:J_system} describes the projection $P(\Lambda;x_1, \dots, x_{\ell-1})$ (recall  the assumption that the system of inequalities \eqref{eq:initial-system-obj-con}-\eqref{eq:initial-system-con} is in canonical form). 
Therefore, to solve \eqref{eq:SILP} it suffices to consider the optimization problem 
\begin{equation}\label{eq:new-problem}
\begin{array}{rl} 
\inf_{z,x_{\ell},\dots, x_{n}} & z \\
\textrm{s.t.} & \eqref{eq:J_system}.
\end{array}
\end{equation}
A further step (Lemma~\ref{lemma:z-projection}) is to examine the geometric projection of $\Lambda$ onto the $z$-variable space   in terms of the data from the output system \eqref{eq:J_system}. It is easier to characterize the boundedness and solvability of~\eqref{eq:SILP} in  this one-dimensional space.

As mentioned in the introduction, other authors have made  systematic study of semi-infinite programming duality using machinery other than Fourier-Motzkin (see, for instance, Goberna and L\'opez \cite{goberna1998linear} and Kortanek \cite{kortanek1974classifying}). We are not the first to provide characterizations of zero duality gap, dual solvability, etc. However, the content of our characterizations are new. We refer to the specifics of system \eqref{eq:J_system}, which has not previously appeared in the literature. A brief comparison of our results with those extant in the literature can be found in Section~\ref{ss:summary}.

\subsection{Primal results}\label{s:primal-results}

\subsubsection{Primal feasibility}\label{s:primal-feasibility}

Feasibility of \eqref{eq:SILP} is determined by looking at the constraints indexed by  $I_1, I_2, I_3$ and $I_4$.

\begin{theorem}[Primal Feasibility]\label{theorem:primalFeasible}\eqref{eq:SILP} is feasible if and only if 
\begin{multicols}{2}
\begin{enumerate}[(i)]
\item $\tilde{b}(h) \le 0$ for all $h \in I_{1}$,
\item \(\displaystyle \sup_{h \in I_2} \frac{\tilde{b}(h)}{\sum_{k=\ell}^{n} |\tilde{a}^{k}(h)|}  < \infty,\) 
\item \(\displaystyle \sup_{h \in I_3} \tilde b(h) < \infty\), 
\item \(\displaystyle \sup_{h \in I_4} \frac{\tilde{b}(h)}{\sum_{k=\ell}^{n} |\tilde{a}^{k}(h)|+1} < \infty \).
\end{enumerate}
\end{multicols}
\end{theorem}
 
\begin{proof}
The result follows directly from applying Theorem \ref{theorem:feasible} to the dirty system \eqref{eq:J_system} with $H_1 = I_1$ and $H_2 = I_2 \cup I_3 \cup I_4$. \end{proof}

Corollary~\ref{cor:primalFeasible} below states some consequences of primal feasibility for (SILP) which are useful later. The proof is analogous to the proof of Theorem~\ref{theorem:feasible}. First introduce the function 
\begin{align}\label{eq:f_delta}
\omega(\delta) := \sup_{h \in I_4} \big\{\tilde{b}(h) - \delta \sum_{k = \ell}^n |\tilde{a}^k(h)| \big\}
\end{align}
that is used throughout the paper. Note $\omega$ can take values in the extended reals. If  $I_4 = \emptyset$ then $\omega(\delta) = -\infty$. However, we show in the following corollary that if \eqref{eq:SILP} is feasible then $\omega(\delta)$ cannot diverge to $\infty$.
Observe $\omega$ is a nonincreasing function of $\delta$ since $\sum_{k = \ell}^n |\tilde{a}^k(h)| \ge 0$. 

\begin{corollary}\label{cor:primalFeasible}
If  \eqref{eq:SILP} is feasible then
\begin{enumerate}[(i)]
\item \(\displaystyle \delta_2 := \sup_{h \in I_2} \frac{\tilde{b}(h)}{ \sum_{k=\ell}^{n} |\tilde{a}^{k}(h)|}  < \infty,\)
\item \(\displaystyle \delta_3 := \sup_{h \in I_3} \tilde b(h) < \infty,\)
\item \(\displaystyle \lim_{\delta \to \infty} \omega(\delta) < \infty, \)
\end{enumerate}
\begin{itemize}
\item[(iv)] $(x(\bar\delta; \ell), \bar z) \in P(\Lambda;x_1, \dots, x_{\ell - 1})$ for all $\bar \delta, \bar z \in \R$ such that $\bar\delta \ge \max \left\{0,\delta_2 \right\}$ and $\bar z \ge \max\{\delta_3, \omega(\bar\delta)\}$. Moreover, by conditions i), ii) and iii) above, at least one such pair $(\bar \delta, \bar z)$ of real number exists.
\end{itemize}
\end{corollary}

\begin{proof}
Conditions i)-ii) follow immediately from Theorem~\ref{theorem:primalFeasible}. Condition iii) follows from the claim below and condition iv) of Theorem~\ref{theorem:primalFeasible}.

\begin{claim}\label{claim:sup-lim-equivalence}
\( \displaystyle
\sup_{h \in I_4} \frac{\tilde{b}(h)}{\sum_{k=\ell}^{n} |\tilde{a}^{k}(h)| + 1}  < \infty
\iff
\lim_{\delta \rightarrow \infty} \omega(\delta) < \infty. \) 
\end{claim}
\begin{proof}[Proof of Claim]\renewcommand{\qedsymbol}{} 
($\Longrightarrow$) Let $\bar \delta = \sup_{h \in I_4} \tilde{b}(h) / (\sum_{k=\ell}^{n} |\tilde{a}^{k}(h)| + 1) < \infty$. This implies $\bar \delta \geq \tilde{b}(h) / (\sum_{k=\ell}^{n} |\tilde{a}^{k}(h)| + 1)$ for every $h \in I_4$. Rearranging, $\bar\delta  (\sum_{k=\ell}^{n} |\tilde{a}^{k}(h)| + 1) \geq \tilde{b}(h)$, which implies $\bar\delta \geq \tilde{b}(h) - \bar\delta  (\sum_{k=\ell}^{n} |\tilde{a}^{k}(h)|)$ for all $h \in I_4$. Thus,   $\bar\delta \geq \sup\{ \tilde{b}(h) - \bar\delta  (\sum_{k=\ell}^{n} |\tilde{a}^{k}(h)|) \, : \,  \,  h \in I_{4}   \}= \omega(\bar\delta)$. Thus, $\infty > \bar\delta \geq \omega(\bar\delta)$ and since $\omega(\delta)$ is a nonincreasing function, this yields $\lim_{\delta \to \infty} \omega(\delta) < \infty$.

($\Longleftarrow$) Since $\lim_{\delta\to\infty} \omega(\delta) < \infty$ and $\omega(\delta)$ is a nonincreasing function, there exists a $\bar \delta < \infty$ such that $\bar\delta \geq \omega(\bar \delta)$. Indeed, having $w(\delta) > \delta$ for all $\delta$ contradicts $\lim_{\delta\to\infty} \omega(\delta) < \infty$. 
Since $\omega(\delta)$ is nonincreasing in $\delta$, $\omega(\bar\delta)\le \omega(\hat\delta)=c\le \bar\delta$.  Now, because $\bar \delta \ge \omega(\bar \delta)$ it follows $\bar\delta \geq \sup\{ \tilde{b}(h) - \bar\delta  (\sum_{k=\ell}^{n} |\tilde{a}^{k}(h)|) \, : \,  \,  h \in I_{4}   \}$. Hence, $\bar\delta \geq \tilde{b}(h) - \bar\delta  (\sum_{k=\ell}^{n} |\tilde{a}^{k}(h)|)$ for $h \in I_4$. Rearranging, $\bar\delta \geq \tilde{b}(h) / (\sum_{k=\ell}^{n} |\tilde{a}^{k}(h)| + 1)$ for $h \in I_4$ and so $\infty > \bar\delta \geq \sup_{h \in I_4} \tilde{b}(h) / (\sum_{k=\ell}^{n} |\tilde{a}^{k}(h)| + 1)$.\quad $\dagger$
\end{proof}

Prove condition iv) in the statement of the corollary by verifying  that the constraints indexed by $I_1, I_2, I_3$ and $I_4$ are satisfied by $(x(\bar\delta; \ell), \bar z) \in P(\Lambda;x_1, \dots, x_{\ell - 1})$ when $\bar\delta \ge \max \left\{0,\delta_2 \right\}$ and $\bar z \ge \max\{\delta_3, \omega(\bar\delta)\}$. Since $\delta_2, \delta_3$ and $\lim_{\delta \to \infty} \omega(\delta)$ are all finite, and $\omega(\delta)$ is a nonincreasing function, there exists at least one such pair $(\bar\delta, \bar z)$ of real numbers.

Since \eqref{eq:SILP} is feasible, the constraints in $I_1$ are satisfied by condition i) in Theorem~\ref{theorem:primalFeasible}.  By definition, 
$
\delta_2 \ge \tilde{b}(h) / \sum_{k=\ell}^{n} |\tilde{a}_{k}(h)|
$
for all $h \in I_2$,
which implies
$
 \bar\delta \sum_{k=\ell}^{n} |\tilde{a}_{k}(h)| \ge \tilde{b}(h)
$
for all $h \in I_2$. 
Since $\sum_{k=\ell}^{n} \tilde{a}_{k}(h)x_k(\bar\delta;\ell) = \bar\delta \sum_{k=\ell}^{n} |\tilde{a}_{k}(h)|$ by construction of $x(\bar\delta;\ell)$, $(x(\bar\delta; \ell), \bar z)$ satisfies the constraints indexed by $I_2$ in \eqref{eq:J_system}. 

Since $\bar z \ge \delta_3$, all the constraints indexed by $I_3$ are satisfied. Finally, since $\bar z \geq \omega(\bar\delta)$, 
$
\bar z \ge \sup_{h \in I_4} \{\tilde b(h) - \bar \delta \sum_{k = \ell}^n |a^k(h)| \}
$
and so for all $h \in I_4$,
$\bar z + \sum_{k = \ell}^n a^k(h)x_k(\bar\delta;\ell) = \bar z  + \bar \delta \sum_{k = \ell}^n |a^k(h)|  \ge \tilde b(h).$ Conclude $(x(\bar\delta; \ell), \bar z))$ satisfies the constraints indexed by $I_4$, and therefore feasible to \eqref{eq:J_system}. Thus, $(x(\bar\delta; \ell), \bar z)) \in P(\Lambda;x_1, \dots, x_\ell)$ by Theorem~\ref{theorem:FM-elim-succ}.\end{proof}

\subsubsection{Primal boundedness}\label{s:primal-boundedness}

To establish boundedness and solvability, we start by giving a characterization of the closure of the projection of the feasible region described by  \eqref{eq:J_system} onto the $z$-variable space.

\begin{lemma}\label{lemma:z-projection}
Assume \eqref{eq:SILP} is feasible and  applying  Fourier-Motzkin elimination  to \eqref{eq:initial-system-obj-con}-\eqref{eq:initial-system-con} gives  \eqref{eq:J_system}. Let $P(\Lambda; x_1, \dots, x_n)$ denote the projection of $\Lambda$ into the $z$-variable space. Then, the closure of $P(\Lambda; x_1, \dots, x_n)$ is  given by the  system of inequalities
\begin{align}
z & \ge \sup_{h \in I_3} \tilde{b}(h) \label{eq:sup-I3}\\
z & \ge \lim_{\delta \rightarrow \infty} \omega(\delta) .\label{eq:lim-f-delta} 
\end{align}
\end{lemma}
\begin{proof}
Since \eqref{eq:SILP} is feasible, conditions ii) and iii) in Corollary~\ref{cor:primalFeasible} imply that $\sup_{h \in I_3} \tilde{b}(h) < \infty$ and $\lim_{\delta \to \infty} \omega(\delta) < \infty$. Let $\delta_2$ and $\delta_3$ be as defined in i)-ii) of Corollary~\ref{cor:primalFeasible}. 

First, we suppose $\bar z$ satisfies \eqref{eq:sup-I3}-\eqref{eq:lim-f-delta} and show $\bar z \in \cl(P(\Lambda; x_1, \dots, x_n))$. 
%
%
Consider the following two exhaustive cases.
\vskip 5pt
\noindent\underline{\em Case 1: $\bar z   > \lim_{\delta \to \infty} \omega(\delta)$.} There exists a $\hat \delta \in \R$ such that $\bar z > \omega(\hat\delta)$. Choose $\bar\delta \geq \max\{0, \hat \delta, \delta_2\}.$   By~\eqref{eq:sup-I3}, $\bar z \ge \sup_{h \in I_3} \tilde{b}(h) = \delta_{3}$. Also, $\bar z >  \omega(\hat\delta) \geq \omega(\bar\delta)$ since $\omega(\delta)$ is nonincreasing. Thus, $(x(\bar\delta; \ell), \bar z)$ satisfies the hypotheses of condition iv) of Corollary~\ref{cor:primalFeasible}.    Therefore  $(x(\bar \delta;\ell), \bar z) \in P(\Lambda; x_1, \ldots, x_{\ell-1})$ and this implies $\bar z \in P(\Lambda; x_1, \ldots, x_n)$.
\vskip 5pt
\noindent\underline{\em Case 2: $\bar z = \lim_{\delta \to \infty} \omega(\delta) $.} Since $\omega(\delta)$  nonincreasing in $\delta$, there exists a sequence of real numbers $(\bar \delta_m)_{m\in \N}$ such that for every $m \in \N$, $\bar \delta_m \geq \max\{0,\delta_2\}$ and $z_m := \omega(\bar \delta_m) \to \bar z$. Since $\omega(\delta)$ is nonincreasing and $\overline{z}$ satisfies~\eqref{eq:sup-I3},    $z_m  = \omega(\bar \delta_m) \geq  \lim_{\delta \to \infty} \omega(\delta)  = \bar z\geq \sup_{h \in I_3} \tilde{b}(h)$. Hence $z_m \ge \max \{\delta_3, \omega(\bar \delta_m)\}$ and  by Corollary~\ref{cor:primalFeasible}(iv), $(x(\bar \delta_m;\ell),z_m) \in P(\Lambda;x_1,\dots,x_{\ell-1})$. Therefore  $z_m \in P(\Lambda; x_1, \ldots, x_n)$ and $z_m \to \bar z$. This implies  $\bar z \in \cl(P(\Lambda; x_1, \ldots, x_n))$.
\vskip 5pt

Conversely, we let $\bar z \in \cl (P(\Lambda; x_1, \dots, x_n))$ and show $\bar z$ satisfies \eqref{eq:sup-I3} and \eqref{eq:lim-f-delta}. Since $z \in \cl (P(\Lambda; x_1, \dots, x_n))$ there exists a sequence $z_m \in P(\Lambda; x_1, \dots, x_n)$ where $z_m \rightarrow \bar z$. Since $z_m \in P(\Lambda; x_1, \dots, x_n)$ there exists an $x^m = (x^m_\ell, \dots, x^m_n)$ such that $(x^m, z_m)$ satisfies the constraints of system \eqref{eq:J_system}. This  implies $z_m \ge \sup_{h \in I_3} \tilde{b}(h)$. Since $z_m \rightarrow \bar z$, conclude $\bar z \ge \sup_{h \in I_3} \tilde{b}(h)$.

Also, since $(x^m, z_m)$ satisfies \eqref{eq:J_system}, 
$z_m \ge \sup_{h \in I_4}\{\tilde b(h) - \sum_{k = \ell}^n a^k(h)x^m_k\}$. Letting $\bar \delta_m  = \max_{k = \ell, \dots, n} |x^m_k|$  gives   $z_m \ge \sup_{h \in I_4}\{\tilde b(h) - \sum_{k = \ell}^n a^k(h)x^m_k\}  \ge \sup_{h \in I_4}\{\tilde b(h) - \bar\delta_m \sum_{k = \ell}^n |a^k(h)|\} = \omega(\bar \delta_m)$.   
Thus, $z_m \ge \omega(\bar \delta_m) \ge \lim_{\delta \rightarrow \infty} \omega(\delta)$ for all $m$, where the last inequality holds since $\omega(\delta)$ is nonincreasing. Since $z_m \rightarrow \bar z$, conclude $\bar z \ge \lim_{\delta \rightarrow \infty} \omega(\delta)$. Hence $\bar z$ is a feasible solution to system \eqref{eq:sup-I3}-\eqref{eq:lim-f-delta}.\end{proof}

By Lemma~\ref{lemma:z-projection},   if~\eqref{eq:SILP} is feasible, then  its optimal value  is found by solving the optimization problem
\begin{equation}\label{eq:newer-problem}
\begin{array}{rl} 
\inf_{z} & z \\
\textrm{s.t.} & \eqref{eq:sup-I3}-\eqref{eq:lim-f-delta}.
\end{array}
\end{equation}
This follows because the optimal value of a continuous objective function over a convex feasible region is the same the optimal value of that objective when optimized over the closure of the region. The next two results follow directly from this observation.
\begin{lemma}\label{lemma:primal-optimal-value}
If  \eqref{eq:SILP} is  feasible then 
$v(\ref{eq:SILP}) = \max\big\{  \sup_{h \in I_3} \tilde{b}(h) ,   \lim_{\delta \rightarrow \infty} \omega(\delta) \big\}.
$
\end{lemma}

\begin{theorem}[Primal boundedness] \label{theorem:primal-not-unbounded} A feasible \eqref{eq:SILP} is bounded if and only if $I_3 \neq \emptyset$ or $\lim_{\delta \rightarrow \infty} \omega(\delta) > -\infty$. 
\end{theorem}
\begin{proof}
By contrapositive in both directions. By Lemma~\ref{lemma:primal-optimal-value}, $v(\ref{eq:SILP}) = -\infty$ if and only if $\max\{  \sup_{h \in I_3} \tilde{b}(h), \lim_{\delta \rightarrow \infty} \omega(\delta) \} = -\infty$ if and only if $\sup_{h \in I_3} \tilde{b}(h) = -\infty$ and $\lim_{\delta \rightarrow \infty} \omega(\delta) = -\infty$. Note that $\sup_{h \in I_3} \tilde{b}(h) = -\infty$ if and only if $I_3 = \emptyset$. 
\end{proof}

\subsubsection{Primal solvability}\label{s:primal-solvability}

An instance of \eqref{eq:SILP} is solvable if the infimum value of  its objective is attained. Note that an optimal solution $v(\ref{eq:SILP})$ may exist to \eqref{eq:newer-problem} even though an optimal solution to \eqref{eq:SILP} does not exist (see for instance Example~\ref{example:not-primal-optimal} below). This is due to the fact that \eqref{eq:newer-problem} is an optimization problem over the \emph{closure} of the projection $P(\Lambda; x_1, \dots, x_n)$, and hence an optimal solution to \eqref{eq:new-problem} may exist in the closure but not the projection itself. Thus, the solution may not ``lift" to an optimal solution of \eqref{eq:SILP}.   A sufficient condition for when this ``lifting" can occur is given in Theorem~\ref{theorem:primal-solvability}. 

\begin{theorem}[Primal solvability]\label{theorem:primal-solvability}
If \eqref{eq:SILP} is feasible and 
$\sup_{h \in I_3} \tilde{b}(h) >  \lim_{\delta\to\infty} \omega(\delta)$, then \eqref{eq:SILP} has an optimal solution with value $v(\ref{eq:SILP}) =\sup_{h \in I_3} \tilde{b}(h)$.
\end{theorem}
\begin{proof}
Let $z^* = v(\ref{eq:SILP})$. Since \eqref{eq:SILP} is feasible, by part (ii) of Corollary~\ref{cor:primalFeasible} it follows that $\infty > \sup_{h \in I_3} \tilde{b}(h) >  \lim_{\delta \rightarrow \infty} \omega(\delta)$. Moreover, by Lemma~\ref{lemma:primal-optimal-value}, $z^* = \sup_{h \in I_3}\tilde{b}(h)$. Let $\delta_2$ be as defined in Corollary~\ref{cor:primalFeasible}. Since $\omega(\delta)$ is a nonincreasing function, there exists a $\delta^* \geq \max\{0, \delta_2\}$ such that $\omega(\delta^*) < \sup_{h \in I_3} \tilde{b}(h) = z^*$. Then, $(x(\delta^*; \ell), z^*)$ satisfies the hypotheses of condition iv) in Corollary~\ref{cor:primalFeasible} and so $(x(\delta^*; \ell), z^* ) \in P(\Lambda; x_1, \ldots, x_\ell)$, showing that there exists a feasible point $(x_1, \dots, x_n, z)$ in $\Lambda$ where $z = z^*$. Thus there is a feasible point for \eqref{eq:SILP} with value $z^* = v(\ref{eq:SILP})$.\end{proof}

In light of the previous result, one may ask whether primal solvability holds when $\lim_{\delta \rightarrow \infty} \omega(\delta) = \sup_{h \in I_3} \tilde{b}(h)$.
The following two examples demonstrate that such problems can be either solvable or not solvable.

\begin{example}\label{example:not-primal-optimal}
Consider the following instance of \eqref{eq:SILP}
\begin{align}\label{eq:not-primal-optimal}
\begin{array}{rcl}
\inf x_{1} \phantom{ + x_{2} + \ \ } && \\
x_{1} + \tfrac{1}{t^{2}}x_{2} \,&\ge& \tfrac{1}{t^{2}} + \tfrac{1}{t} \quad \text{ for } t \ge 1 \\
          x_1 \phantom{ + \tfrac{1}{t^{2}}x_{2} \ \,} & \ge & 0.
\end{array}
\end{align}
Applying Fourier-Motzkin elimination  to
\begin{align}\label{eq:full-system-example}
\begin{array}{rcl}
-x_1 \phantom{ + \tfrac{1}{t^{2}} x_{2}} + z &\ge& 0 \\
\phantom{-}x_1 + \tfrac{1}{t^{2}} x_{2}  \phantom{ + z} &\ge& \tfrac{1}{t^{2}} + \tfrac{1}{t}\quad \text{ for } t \ge 1 \\
x_1 \phantom{ + \tfrac{1}{t^{2}} x_{2} + z} & \ge & 0
\end{array}
\end{align}
yields (by eliminating $x_1$)
\begin{equation}\label{eq:project-out-x1}
\begin{array}{rcl}
\tfrac{1}{t^{2}} x_{2} + z  &\ge& \tfrac{1}{t^{2}} + \tfrac{1}{t}\quad \text{ for } t \ge 1 \\ 
\phantom{\frac{1}{t^{2}} x_{2} +}  z & \ge & 0.
\end{array}
\end{equation}
The only $I_3$ constraint is $z \ge 0$  so $\sup_{h \in I_3} \tilde{b}(h) = 0$. Note that for $\delta \ge 3/2,$
\begin{eqnarray*}
\omega(\delta) = \sup_{t \ge 1} \left\{ \tfrac{1}{t^{2}} + \tfrac{1}{t}  -   \tfrac{\delta}{t^{2}} \right\} =   \sup_{t \ge 1} \left\{ \tfrac{(1 - \delta)}{t^{2}} + \tfrac{1}{t}  \right\}   = \tfrac{1}{4(\delta - 1)}. 
\end{eqnarray*}
When $\delta \ge 1$ and $t \neq 0$, the  function $\tfrac{(1 - \delta)}{t^{2}} + \tfrac{1}{t}$  is  concave and quadratic in $\frac{1}{t}.$ The supremum is attained by $t^{*} = -2(1 - \delta)$.   When  $\delta \ge 3/2$,   $t^{*} \ge 1$ and substituting  the optimal value of $t^{*}$ into $\tfrac{(1 - \delta)}{t^{2}} + \tfrac{1}{t}$ gives $\frac{1}{4(\delta -1)}.$
Clearly, $\lim_{\delta \rightarrow \infty} \omega(\delta) =    0 = \sup_{h \in I_3} \tilde{b}(h)$ and so by Lemma~\ref{lemma:primal-optimal-value} the optimal value is $0$. 

However, for $z = 0$ the system \eqref{eq:project-out-x1} has no possible feasible assignment for $x_2$. Indeed, for any proposed $\bar x_2$ take $t \ge \bar x_2$. This implies $\tfrac{1}{t^{2}} \bar x_{2} + 0  \le \tfrac{1}{t} < \tfrac{1}{t^2} + \frac{1}{t}$, which means $(\bar x_2,0)$ is infeasible to \eqref{eq:project-out-x1} and  the primal is not solvable.  \hfill $\triangleleft$
\end{example}

\begin{example}\label{example:primal-solvable}
Consider the following instance of \eqref{eq:SILP}
\begin{eqnarray*}
\inf\  x_1  \phantom{+ \tfrac{1}{i}x_2}  && \\
\phantom{\inf\ } x_1 \phantom{+ \tfrac{1}{i}x_2} &\ge& 0 \\
\phantom{\inf\  \tfrac{1}{i}x_1} - x_2&\ge& -1 \\
\phantom{\inf\ } x_1 - \tfrac{1}{i}x_2 &\ge& 0 \quad \text{ for $i = 3, 4, \dots$}
\end{eqnarray*}
Applying  Fourier-Motzkin elimination (after introducing the $z - x_1$ constraint) to
yields (after projecting out $x_1$)
\begin{eqnarray*}
- x_2 \phantom{ + z}  &\ge& -1 \\
 \phantom{+ 2x_2 + } z &\ge& 0 \\
 - \tfrac{1}{i}x_2 + z &\ge& 0 \quad \text{ for $i = 3, 4, \dots$}
\end{eqnarray*}
Observe $I_3 = \left\{1\right\}$ and  $\sup_{h \in I_3} \tilde{b}(h) = 0$. 
Note  
$
\omega(\delta)  = \sup \big\{\tilde b(h) - \delta \sum_{h \in I_4} |\tilde a^k(h)| \, \, : \, \, h \in I_4 \big\} \\
           = \sup \left\{0 - \delta / h \, \, : \, \, h = 3, 4, \dots \right\} = 0.
$ 
Thus, $\lim_{\delta \rightarrow \infty} \omega(\delta) = 0 = \sup_{h \in I_3} \tilde{b}(h)$. By Lemma~\ref{lemma:primal-optimal-value}, this implies $v(\ref{eq:SILP}) = 0$ and this value is obtained for the feasible solution $x_1 = x_2 = 0$ and  the primal is solvable. \hfill $\triangleleft$
\end{example}

\subsection{Dual results}\label{s:dual-results}

The next step is to develop a duality theory for \eqref{eq:SILP} using Fourier-Motzkin elimination. 
The standard dual problem in the semi-infinite linear programming literature (see for instance \cite{charnes1963duality})  is the  finite support (Haar) dual
introduced in Section~\ref{s:introduction} and reproduced here for convenience.
\begin{align*}
\begin{array}{rrll}
\sup &\sum_{i \in I} b(i) v(i) & & \\
   {\rm s.t.} & \sum_{i \in I} a^{k}(i) v(i) &= c_k & \text{ for } k = 1, \ldots, n \\
& v &\in \R_+^{(I)} &
\end{array}\tag{\text{FDSILP}}
\end{align*}

In this section, we characterize when \eqref{eq:FDSILP} is  feasible, bounded, and solvable. Later in Section~\ref{section:duality-gap} we characterize when there is zero duality gap between \eqref{eq:SILP} and \eqref{eq:FDSILP}; that is, $v(\ref{eq:SILP}) = v(\ref{eq:FDSILP})$. 

In the remainder of this section, assume Fourier-Motzkin elimination has been applied to \eqref{eq:initial-system-obj-con}-\eqref{eq:initial-system-con} yielding \eqref{eq:J_system}. Our attention turns to the multipliers  generated in Step 2.b.(iii) of the Fourier-Motzkin elimination procedure. These multipliers generate  solutions to  \eqref{eq:FDSILP}. 

First a small, but important, distinction. 
The multipliers $u^h$ generating  \eqref{eq:J_system} are real-valued functions defined on the set $\left\{0\right\} \cup I$ where the inequality \eqref{eq:initial-system-obj-con} has  index $0$. However, solutions to \eqref{eq:FDSILP} are real-valued functions defined only on $I$. Thus, it is useful to work with the restriction $v^h: I \to \R$ of $u^h$ to $I$. That is, $v^h(i) = u^h(i)$ for $i \in I$.  Conversely, given a function $v : I \to \R$ and a real number $v_0$, let $u = (v_0, v)$ denote the \emph{extension} of $v$ onto the index set $\left\{0\right\} \cup I$ where $u(0) = v_0$ and $u(i) = v(i)$ for all $i \in I$.  Lemma \ref{lemma:feasible-FDSILP}  gives basic properties of $v^h$ that are used later.

\begin{lemma}\label{lemma:feasible-FDSILP}
If  Fourier-Motzkin elimination is applied to \eqref{eq:initial-system-obj-con}-\eqref{eq:initial-system-con}  yielding \eqref{eq:J_system}, then 
\begin{enumerate}[(i)]
\item for every $h \in I_1 \cup I_2 \cup I_3 \cup I_4$, $\tilde b(h) = \langle b, v^h \rangle$. \label{item:b-property}
\item for $h \in I_1$, $u^h(0) = 0$ and $v^h$ is a recession direction for the feasible region of \eqref{eq:FDSILP}.  \label{item:I1-property}
\item for $h \in I_2$, $u^h(0) = 0$ and $v^h$ satisfies $\sum_{i \in I} a^{k}(i) v^h(i) = 0$ for $k = 1, \ldots, \ell-1$, and $\sum_{i \in I} a^{k}(i) v^h(i) = \tilde{a}^k(h)$ for $k=\ell, \ldots, n$. \label{item:I2-property}
\item for $h \in I_3$, $u^h(0) = 1$ and $v^h$ is a feasible solution to \eqref{eq:FDSILP}, and \label{item:I3-property}
\item for $h \in I_4$, $u^h(0) = 1$ and $v^h$ satisfies $\sum_{i \in I} a^{k}(i) v^h(i) - c_k = 0$ for $k = 1, \ldots, \ell-1$, and $\sum_{i \in I} a^{k}(i) v^h(i) - c_k = \tilde{a}^k(h)$ for $k=\ell, \ldots, n$. \label{item:I4-property}
\end{enumerate}
\end{lemma} 
\begin{proof}
We establish part 
\eqref{item:I3-property} only. 
The constraints indexed by $I_3$ must involve $z$ and so the multipliers $u^h$ for $h \in I_3$ must have $u^h(0) > 0$. Assume $u^h(0) = 1$, which is without loss by Remark~\ref{rem:initial_remark}. By Theorem~\ref{theorem:FM-elim-succ}(ii), for $h \in I_3$, $0 = \langle (-c_k, a^k), u^h \rangle = -c_k + \langle a^k, v^h \rangle$ for all $k = 1, \dots, n$. This implies $v^h$ satisfies the equality constraints of \eqref{eq:FDSILP}.  In addition, $u^h \ge 0$ implies $v^h \ge 0$ and $v^h$ is a feasible solution to \eqref{eq:FDSILP}. \end{proof}

\subsubsection{Dual feasibility}\label{s:dual-feasibility}

The next two subsections relate dual feasibility and boundedness to properties of the projected system \eqref{eq:J_system}. Theorem~\ref{theorem:all-rays} and Lemma~\ref{lemma:feasible-FDSILP} play pivotal roles in the proofs. 

\begin{theorem}[Dual Feasibility]\label{theorem:dual-infeasibility} 
\eqref{eq:FDSILP}  is feasible if and only if $I_3 \neq \emptyset.$
\end{theorem}

\begin{proof} 
($\Longrightarrow$) 
If \eqref{eq:FDSILP} is feasible, there is a $\overline{v} \ge 0$ with finite support such that $\sum_{i \in I} a_{k}(i) \overline{v}_{i} = c_k,  k = 1, \ldots, n$   
 and this implies  $\langle (-c_{k}, a^{k}), (1, \overline{v}) \rangle = 0, k = 1, \ldots, n$. Then, by applying
Theorem~\ref{theorem:all-rays} to \eqref{eq:initial-system-obj-con}-\eqref{eq:initial-system-con} with $M=n$, there exist a finite index set $\bar{I}  \subseteq (I_{1} \cup I_{3})$ and multipliers $u^h : \left\{0\right\} \cup I \to \R$ for $h \in \bar I$ such that 
\[
\begin{array}{rcl}
(1, \overline{v})  & = &  \sum_{h \in \bar{I}} \lambda_{h} u^h \\
& = & \sum_{h \in \bar{I} \cap I_{1} } \lambda_{h} u^h + \sum_{h \in \bar{I} \cap I_{3} } \lambda_{h} u^h \\
& = & \sum_{h \in \bar{I} \cap I_{1}} \lambda_{h} (0, v^{h}) + \sum_{h \in \bar{I} \cap I_{3}} \lambda_{h} (1, v^{h})
\end{array}
\]
where $\lambda_{h} \ge 0$ for all $h \in \bar{I}$ and $v^h$ is the restriction of $u^h$ onto $I$. The third equality follows from Lemma~\ref{lemma:feasible-FDSILP}\eqref{item:I1-property}~and~\eqref{item:I3-property}. 
Now, the $1$ in the first component of $(1, \overline{v})$ implies that $\bar{I} \cap I_{3}$ cannot be empty, and hence $I_{3}$ cannot be empty. 

\vskip 3pt

\noindent($\Longleftarrow$) 
Take any $u^h$ with $h \in I_3$. By Lemma~\ref{lemma:feasible-FDSILP}\eqref{item:I3-property}, $v^h$ is a feasible solution to \eqref{eq:FDSILP}.    
\end{proof}

\subsubsection{Dual boundedness}\label{s:dual-boundedness}

To characterize dual boundedess,  first establish weak duality.

\begin{lemma}[Weak Duality]\label{lemma:weak-duality} Suppose $\tilde{b}(h) \leq 0$ for all $h \in I_1$.
If $\overline{v}$ is a feasible dual solution to problem \eqref{eq:FDSILP} then   
\begin{itemize}

\item[(i)]  there exists an $\bar h \in I_3$ such that $\tilde b(\bar h) \ge    \langle b, \overline{v}  \rangle$,

\item[(ii)]     $ \langle b, \overline{v} \rangle$  is a lower bound on the optimal solution value of \eqref{eq:SILP}. 
\end{itemize}
\end{lemma}

\begin{proof}
Applying Theorem~\ref{theorem:all-rays} as in the proof of Theorem~\ref{theorem:dual-infeasibility} implies there exists an index set $\bar{I} \subseteq I_{1} \cup I_{3}$ such that $(1, \overline{v}) = \sum_{h \in \bar{I} \cap I_{1}} \lambda_{h} (0, v^{h}) + \sum_{h \in \bar{I} \cap I_{3}} \lambda_{h} (1, v^{h}).$ Reasoning about the components of $(1, \bar v)$ separately gives, 
\begin{align}\label{eq:prop-weak-duality-con1}
\bar v = \sum_{h \in \bar{I} \cap I_{1}} \lambda_{h} v^{h} + \sum_{h \in \bar{I} \cap I_{3}} \lambda_{h} v^{h}
\end{align}
and $1 = \sum_{h \in \bar{I} \cap I_{3}} \lambda_h$. Lemma~\ref{lemma:feasible-FDSILP}(i) and the hypothesis $\tilde b(h) \le 0$ for all $h \in I_{1}$ imply  $\langle b, v^{h}  \rangle \le 0$ for all $h \in I_{1}$. Thus, 
(\ref{eq:prop-weak-duality-con1}) gives
$\langle b, \overline{v} \rangle \le \sum_{h \in \bar{I} \cap I_{3}} \lambda_{h} \langle b, v^{h} \rangle \le \langle b, v^{\bar h} \rangle = \tilde b(\bar h)$ for some $\bar h \in \bar I\cap I_3$, where the second inequality follows because the $\lambda_{h}$ for $h \in I_3$ are nonnegative and sum to $1$. This implies i).  Now ii) follows immediately from Lemma~\ref{lemma:primal-optimal-value}.
\end{proof}

\begin{theorem}[Dual boundedness]\label{theorem:dual-boundedness}
Suppose \eqref{eq:FDSILP} is feasible. Then \eqref{eq:FDSILP} is bounded if and only if
\begin{multicols}{2} 
\begin{enumerate}
\item[(i)] $\tilde{b}(h) \le 0$ for all $h \in I_1$ and 
\item[(ii)] $\sup_{h \in I_3} \tilde{b}(h) < \infty$.
\end{enumerate}
\end{multicols}
\end{theorem}
\begin{proof}
($\Longleftarrow$) By contrapositive. We suppose \eqref{eq:FDSILP} is unbounded and show  that if condition (i) holds, then (ii) does not hold. Assume $\tilde{b}(h) \leq 0$ for all $h \in I_1$. Since \eqref{eq:FDSILP} is unbounded, for every $M \in \N$ there exists a feasible $\bar v_M$ with $\langle b, \bar v_M \rangle \ge M$. By Lemma~\ref{lemma:weak-duality}, there exist some $h_M \in I_3$ such that $\tilde b(h_M) \geq \langle b, \bar v_M \rangle \geq M$. Thus, $\sup_{h \in I_3} \tilde{b}(h) \ge \tilde{b}(h_M) \geq M$ for all $M \in \N$ and this implies $\sup_{h \in I_3} \tilde{b}(h) = \infty$.  Therefore, (ii) does not hold.

($\Longrightarrow$) By contrapositive. Assume condition i) does not hold. Thus, there exists an $h^* \in I_1$ such that $\tilde{b}(h^*) > 0$ and by Lemma~\ref{lemma:feasible-FDSILP}(ii), $\langle a^k, v^{h^*}\rangle = 0$ for all $k = 1, \ldots, n$. Now, consider any $\bar v$ feasible to \eqref{eq:FDSILP}, which exists since \eqref{eq:FDSILP} is feasible. Then, $\bar v + \lambda v^{h^*}$ is also feasible for all $\lambda \geq 0$. Now, the objective value for these feasible solutions equal $\langle b, \bar v + \lambda v^{h^*} \rangle = \langle b, \bar v \rangle + \lambda \langle b, v^{h^*} \rangle$. Since $\langle b, v^{h^*} \rangle = \tilde b(h^*)  > 0$, letting $\lambda \to \infty$, yields unbounded values for the objective value of \eqref{eq:FDSILP}.

Next assume condition ii) does not hold. This implies there is a sequence of $\{h_m\}_{m \in \N}$ in $I_3$ 
such that, by Lemma~\ref{lemma:feasible-FDSILP}(i), $\langle b, v^{h_m}\rangle = \tilde{b}(h_m) \to \infty$. By Lemma~\ref{lemma:feasible-FDSILP}(iii), each $v^{h_m}$ is a feasible solution to \eqref{eq:FDSILP} and thus \eqref{eq:FDSILP} is unbounded.
\end{proof}

\begin{remark}\label{rem:boundedness}
Observe that there are two distinct ways for a feasible \eqref{eq:FDSILP} to be unbounded. The first is when there is a recession direction to the feasible region that drives the objective value to $+\infty$. From Lemma~\ref{lemma:feasible-FDSILP}\eqref{item:I1-property} every $h \in I_1$ yields a recession direction $v^h$. In addition, if $\tilde b(h) > 0$ then $ \langle b, v^h \rangle > 0$ and so moving within the feasible region along recession direction $v^h$ drives the objective to $+\infty$. This argument was given in full detail in the proof of Theorem~\ref{theorem:dual-boundedness}.

Contrary to our intuition from finite dimensions, the second way~\eqref{eq:FDSILP} may have an unbounded objective value can occur when the feasible region itself is bounded. This happens when there are no recession directions and $\sup_{h \in I_3} \tilde{b}(h) = \infty$. This occurs when \eqref{eq:FDSILP} has a sequence of feasible solutions whose values converge to $+\infty$. Consider the semi-infinite linear program:
\[
\begin{array}{rcl}
\inf x_1 && \\
\text{ s.t. } x_1 &\ge& i \quad \text{ for $i\in \N$}
\end{array}
\]
with finite support dual
\[
\begin{array}{rcl}
\sup \sum_{i \in \N} iv(i) && \\
\text{ s.t. } \sum_{i \in \N} v(i) &=& 1 \\
\phantom{\text{ s.t. } \sum_{i \in \N}} v(i) &\ge& 0 \quad \text{ for $i \in \N$}
\end{array}
\]
The feasible region of the finite support dual is bounded (note that $0 \le v(i) \le 1$ for all $i$) and there is no recession direction. However, the problem is still unbounded. Consider the sequence of feasible extreme point solutions $e^m$. Clearly, $ \sum_{i \in \N} ie^m(i) = m \rightarrow \infty$ as $m \to \infty$. Thus \eqref{eq:FDSILP} is unbounded. 
 
Fourier-Motzkin elimination can identify which of the conditions of  Theorem~\ref{theorem:dual-boundedness} are violated and  result in an unbounded problem. Applying Fourier-Motzkin elimination (after eliminating $x_1$) the system:
$z \ge i$ for $i = 1, 2, \dots$.
Thus, $I_1 = \emptyset$ so there are no recession directions, but $I_3 = \left\{1,2, \dots \right\}$ and $\sup_{h \in I_3} \tilde{b}(h) = \infty$. \hfill $\triangleleft$
\end{remark}

\subsubsection{Dual solvability}\label{s:dual-solvability}

To characterize dual solvability, begin with a characterization of  the optimal dual value.

\begin{theorem}\label{thm:dual-optimal-value}
If $\tilde{b}(h) \leq 0$ for all $h \in I_1$ then  $v\eqref{eq:FDSILP} = \sup_{h \in I_3} \tilde{b}(h)$.
\end{theorem}
\begin{proof}
By Lemma~\ref{lemma:weak-duality}(ii), for every dual feasible solution $\bar v$ there exists an $h \in I_3$ with $ \tilde b(h) \ge \langle b, \bar v \rangle$. Hence, $\sup_{h \in I_3} \tilde b(h) \ge \langle b, \bar v \rangle$ for all feasible $\bar v$. This implies $\sup_{h \in I_3} \tilde{b}(h) \ge v(\ref{eq:FDSILP})$. Conversely, by Lemma~\ref{lemma:feasible-FDSILP}(iii), every $h \in I_3$ yields a $v^h$ with $v^h$ feasible to \eqref{eq:FDSILP} and  $\tilde b(h) = \langle b, v^h \rangle$. Hence $ \tilde b(h) = \langle b, v^h \rangle \le v(\ref{eq:FDSILP})$ for all $h \in I_3$. Thus, $\sup_{h \in I_3} \tilde{b}(h) \le v(\ref{eq:FDSILP})$ and the result  follows.
\end{proof}

\begin{corollary}\label{cor:dual-optimal-value}
If either  \eqref{eq:SILP} is feasible or \eqref{eq:FDSILP} is feasible and bounded, then $v\eqref{eq:FDSILP} = \sup_{h \in I_3} \tilde{b}(h)$.
\end{corollary}
\begin{proof}
If \eqref{eq:SILP} is feasible, then by Theorem~\ref{theorem:primalFeasible}(ii) $\tilde{b}(h) \leq 0$ for all $h \in I_1$. The result follows from Theorem~\ref{thm:dual-optimal-value}. If \eqref{eq:FDSILP} is feasible and bounded then by Theorem~\ref{theorem:dual-boundedness}(i) $\tilde{b}(h) \leq 0$ for all $h \in I_1$. Once again, the result follows from Theorem~\ref{thm:dual-optimal-value}.
\end{proof}

\begin{theorem}[Dual solvability]\label{theorem:dual-solvability}
\eqref{eq:FDSILP} has an optimal solution if and only if 
\begin{multicols}{2}
\begin{enumerate}[(i)]
\item $\tilde{b}(h) \leq 0$ for all $h \in I_1$, and
\item $\sup_{h \in I_3} \tilde{b}(h)$ is attained.
\end{enumerate}
\end{multicols}

\end{theorem}
\begin{proof}
($\Longrightarrow$)
Let $v^*$ be an optimal solution to \eqref{eq:FDSILP} with optimal  value $v(\ref{eq:FDSILP}) = \langle b, v^* \rangle$. This implies \eqref{eq:FDSILP} is both feasible and bounded. By Theorem~\ref{theorem:dual-boundedness}(i), $\tilde{b}(h) \leq 0$ for all $h \in I_1$, establishing condition (i). Apply Lemma~\ref{lemma:weak-duality}(i) and conclude  there  exists a $v^{h^*}$ for some $h^* \in I_3$ with $\langle b, v^{h^*} \rangle \ge \langle b, v^* \rangle = v(\ref{eq:FDSILP})$. By Lemma~\ref{lemma:feasible-FDSILP}\eqref{item:I3-property}, $v^{h^*}$ is feasible to \eqref{eq:FDSILP} and  $\langle b, v^{h^*} \rangle \le v(\ref{eq:FDSILP})$. Hence $\tilde b(h^*) = \langle b, v^{h^*} \rangle = v(\ref{eq:FDSILP}) = \sup_{h \in I_3} \tilde{b}(h)$, where the first equality holds from Lemma~\ref{lemma:feasible-FDSILP}\eqref{item:b-property}, the second equality holds from the arguments in the previous two sentences, and the third equality holds from Corollary~\ref{cor:dual-optimal-value}. Thus, $ \tilde{b}(h^*) = \sup_{h \in I_3} \tilde{b}(h)$, establishing condition (ii). 

\noindent($\Longleftarrow$)
By hypothesis there is an  $h^* \in I_3$  such that 
$\sup_{h \in I_3} \tilde{b}(h) = \tilde{b}(h^*) < \infty$. The fact that $I_3$ is nonempty implies \eqref{eq:FDSILP} is feasible by Theorem~\ref{theorem:dual-infeasibility}. Thus, by Theorem~\ref{theorem:dual-boundedness} \eqref{eq:FDSILP} is bounded.   
Since \eqref{eq:FDSILP} is feasible and bounded, by Corollary~\ref{cor:dual-optimal-value}  $\sup_{h \in I_3} \tilde{b}(h) = v(\ref{eq:FDSILP})$. Moreover, Lemma~\ref{lemma:feasible-FDSILP}\eqref{item:b-property} and~\eqref{item:I3-property} imply that $\tilde{b}(h^*)  = \langle b, v^{h^*} \rangle $ and $v^{h^*}$ is a feasible solution to \eqref{eq:FDSILP}. Putting this together, $v(\ref{eq:FDSILP}) = \sup_{h \in I_3} \tilde{b}(h) = \tilde{b}(h^*) = \langle b, v^{h^*} \rangle$ and $v^{h^*}$ is an optimal solution to \eqref{eq:FDSILP}. \old{By Lemma~\ref{lemma:weak-duality}(i) for every feasible solution $\bar v$ to \eqref{eq:FDSILP} there exists an $h \in I_3$ with $ \langle b, v^h \rangle \ge \langle b, \bar v \rangle$. 
Thus, $\langle b, \bar v \rangle \le \langle b, v^h = \tilde b(h) \rangle \le \sup_{h \in I_3} \tilde{b}(h) = \langle b, v^{h^*} \rangle$.}  
\end{proof}

\subsubsection{Zero duality gap and strong duality}\label{section:duality-gap}

The primal-dual pair \eqref{eq:SILP} and \eqref{eq:FDSILP} has  a \emph{zero duality gap} if \eqref{eq:SILP} is feasible and $v(\ref{eq:SILP}) = v(\ref{eq:FDSILP})$.  

\begin{theorem}[Zero Duality Gap]\label{theorem:zero-duality-gap}
There is a zero duality gap for the primal-dual pair \eqref{eq:SILP} and \eqref{eq:FDSILP} if and only if
\begin{multicols}{2}
\begin{itemize}
\item[(i)]  \eqref{eq:SILP} is feasible, and
\item[(ii)] $\sup_{h \in I_3} \tilde{b}(h) \geq \lim_{\delta \to \infty}\omega(\delta)$.
\end{itemize}
\end{multicols}
\end{theorem}

\begin{proof}
($\Longrightarrow$)  Assume zero duality gap. Condition (i) holds by definition of zero duality gap. \old{ If \eqref{eq:SILP} is unbounded, then by  Lemma~\ref{lemma:primal-optimal-value},  $\sup_{h \in I_3} \tilde{b}(h) = \lim_{\delta \to \infty}\omega(\delta) = -\infty$ and ii) holds.   Now assume \eqref{eq:SILP} is bounded and thus, v(\ref{eq:SILP}) is a finite real number. Moreover, the zero duality gap implies $v(\ref{eq:FDSILP}) = v(\ref{eq:SILP})$, which means that \eqref{eq:FDSILP} is feasible and bounded. Thus by Corollary~\ref{cor:dual-optimal-value}, $\sup_{h \in I_3} \tilde{b}(h) = v(\ref{eq:FDSILP}) = v(\ref{eq:SILP}) = \max\{\sup_{h \in I_3} \tilde{b}(h), \lim_{\delta \to \infty}\omega(\delta)\} \geq \lim_{\delta \to \infty}\omega(\delta)$, where the third equality holds by Lemma~\ref{lemma:primal-optimal-value}. Thus condition ii) holds.}Since (\ref{eq:SILP}) is feasible, by Corollary~\ref{cor:dual-optimal-value}, $$\sup_{h \in I_3} \tilde{b}(h) = v(\ref{eq:FDSILP}) = v(\ref{eq:SILP}) = \max\{\sup_{h \in I_3} \tilde{b}(h), \lim_{\delta \to \infty}\omega(\delta)\} \geq \lim_{\delta \to \infty}\omega(\delta),$$ where the third equality holds by Lemma~\ref{lemma:primal-optimal-value}. Thus condition (ii) holds.

($\Longleftarrow$) Now assume conditions (i) and (ii) hold.  By (i)~\eqref{eq:SILP} is feasible. \old{ Then \eqref{eq:SILP} is either unbounded or bounded.  If \eqref{eq:SILP} is unbounded, then  $v(\ref{eq:SILP}) = v(\ref{eq:FDSILP}) = -\infty$ and there is a zero duality gap.  Assume  \eqref{eq:SILP} is bounded. }By Lemma~\ref{lemma:primal-optimal-value}, $v(\ref{eq:SILP})= \max\{\sup_{h \in I_3} \tilde{b}(h), \lim_{\delta \to \infty}\omega(\delta)\} = \sup_{h \in I_3} \tilde{b}(h)$, where the second equality follows from condition (ii). \old{Thus, $-\infty < \sup_{h \in I_3} \tilde{b}(h) < \infty $. The first inequality implies that $I_3 \neq \emptyset$, thus by Theorem~\ref{theorem:dual-infeasibility} \eqref{eq:FDSILP} is feasible. Since \eqref{eq:SILP} is feasible, we have $\tilde{b}(h) \leq 0$ for all $h \in I_1$ by Theorem~\ref{theorem:primalFeasible}. Combined with the fact that $\sup_{h \in I_3} \tilde{b}(h) < \infty$, Theorem~\ref{theorem:dual-boundedness} implies that \eqref{eq:FDSILP} is bounded and}Also, Corollary~\ref{cor:dual-optimal-value} implies $v(\ref{eq:FDSILP}) = \sup_{h \in I_3} \tilde{b}(h)$. Thus, $v(\ref{eq:SILP}) = v(\ref{eq:FDSILP})$ and there is a zero duality gap.
\end{proof}

Combining solvability and duality, {\it strong duality} holds if there is a zero duality gap and there is an optimal solution to \eqref{eq:SILP} and \eqref{eq:FDSILP}. Putting several previous results together gives Theorem \ref{theorem:strong-duality}. 

\begin{theorem}[Strong Duality]\label{theorem:strong-duality} Strong duality holds for the primal-dual pair \eqref{eq:SILP} and \eqref{eq:FDSILP} if
\begin{itemize}
\item[(i)]  \eqref{eq:SILP} is feasible,

\item[(ii)] $\sup_{h \in I_3} \tilde{b}(h) > \lim_{\delta \rightarrow \infty} \omega(\delta)$,

\item[(iii)]  $\sup_{h \in I_3} \tilde{b}(h)$ is attained for at least one $h \in I_3.$ 
\end{itemize}
Conversely, if strong duality holds for the primal-dual pair \eqref{eq:SILP} and \eqref{eq:FDSILP} then (i) and (iii) hold as well as
\begin{itemize}
\item[(ii')] $\sup_{h \in I_3} \tilde{b}(h) \ge \lim_{\delta \rightarrow \infty} \omega(\delta)$.
\end{itemize}
\end{theorem}
\begin{proof}
Suppose conditions (i) to (iii) hold. Conditions (i) and (ii) imply primal solvability via Theorem~\ref{theorem:primal-solvability}. 
Since \eqref{eq:SILP} is feasible, by Theorem~\ref{theorem:primalFeasible}(i), $\tilde b(h) \le 0$ for all $h \in I_1$. Combined with condition (iii) dual solvability follows from Theorem~\ref{theorem:dual-solvability}. 

Conditions (i) and (ii) imply the sufficient conditions for zero duality gap given in Theorem~\ref{theorem:zero-duality-gap} and the duality gap is zero. 

Conversely, suppose strong duality holds. Then there is a  zero duality gap and so  Theorem~\ref{theorem:zero-duality-gap}, (i) and (ii') hold.  Theorem~\ref{theorem:dual-solvability}(ii) implies condition (iii). 
\end{proof}

\begin{remark}
Some authors define \emph{strong duality} to mean zero duality gap and dual solvability, excluding the requirement of primal solvability. Under this definition, properties (i), (ii') and (iii) characterize strong duality. \hfill $\triangleleft$
\end{remark}

The next two examples demonstrate how strong duality may either hold or not hold when $\sup_{h \in I_3} \tilde b(h) = \lim_{\delta \to\infty} \omega(\delta)$. 
\begin{example}[Example~\ref{example:not-primal-optimal} revisited] 
In this example the primal is feasible but not solvable, so strong duality fails. However, we showed that $\lim_{\delta\to\infty} \omega(\delta) = \sup_{h \in I_3} \tilde{b}(h) = 0$. \hfill $\triangleleft$
\end{example}
\begin{example}[Example~\ref{example:primal-solvable} revisited] In this example the primal is solvable with objective value $v(\ref{eq:SILP}) = 0$. Recall also that $\sup_{h \in I_3} \tilde{b}(h) = 0$ is attained since $I_3$ is a singleton. This implies it is dual solvable and there is zero duality gap. This problem satisfies strong duality.   However,   $\sup_{h \in I_3} \tilde{b}(h) = \lim_{\delta \rightarrow \infty} \omega(\delta)$.  Therefore condition (ii) in Theorem~\ref{theorem:strong-duality} is not satisfied,  but condition (ii') is satisfied. \hfill $\triangleleft$
\end{example}

\subsection{Summary of primal and dual results}\label{ss:summary}

Table~\ref{table:summary} summarizes the main results of this section. For brevity in displaying conditions, define $S  :=\sup_{h \in I_3} \tilde{b}(h)$ and $L  := \lim_{\delta \rightarrow \infty} \omega(\delta)$. 
\begin{table}[h]
\small
\begin{center}
\begin{tabular}{| c || c | c |}
  \hline  
  Result                   & Sets involved  & Characterization \\
  \hline             
  \hline         
  Primal feasibility (Thm~\ref{theorem:primalFeasible}) & $I_1, I_2, I_3, I_4$ & Conditions i)-iv) of Theorem~\ref{theorem:primalFeasible}  \\
  Primal boundedness (Thm~\ref{theorem:primal-not-unbounded}) & $I_3, I_4$ & Primal feas. and ($I_3 \neq \emptyset$ OR $ L > -\infty$) \\
  Primal solvability* (Thm~\ref{theorem:primal-solvability}) & $I_3, I_4$ & Primal feasible and $S > L$ \\
  \hline  
  Dual feasibility (Thm~\ref{theorem:dual-infeasibility})  & $I_3$ & $I_3 \neq \emptyset$ \\
  Dual boundedness (Thm~\ref{theorem:dual-boundedness})  & $I_1, I_3$ & Dual feas., $\tilde b(h) \le 0$ for all $h \in I_1$, $S < \infty$ \\
  Dual solvability (Thm~\ref{theorem:dual-solvability})  & $I_1,I_3$ & $\tilde b(h) \le 0$ for all $h \in I_1$, $\sup$ defining $S$ attained \\
  \hline
  Zero duality gap (Thm~\ref{theorem:zero-duality-gap})   & $I_3, I_4$ & $S \ge L$ and Primal feasible\\
  \hline
\end{tabular}
\parbox{5in}{\caption{ \emph{Summary of results from Section~\ref{s:silp-classification}}. All results are characterizations except primal solvability, where a sufficient conditions is given. 
}\label{table:summary}}
\end{center}
\end{table}
\label{table1_page}

As discussed in the introduction, alternate characterizations of these properties have been obtained by other authors. These characterizations build on a different perspective of semi-infinite linear programming, typically based around topological conditions such as lower semicontinuity and closedness in the primal constraint space. They are not immediate consequences of our characterizations, or vice versa. 

We invite the reader to compare our results with the following in the literature: primal feasibility (Table II of Kortanek \cite{kortanek1974classifying}, Theoerem 4.4 of Goberna and L\'opez \cite{goberna1998linear}), primal boundedness (Table II of Kortanek \cite{kortanek1974classifying}, Theorem 9.3 of Goberna and L\'opez \cite{goberna1998linear}), primal solvability (Theorem 7 of Kortanek \cite{kortanek1974classifying}, Table 8.1 of Goberna and L\'opez \cite{goberna1998linear}, Theorem 2.1 of Shapiro \cite{shapiro2009semi}), dual feasibility (Table II of Kortanek~\cite{kortanek1974classifying}), dual boundedness (Table II of Kortanek~\cite{kortanek1974classifying}, Theorem 9.7 of Goberna and L\'opez \cite{goberna1998linear}), dual solvability (Table 8.1 of Goberna and L\'opez \cite{goberna1998linear}, Theorem 2.3 of Shapiro \cite{shapiro2009semi}), zero duality gap (Table 8.1 of Goberna and L\'opez \cite{goberna1998linear}, Theorems 2.1 and 2.3 in Shapiro \cite{shapiro2009semi}).  
The next two subsections illustrate insights that are gained by applying the results in Table~\ref{table:summary} to  two special cases of \eqref{eq:SILP}. 

\subsection{Tidy semi-infinite linear programs}\label{s:tidy-semi-infinite}

An instance of \eqref{eq:SILP} is \emph{tidy} if, after applying  Fourier-Motzkin elimination  to \eqref{eq:initial-system-obj-con}-\eqref{eq:initial-system-con},  $z$ is the only dirty variable remaining. Fortunately, tidiness is invariant under variable permutations and alternate orders of variable elimination in the Fourier-Motzkin elimination procedure. This follows from the comments in Remark~\ref{remark:canonical} and Theorem~\ref{theorem:invariance-of-cleanliness} in the Electronic Companion. 

Tidy semi-infinite linear programs play a fundamental role in applications of our theory in later sections. The key properties of tidy systems are summarized in the following theorem.

\begin{theorem}[Tidy semi-infinite linear programs]\label{theorem:all-clean-system}
If \eqref{eq:SILP} is feasible and tidy then
\begin{enumerate}[(i)]
\item \eqref{eq:SILP} is solvable,
\item \eqref{eq:FDSILP} is  feasible and bounded,
\item there is a zero duality gap for the primal-dual pair \eqref{eq:SILP} and \eqref{eq:FDSILP}.
\end{enumerate} 
\end{theorem}
\begin{proof}
Since \eqref{eq:SILP} is tidy, $I_2 = I_4 = \emptyset$. Since $z$ cannot be eliminated, $I_4 = \emptyset$ implies $I_3 \neq \emptyset$.  In addition, $I_4 = \emptyset$ means $\omega(\delta) = -\infty$ for all $\delta$ and  $\lim_{\delta \to \infty} \omega(\delta) = - \infty$. Moreover, since $I_3 \neq \emptyset$ it follows that $\sup_{h \in I_3} \tilde{b}(h) > -\infty$.  Then, $\sup_{h \in I_3} \tilde{b}(h) > \lim_{\delta \to \infty} \omega(\delta)$ and Theorem~\ref{theorem:primal-solvability} implies that the primal is solvable.   This establishes (i). 

Since $I_3 \neq \emptyset$, \eqref{eq:FDSILP} is feasible by Theorem~\ref{theorem:dual-infeasibility}. Since the primal is feasible, Theorem~\ref{theorem:primalFeasible}(i) and (ii) imply  that the dual is bounded via Theorem~\ref{theorem:dual-boundedness}. This establishes (ii).

Since the  primal is feasible  and $\sup_{h \in I_3} \tilde{b}(h) > \lim_{\delta \to \infty} \omega(\delta)$,  Theorem~\ref{theorem:zero-duality-gap} implies that there is a zero duality gap. This establishes (iii).  
\end{proof}

The following result provides a sufficient condition for the tidiness of a semi-infinite linear program. A similar result can be found in Goberna and L\'opez \cite{goberna1998linear}. 

\begin{theorem}[Bounded System]\label{theorem:bounded-zero-duality-gap} If there exists a $\gamma \in \R$ such that the system\begin{equation}\label{eq:bounded-zero-duality-gap}
\begin{array}{rcl}
-  c_{1} x_{1} -  c_{2} x_{2} - \cdots - c_{n} x_{n}    &\ge&   -\gamma    \\
  a^{1}(i) x_{1} + a^{2}(i) x_{2} + \cdots + a^{n}(i) x_{n}  &\ge&  b(i) \quad  \text{ for } i \in I  
\end{array}
\end{equation}
is feasible and bounded then \eqref{eq:SILP} is feasible and tidy. In particular, if the set of  solutions $(x_1, \dots, x_n)$ that satisfy \eqref{eq:bounded-zero-duality-gap} is feasible and bounded for some $\gamma\in \R$, then \eqref{eq:SILP} is solvable and there is zero duality gap. 
\end{theorem}
\begin{proof}
Let $\Gamma_\gamma$ denote the set of those $x \in \R^n$ that satisfy \eqref{eq:bounded-zero-duality-gap}. Observe that the columns in  systems (\ref{eq:bounded-zero-duality-gap}) and~\eqref{eq:initial-system-obj-con}-\eqref{eq:initial-system-con} are identical for variables $x_{1}, \ldots, x_{n}$. This means if $x_k$ is eliminated when  Fourier-Motzkin elimination  is applied to one system, it will be eliminated in exactly the same order in the other. In particular, at each step of the elimination process, the sets $\mathcal{H}_0(k), \mathcal{H}_+(k)$ and $\mathcal{H}_-(k)$ are identical for the two systems.
\old{\vskip 3pt
\noindent $(\Longrightarrow)$
If \eqref{eq:SILP} is feasible and tidy, then by Theorem~\ref{theorem:bounded-zero-duality-gap} it is also bounded and solvable. Hence, there exists a feasible solution $x^*$ with $c^\top x^* = v(\ref{eq:SILP})$ since \eqref{eq:SILP} is solvable. Therefore, $\Gamma_\gamma$ is nonempty for $\gamma = v(\ref{eq:SILP})$. 

Claim $\Gamma_\gamma$ is also bounded.  Since \eqref{eq:SILP} is tidy, $x_1, \dots, x_n$ can be eliminated in \eqref{eq:initial-system-obj-con}-\eqref{eq:initial-system-con}. Thus, when applying Fourier-Motzkin elimination to \eqref{eq:bounded-zero-duality-gap}, \emph{all} variables can be eliminated and thus the system is clean. In the notation of \eqref{eq:defineI1}-\eqref{eq:defineI2} this means $H_2 = \emptyset$ and so by Theorem~\ref{theorem:region-boundedness} the set $\Gamma_\gamma$ is bounded. 

\noindent $(\Longleftarrow)$ } By hypothesis, $\Gamma_\gamma$ is non-empty and bounded so  Theorem~\ref{theorem:region-boundedness} guarantees that applying Fourier-Motzkin elimination  to~\eqref{eq:bounded-zero-duality-gap} results in a clean system. Thus, variables $x_1, \dots, x_n$ are eliminated during the procedure and so \old{by the observation made at the outset of the proof,}those variables are eliminated when applying  Fourier-Motzkin elimination  to \eqref{eq:defineI1}-\eqref{eq:defineI2}. Thus, \eqref{eq:SILP} is tidy. Since $\Gamma_\gamma$ is non-empty,   \eqref{eq:SILP} is feasible and tidy and the hypotheses of Theorem~\ref{theorem:all-clean-system} are met.  Then by Theorem~\ref{theorem:all-clean-system}, \eqref{eq:SILP} is solvable and there is a zero duality gap for the primal-dual pair  \eqref{eq:SILP} and \eqref{eq:FDSILP}.\end{proof}

\subsection{Finite linear programs}

Another special case is a semi-infinite linear program with finitely many constraints, i.e.  a finite linear program, or just a linear program. Finite linear programs are a special case of \eqref{eq:SILP} and  our analysis applies directly.

For finite linear programs,   $I_1$, $I_2$, $I_3$ and $I_4$ are always  finite sets. This  simplifies the characterizations  in Table~\ref{table:summary} since the supremums are taken over finite sets. Take, for example, primal feasibility (Theorem~\ref{theorem:primalFeasible}). Conditions ii)-iv) always hold from the finiteness of $I_2$, $I_3$ and $I_4$ respectively. Thus to determine primal feasibility it suffices to check if $\tilde b(h) \le 0$ for all $h \in I_1$. This result is well known (see for instance, Motzkin~\cite{motzkin36}).

As another example,  strong duality  holds for a finite linear program when the primal is feasible and bounded. Our framework recovers this result.

\begin{theorem}[Finite Case]\label{theorem:finite-strong-duality} If $I$ is a finite index set and \eqref{eq:SILP} is feasible and bounded, then strong duality holds for the primal-dual pair \eqref{eq:SILP} and \eqref{eq:FDSILP}.
 \end{theorem}

\begin{proof}
\noindent Note that  conditions (i)-(iii) of Theorem~\ref{theorem:strong-duality} hold. By hypothesis  \eqref{eq:SILP} is feasible and bounded so i) holds. When $I$ is a finite set, $I_4$ has finite cardinality so  $\lim_{\delta \to \infty} \omega(\delta) = -\infty$. Combining this with the hypothesis that the primal is bounded implies $I_3 \neq \emptyset$ by Theorem~\ref{theorem:primal-not-unbounded}. Thus  condition (ii) in Theorem~\ref{theorem:strong-duality} holds. Finally, (iii) holds since $I_{3}$ is finite whenever $I$ is finite. 
\end{proof}

In Section~\ref{ss:silp-finite-lp} of the Electronic Companion we illuminate  further differences between semi-infinite linear programs and finite linear programs using the tool of Fourier-Motzkin elimination.

\section{Feasible sequences and regular duality of semi-infinite linear programs}

When $I_3$ is empty in~\eqref{eq:J_system}, Theorem~\ref{theorem:dual-infeasibility} implies that the finite support dual is infeasible. 
Nevertheless, if the primal problem has optimal solution value $z^{*}$, we  show there is a sequence $\{h_m\} \in I_4$ for $m \in \N$ with the desirable property that for all $k = 1,\dots, n$, $\tilde{a}^k(h_m)$ converge to zero and $\tilde{b}(h_m)$ converges to $z^{*}$ as $m \to \infty$.  In Theorem~\ref{theorem:regular-duality-silp} it is shown that there is a sequence of finite support elements with nice limiting properties, and whose objective values converges to the {\em primal optimal value}. The terminology for this phenomenon,  standard  in conic programming, is introduced next.  The concepts date back to Duffin \cite{duffin56}.

A sequence $v^m \in \R^{(I)}$, $m \in \N$ of finite support elements is a {\em feasible sequence} for \eqref{eq:FDSILP} if $v^m \geq 0$ for all $m\in \N$, and for every $k = 1, \ldots, n$,    $\lim_{m \to \infty} (\sum_{i\in I} a^k(i)v^m(i)) = c_k$. For a feasible sequence $(v^m)_{m\in\N}$,    its {\em value} is defined by  $\text{value}((v^m)_{m\in \N}) := \lim\sup_{m \to \infty} \sum_{i\in I} b(i)v^m(i)$. For a given \eqref{eq:FDSILP},   its {\em limit value} (a.k.a. {\em subvalue})  is 
$$
\sup \{\text{value}((v^m)_{m\in \N}) \st (v^m)_{m \in N} \textrm{ is a feasible sequence for \eqref{eq:FDSILP}}\}.
$$
Since any feasible solution $v\in \R^{(I)}$ to \eqref{eq:FDSILP} naturally corresponds to a feasible sequence (where every element in the sequence is $v$),    the limit value of \eqref{eq:FDSILP} is greater than or equal to its optimal value. We   prove a remarkable theorem (Theorem~\ref{theorem:regular-duality-silp} below) relating the limit value of \eqref{eq:FDSILP} and the optimal value of the primal \eqref{eq:SILP}.

\begin{lemma}[Weak Duality-II]\label{lemma:seq-weak-duality} Let $\bar x$ be a feasible solution to the primal \eqref{eq:SILP} and let $(v^m)_{m \in \N}$ be a feasible sequence for \eqref{eq:FDSILP}. Then $c^\top \bar x \geq \text{value}((v^m)_{m\in \N})$.
\end{lemma}
\begin{proof}
Since $\bar x$ is a feasible solution to the primal \eqref{eq:SILP},    $a^1(i)\bar x_1 + \ldots + a^n(i)\bar x_n \geq b(i)$ for every $i\in I$. For each $v^m$, since $v^m(i) \geq 0$ for all $i\in I$,    $v^m(i)a^1(i)\bar x_1 + \ldots + v^m(i)a^n(i)\bar x_n \geq b(i)v^m(i)$ for every $i\in I$. Therefore, summing over all the indices $i\in I$,   gives $(\sum_{i\in I}v^m(i)a^1(i))\bar x_1 + \cdots + (\sum_{i\in I}v^m(i)a^n(i))\bar x_n  \geq \sum_{i\in I}b(i)v^m(i)$ for all $m \in \N$. Thus,
$$\begin{array}{rcl}
c_1\bar x_1 + \ldots + c_n\bar x_n & = & \lim_{m \to \infty} [(\sum_{i\in I}v^m(i)a^1(i))\bar x_1 + \cdots + (\sum_{i\in I}v^m(i)a^n(i))\bar x_n] \\
& = &\lim\sup_{m \to \infty} [(\sum_{i\in I}v^m(i)a^1(i))\bar x_1 + \cdots + (\sum_{i\in I}v^m(i)a^n(i))\bar x_n] \\
& \geq & \lim\sup_{m\to \infty} [\sum_{i\in I}b(i)v^m(i)] \\
& = & \text{value}((v^m)_{m\in \N}),
\end{array}$$
where the first equality follows from the definition of feasible sequence.\end{proof} 

The following lemma is required for the main result of the section (Theorem~\ref{theorem:regular-duality-silp}).   Applying Fourier-Motzkin elimination on \eqref{eq:SILP} gives~\eqref{eq:J_system}. Recall the function $\omega(\delta) = \sup \{ \tilde{b}(h) -  \delta \sum_{k=\ell}^{n} |\tilde{a}^{k}(h)| \, : \, h \in I_4\}$ defined in~\eqref{eq:f_delta}.

\begin{lemma}\label{lemma:convergence_lim_case} 
Suppose $\lim_{\delta \rightarrow \infty} \omega(\delta) = d$ such that $-\infty < d < \infty$. Then there exists a sequence of indices $h_m$ in $I_4$ such that $\lim_{m \to \infty}\tilde b(h_m) = d$ and $\lim_{m\to \infty}\tilde a^k(h_m) = 0$ for all $k = \ell, \ldots, n$.

%
\end{lemma}

\begin{proof}
Since $\omega(\delta)$ is a nonincreasing function of $\delta$, $\omega(\delta) \geq d$ for all $\delta$. Therefore, $d \leq \sup \{ \tilde{b}(h) -  \delta \sum_{k=\ell}^{n} |\tilde{a}^{k}(h)| \, : \, h \in I_4\}$ for every $\delta$. Define $\bar I := \{ h \in  I_4 \, : \, \tilde b(h) < d \}$. We consider two cases.

{\em \underline{Case 1: $I_4\setminus\bar I = \emptyset.$}} For any $m \in\N$, setting $\delta =  m$, we have that $d \leq \sup \{ \tilde{b}(h) -  m \sum_{k=\ell}^{n} |\tilde{a}^{k}(h)| \, : \, h \in I_4\}$ and thus, there exists $h_m \in I_4$ such that $d - \frac{1}{m} < \tilde b(h_m) - m \sum_{k=\ell}^{n} |\tilde{a}^{k}(h_m)|$. $I_4\setminus\bar I = \emptyset$ implies $\tilde b(h) < d$ for all $h \in I_4$ and therefore we have $$\begin{array}{rl}& d-\frac{1}{m} < d - m\sum_{k=\ell}^{n} |\tilde{a}^{k}(h_m)| \\ \Rightarrow & \sum_{k=\ell}^{n} |\tilde{a}^{k}(h_m)| < \frac{1}{m^2}.\end{array}$$ 

This shows that $\lim_{m \to \infty} \sum_{k=\ell}^n| \tilde{a}^k(h_m)| = 0$ which in turn implies that $\lim_{m\to\infty} \tilde a^k(h_m) = 0$ for all $k = \ell, \ldots, n$. 

Also, $$\begin{array}{rl}& d-\frac{1}{m} < \tilde b(h_m) - m \sum_{k=\ell}^{n} |\tilde{a}^{k}(h_m)| \\ \Rightarrow & d - \frac{1}{m} < \tilde b(h_m)\end{array}$$

since $m \sum_{k=\ell}^{n} |\tilde{a}^{k}(h_m)| \geq 0$. Since $\tilde b(h_m) < d$ we get $ d - \frac{1}{m} < \tilde b(h_m) < d$. And so $\lim_{m\to\infty} \tilde b(h_m) = d$.

%

{\em \underline{Case 2: $I_4\setminus\bar I \neq \emptyset.$}} We show it is sufficient to consider indices in $I_4 \setminus \bar I.$  Given any  $\delta \geq 0$, $\tilde{b}(h) -  \delta \sum_{k=\ell}^{n} |\tilde{a}^{k}(h)| < d$ for all $h \in \bar I$. Since $d \leq \sup \{ \tilde{b}(h) -  \delta \sum_{k=\ell}^{n} |\tilde{a}^{k}(h)| \, : \, h \in I_4\}$,  given $\delta \geq 0$,
$\sup \{ \tilde{b}(h) -  \delta \sum_{k=\ell}^{n} |\tilde{a}^{k}(h)| \, : \, h \in I_4\} = \sup \{ \tilde{b}(h) -  \delta \sum_{k=\ell}^{n} |\tilde{a}^{k}(h)| \, : \, h \in I_4\setminus \bar I\}.$
 Thus, $\omega(\delta) = \sup \{ \tilde{b}(h) -  \delta \sum_{k=\ell}^{n} |\tilde{a}^{k}(h)| \, : \, h \in I_4\setminus \bar I\}$ for all $\delta \geq 0$. 

First we show that there exists a sequence of indices $h_m \in I_4 \setminus \bar I$ such that $\tilde a^k(h_m) \to 0$ for all $k = \ell, \ldots, n$.  We begin by showing that $\inf\{\sum_{k=\ell}^n| \tilde{a}^k(h)| : h \in I_4\setminus \bar I\} = 0$. This implies that there is a sequence $h_m \in I_4\setminus \bar I$ such that $\lim_{m \to \infty} \sum_{k=\ell}^n| \tilde{a}^k(h_m)| = 0$ which in turn   implies that $\lim_{m\to\infty} \tilde a^k(h_m) = 0$ for all $k = \ell, \ldots, n$. 

Suppose to the contrary that $\inf\{\sum_{k=\ell}^n| \tilde{a}^k(h)| : h \in I_4\setminus \bar I\} = \beta > 0$. Since $\omega(\delta)$ is nonincreasing and $\lim_{\delta \to \infty}\omega(\delta) = d < \infty$, there exists $\bar \delta \geq 0$ such that $\omega(\bar\delta) < \infty$. Observe that $d = \lim_{\delta \to \infty}\omega(\delta) = \lim_{\delta \to\infty} \omega(\bar\delta + \delta)$. Then, for every $\delta \geq 0$, 
\[
\begin{array}{rcl}
\omega(\bar\delta + \delta) &=& \sup \{ \tilde{b}(h) -  (\bar\delta + \delta) \sum_{k=\ell}^{n} |\tilde{a}^{k}(h)| \, : \, h \in I_4\setminus \bar I\} \\
&=& \sup \{ \tilde{b}(h) -  \bar\delta  \sum_{k=\ell}^{n} |\tilde{a}^{k}(h)|  -\delta \sum_{k=\ell}^{n} |\tilde{a}^{k}(h)| \, : \, h \in I_4\setminus \bar I\} \\
&\le& \sup \{ \tilde{b}(h) -  \bar\delta  \sum_{k=\ell}^{n} |\tilde{a}^{k}(h)|  -\delta \beta \, : \, h \in I_4\setminus \bar I\} \\
&=& \sup \{ \tilde{b}(h) -  \bar\delta \sum_{k=\ell}^{n} |\tilde{a}^{k}(h)| \, : \, h \in I_4\setminus \bar I\} - \delta\beta\\
 &=& \omega(\bar \delta) - \delta\beta.
\end{array}
\]
Therefore, $d = \lim_{\delta\to\infty} \omega(\bar \delta + \delta) \leq \lim_{\delta \to \infty} (\omega(\bar\delta) - \delta\beta) = -\infty$, since $\beta > 0$ and $\omega(\bar\delta) < \infty$.  This contradicts $-\infty < d$. Thus $\beta = 0$ and  there is a sequence $h_m \in I_4 \setminus \bar I$ such that $\tilde a^k(h_m) \to 0$ for all $k = \ell, \ldots, n$.

Now we show there is a  subsequence of  $\tilde b(h_m)$ that converges to $d$.  Since $\lim_{\delta \to \infty}\omega(\delta) = d$, there is a sequence $(\delta_p)_{p \in \N}$ such that $\delta_p \geq 0$ and $\omega(\delta_p) < d + \frac{1}{p}$ for all $p \in \N$.   It was shown above that the sequence $h_m \in I_4\setminus \bar I$ is such that $\lim_{m \to \infty} \sum_{k=\ell}^n |\tilde{a}^k(h_m)| = 0$.  This implies that for every $p \in \N$ there is an $m_{p} \in \N$ such that for all $m \ge m_{p},$ $\delta_{p} \sum_{k = \ell}^{n} | \tilde{a}^{k}(h_{m})| < \frac{1}{p}.$ Thus, one can extract a subsequence $(h_{m_p})_{p\in \N}$ of $(h_m)_{m\in \N}$ such that $\delta_{p} \sum_{k = \ell}^{n} | \tilde{a}^{k}(h_{m_p})| < \frac{1}{p}$ for all $p \in \N.$ Then
\[
d + \frac{1}{p} > \omega(\delta_{p}) = \sup \{\tilde{b}(h) - \delta_p \sum_{k=\ell}^{n} |\tilde{a}^k(h)| \, : \, h \in I_4\setminus \bar I \} \ge \tilde{b}(h_{m_{p}}) -  \delta_{p} \sum_{k = \ell}^{n} | \tilde{a}^{k}(h_{m_{p}})|.
\]
The second inequality,  along with $\delta_{p} \sum_{k = \ell}^{n} | \tilde{a}^{k}(h_{m})| < \frac{1}{p},$ and the fact that $h_{m_{p}} \in I_{4} \backslash \overline{I}$ implies $\tilde{b}(h_{m_{p}}) \ge d$, gives
$
d + \frac{2}{p} \ge \tilde{b}(h_{m_{p}}) \ge d
$
and  $\tilde{b}(h_{m_{p}}), p \in \N$ is the desired subsequence.
\end{proof}

\begin{theorem}[Regular duality of semi-infinite linear programs]\label{theorem:regular-duality-silp}  If  \eqref{eq:SILP}  has an optimal primal value $z^{*}$, where  $-\infty < z^* <\infty$,  then the limit value $\hat d$ of \eqref{eq:FDSILP} is finite and $z^* = \hat d$.
\end{theorem}
\begin{proof}
 By Lemma~\ref{lemma:primal-optimal-value}, $z^{*} = \max\{  \sup\{ \tilde{b}(h) \, : \, h \in I_3 \},   \lim_{\delta \rightarrow \infty} \omega(\delta) \}$.    If $z^* = \sup\{ \tilde{b}(h) \, : \, h \in I_3 \}$, then by Theorem~\ref{theorem:zero-duality-gap}, there is a  zero duality gap, i.e., $z^* = d^*$ where $d^*$ is the optimal value of \eqref{eq:FDSILP}.   From Lemma~\ref{lemma:seq-weak-duality},  $\hat d \leq z^*$, so  $z^{*} = d^{*}$ implies $\hat{d} \le d^{*}.$ By definition of limit value, $\hat{d} \ge d^{*}.$  Therefore, $d^{*} = \hat{d} = z^*.$

In the other case when $z^* =  \lim_{\delta \rightarrow \infty} \omega(\delta)$,  by   Lemma~\ref{lemma:convergence_lim_case}  there is  a sequence $h_m \in  I_4$ such that $\lim_{m \to \infty}\tilde b(h_m) = z^*$ and $\lim_{m\to \infty}\tilde a^k(h_m) = 0$ for all $k = \ell, \ldots, n$. By Lemma~\ref{lemma:feasible-FDSILP} there exist $v^{h_m} \in \R^{(I)}_+$ for each $m \in \N$ such that $- c_k + \sum_{i\in I}v^{h_m}(i)a^k(i) = 0$ for $k = 1, \ldots, \ell-1$, $- c_k + \sum_{i\in I}v^{h_m}(i)a^k(i) = \tilde a^k(h_m)$ for $k = \ell, \ldots, n$, and $\sum_{i\in I}b(i)v^{h_m}(i) = \tilde b(h_m)$.   Since $\lim_{m\to \infty}\tilde a^k(h_m) = 0$ for all $k = \ell, \ldots, n$, and $\lim_{m \to \infty}\tilde b(h_m) = z^*$,  $v^{h_m}$, $m \in \N$ is a feasible sequence with value $z^*$. Thus, $\hat d \geq z^*$.      Again, from Lemma~\ref{lemma:seq-weak-duality},  $\hat d \leq z^*$,  so $z^* = \hat d$. 
\end{proof}

\section{Application: Convex programs} \label{s:application-convex}

Recall the convex program \eqref{eq:CP} and its Lagrangian dual \eqref{eq:LD} introduced in Section~\ref{s:introduction}. 
%
Construct the semi-infinite linear program
%
\begin{equation}\label{eq:convex-silp}
\begin{array}{clcl}
 \inf &  \sigma && \\
 {\rm s.t.} & \sigma - \sum_{i=1}^p\lambda_ig_i(x) &\ge& f(x) \quad \text{ for } x \in \Omega \\
& \phantom{\sigma + \sum_{i=1}^p}\lambda & \geq & 0.
\end{array}\tag{\text{CP-SILP}} 
\end{equation}
along with its finite support dual for~\eqref{eq:convex-silp}. There are two sets of  constraints in~\eqref{eq:convex-silp}.  There are typically an uncountable number of constraints indexed by $x\in \Omega$ and a finite number of nonnegativity, $\lambda \ge 0,$  constraints indexed by  $\{1, \ldots, p\}$. Thus, the finite support dual elements belong to $\R^{(\Omega\cup\{1, \ldots, p\})}$. The finite support  dual defined over $(u, v) \in \R^{(\Omega)}\times \R^p$ is
\begin{eqnarray}\label{eq:convex-fdsilp}
 (\text{CP-FDSILP})  \qquad   \sup  \sum_{x\in \Omega} u(x)f(x) && \label{eq:obj}\\
 {\rm s.t.} \qquad \sum_{x\in \Omega} u(x) & = & 1 \label{eq:convex}\\
-\sum_{x\in \Omega}u(x)g_i(x) + v_i & = & 0 \quad \text{ for } i = 1, \ldots, p \label{eq:g}\\
 (u, v) & \in & \R^{(\Omega)}_+\times \R^p_+. \label{eq:nonnegative}
\end{eqnarray} 
Recall $v(\text{CP})$ is the  optimal value of \eqref{eq:CP}, $v(\ref{eq:LD})$ is the optimal value of \eqref{eq:LD}, $v(\ref{eq:convex-silp})$ is the optimal value of \eqref{eq:convex-silp} and $v(\ref{eq:convex-silp})$ is the optimal value of \eqref{eq:convex-fdsilp}-\eqref{eq:nonnegative}. 
\begin{remark}\label{rem:optimal-inequalities}
We show in the appendix (Theorems~\ref{theorem:LD-SILP}~and~\ref{theorem:CP-FDSILP}) the following holds:
\begin{align*}
v(\text{CP-SILP}) = v(\text{LD}) \geq v(\text{CP}) = v(\text{CP-FDSILP})
\end{align*}
where the inequality follows from weak duality of the Lagrangian dual (or the weak duality of semi-infinite linear programs as discussed in Section~\ref{s:silp-classification}). 
\hfill$\triangleleft$
\end{remark}

We are now able to provide a new proof of a very well-known sufficient condition for zero duality gaps in convex programming.
\begin{theorem}[Slater's theorem for convex programs]\label{thm:slater-convex}
Assume  the convex program~\eqref{eq:CP} is feasible and bounded, i.e., $-\infty < v(\text{CP}) < \infty$ and 
there exists a $x^* \in \Omega$ such that $g_i(x^*) > 0$ for all $i=1, \ldots, p$. Then there is a zero duality gap between the convex program~\eqref{eq:CP} and its Lagrangian dual~\eqref{eq:LD} and  there exists a $\lambda^* \geq 0$ such that $v(\text{LD}) = L(\lambda^*)$, i.e., the Lagrangian dual is solvable.
\end{theorem}
\begin{proof} 
Since $v(\text{\ref{eq:CP}}) < \infty$, it is valid to  replace the objective function $f(x),$ by the concave function $\tilde{f}(x) = \min\{f(x), B\}$,  where $B$ is an upper bound on $v(\text{CP})$. Thus, we assume that \eqref{eq:convex-silp} is feasible : $\sigma = B$, $\lambda = 0$ where $B$ is an upper bound on $f(x)$. 

We now perform Fourier Motzkin on \eqref{eq:convex-silp} after reformulating as in Section~\ref{s:silp-classification}:

$$
\begin{array}{rclcl}
 z & - & \sigma &\ge & 0\\
 & & \sigma - \sum_{i=1}^p\lambda_ig_i(x) &\ge& f(x) \qquad \forall x \in \Omega \\
& & \phantom{\sigma + \sum_{i=1}^p}\lambda_i & \geq & 0 \qquad i=1, \ldots, p
\end{array}$$

We first eliminate variable $\sigma$ and end up in the following intermediate system during the Fourier-Motzkin elimination procedure:

\begin{equation}\label{eq:intermediate}
\begin{array}{clcl}
 & z - \sum_{i=1}^p\lambda_ig_i(x) &\ge& f(x) \qquad \forall x \in \Omega \\
& \phantom{\sigma + \sum_{i=1}^p}\lambda_i & \geq & 0 \qquad i=1, \ldots, p
\end{array}\end{equation}

\begin{claim}\label{claim:slater-convex-clean}
The variables $\lambda_1, \ldots, \lambda_p$ remain clean as the Fourier Motzkin elimination procedure proceeds on \eqref{eq:intermediate}.
\end{claim}
\begin{proof}[Proof of Claim]\renewcommand{\qedsymbol}{}  
We now track the intermediate inequalities produced by the Fourier Motzkin elimination procedure as we go through $\lambda_1, \ldots, \lambda_p$. We claim that after processing variables $\lambda_1, \lambda_2, \ldots \lambda_k$ where $1 \leq k \leq p$ we have the inequality $z - \sum_{i=k+1}^p\lambda_ig_i(x^*) \ge f(x^*)$ in the intermediate system of inequalities. We prove this by induction on $k$.

Consider $k=1$ first. We have the constraint corresponding to $x^*$: $z - \sum_{i=1}^p\lambda_ig_i(x^*) \ge f(x^*)$ in \eqref{eq:intermediate}. Since $g_1(x^*) > 0$ by hypothesis, the coefficient of $\lambda_1$ is negative in this constraint. Moreover, we have the constraint $\lambda_1 \geq 0$. We can multiply the constraint $\lambda_1 \geq 0$ by $g_1(x^*)$ and add to $z - \sum_{i=1}^p\lambda_ig_i(x^*) \ge f(x^*)$, resulting in the inequality $z - \sum_{i=2}^p\lambda_ig_i(x^*) \ge f(x^*)$. So the base case is done.

Now for the induction step for $k>1$. By the induction hypothesis, we have the constraint $z - \sum_{i=k}^p\lambda_ig_i(x^*) \ge f(x^*)$ after processing $\lambda_1, \ldots, \lambda_{k-1}$. Since $g_k(x^*) > 0$ the coefficient of $\lambda_k$ is negative in this constraint. We also have the constraint $\lambda_{k} \geq 0$ in the intermediate system obtained after processing $\lambda_1, \ldots, \lambda_{k-1}$. Multiplying the constraint $\lambda_{k} \geq 0$ by $g_k(x^*)$ and adding to $z - \sum_{i=k}^p\lambda_ig_i(x^*) \ge f(x^*)$, we obtain the constraint $z - \sum_{i=k+1}^p\lambda_ig_i(x^*) \ge f(x^*)$. Thus the induction is complete.\quad $\dagger$
\end{proof}

By Claim~\ref{claim:slater-convex-clean}, we have that all variables except $z$ are clean throughout the Fourier-Motzkin elimination procedure. Since \eqref{eq:convex-silp} is feasible (by the discussion in the first paragraph of the proof), by Theorem~\ref{theorem:all-clean-system} $v(\text{CP-SILP}) = v(\text{CP-FDSILP})$ and $(\text{CP-SILP})$ is solvable. By Remark~\ref{rem:optimal-inequalities} we have $v(\text{CP}) = v(\text{LD})$. Moreover, since~\eqref{eq:convex-silp} is solvable, by Theorem~\ref{theorem:LD-SILP} there exists $\lambda^*$ such that $v(\text{LD}) = L(\lambda^*)$.\end{proof}

The following example demonstrates that  it is possible to identify  a zero duality gap with  techniques of this paper, even when a Slater condition fails. 

\begin{example}\label{ex:no-slater-duality}
Consider the    convex optimization problem
\begin{align}\label{eq:no-slater-example}
\begin{array}{rl}
\qquad  \max_{x\in \R^n} & 0 \\
 \textrm{s.t.} & \phantom{-}1 - x_1^2 - x_2^2 \ge 0 \\
               & -1 + x_1 \phantom{- 3x_2^2} \ge 0.
\end{array}
\end{align}
The feasible region is the singleton $\left\{(1,0) \right\}$ and so no Slater point exists, however there is a zero duality gap. For this instance, \eqref{eq:convex-silp} is
\begin{equation}\label{eq:convex-example}
\begin{array}{clcl}
\inf &  \sigma && \\
 {\rm s.t.} & \sigma + \lambda_1(x_1^2 + x_2^2 - 1) + \lambda_2 (1 - x_1) &\ge& 0 \quad \text{ for } x \in \R^n \\
& \phantom{\sigma - \lambda_1(x_1^2 + x_2^2 - 1) + }\lambda & \geq & 0.
\end{array}
\end{equation}
Setting $(\sigma, \lambda_1, \lambda_2) = (0,0,0)$ shows that this semi-infinite linear program  \eqref{eq:SILP}  is feasible. Notice also that the right-hand function $b$ is  the zero function.   Applying  Fourier-Motzkin elimination  to~\eqref{eq:convex-example} gives   $\tilde{b}(h) = 0$  for all  $h$ and  this implies  $\sup_{h \in I_3} \tilde b(h) = 0$. Also, for any $\delta \geq 0$,
$
\omega(\delta) = \sup_{h \in I_4} \left\{\tilde b(h) - \delta \sum_{k=\ell}^n |\tilde a^k(h)|\right\} = \sup_{h \in I_4} \left\{- \delta \sum_{k=\ell}^n |\tilde a^k(h)|\right\} \le 0.
$
\ Then $\sup_{h \in I_3} \tilde{b}(h) \geq \lim_{\delta \to \infty}\omega(\delta)$ and by  Theorem~\ref{theorem:zero-duality-gap} there is a zero duality gap between \eqref{eq:convex-example} and its finite support dual. By Theorem~\ref{theorem:LD-SILP}~and~\ref{theorem:CP-FDSILP} this implies there is a zero duality gap between \eqref{eq:no-slater-example} and its Lagrangian dual. \hfill $\triangleleft$
\end{example}
\section{Application: Generalized Farkas' Theorem}\label{s-application-farkas-minkowski}

In this section,  Fourier-Motzkin elimination provides an alternate proof of the generalized Farkas' theorem, a well-known cornerstone result in the semi-infinite linear programming literature (see Goberna and L\'opez~\cite{goberna1998linear}).  Consider a closed convex set given as the intersection of (possibly infinitely many) halfspaces  
\begin{equation}\label{eq:convex_set}
P = \{x \in \R^n \st a^1(i)x_1 + \cdots + a^n(i)x_n \geq b(i) \text{ for } i\in I\},
\end{equation} 
where $I$ is any index set, $a^1, \ldots, a^n$ and $b$ are elements of $\R^I$. An inequality $c^\top x \geq d$ is a {\em consequence} of the system of inequalities $a^1(i)x_1 + \ldots a^n(i)x_n \geq b(i)$, $i\in I$ if $c^\top x \geq d$ for every $x \in P$. If  $P = \emptyset$, then   every inequality is a consequence the inequalities  $a^1(i)x_1 + \ldots a^n(i)x_n \geq b(i)$, $i\in I.$   Let $\alpha^i$ denote the vector in $\R^n$ given by $\alpha^i = (a^1(i), \ldots, a^n(i))^{\top}$. The notation $0_n$ is used to denote the $n$-dimensional vector of zeros.

In the theorem below, the difficulty is proving necessity of the conditions. We show how our Fourier-Motkzin approach can be used to prove necessity, as opposed to a separating hyperplane theorem, as was done in Goberna and L\'opez~\cite{goberna1998linear}. The sufficiency direction is identical to that of Theorem 3.1 in Goberna and L\'opez~\cite{goberna1998linear} and is omitted.

\begin{theorem}[Generalized Farkas' Theorem, see Theorem 3.1 in Goberna and L\'opez~\cite{goberna1998linear}]\label{thm:gen-farkas-minkowski}
The inequality $c^\top x \geq d$ is a consequence of $(\alpha^i)^\top x \geq b(i)$ for all $i\in I$, if and only if at least one of the following holds:
\begin{enumerate}[(i)]
\item \(\displaystyle\left[\begin{array}{c}c \\ d\end{array}\right] \in \cl\bigg(\cone\bigg(\left\{\left[\begin{array}{c}0_n \\ -1\end{array}\right], \left[\begin{array}{c}\alpha^i \\ b(i)\end{array}\right];\; i\in I\right\}\bigg)\bigg)\) \label{first-part}
\item \(\displaystyle\left[\begin{array}{c}0_n \\ 1\end{array}\right] \in \cl\bigg(\cone\bigg(\left\{\left[\begin{array}{c}\alpha^i \\ b(i)\end{array}\right];\; i\in I\right\}\bigg)\bigg).\) \label{second-part}
\end{enumerate}
\end{theorem}

\begin{proof}

Assume $c^\top x \geq d$ is a consequence. There are two cases, depending on whether $P$ is empty or not. 

\noindent\underline{\emph{Case 1: $P = \emptyset$.}} Apply the Fourier-Motzkin elimination procedure to the constraints  that define $P$ in \eqref{eq:convex_set} and obtain the system (\ref{eq:defineI1})-(\ref{eq:defineI2}) with the corresponding index sets $H_1$ and $H_2$. Since $P = \emptyset$, by Theorem~\ref{theorem:feasible} either $\tilde{b}(h)  > 0$ for some $h^* \in H_{1}$, or $ \sup \{ \tilde{b}(h) / \sum_{k=\ell}^{n} |\tilde{a}^{k}(h)|  \, :  \, h \in H_{2}   \}  = \infty$.   Consider these two cases in turn:
\medskip

\underline{\emph{Case 1a:}  $\tilde{b}(h^*)  > 0$ for some $h^* \in H_{1}$.}  By Theorem~\ref{theorem:FM-elim-succ}, there exists $u^{h^*} \in \R^{(I)}_+$ with finite support such that $\langle a^j, u^{h^*} \rangle = 0$ for all $j = 1, \ldots, n$ and $\langle b, u^{h^*}\rangle > 0$. Using the multiplers $\frac{u^{h^*} }{\langle b, u^{h^*}\rangle}$ for the constraints corresponding to the non-zero elements in $u^{h^*}$ to aggregate constraints, gives  $\left[\begin{array}{c}0_n \\ 1\end{array}\right] \in \cone\bigg(\left\{\left[\begin{array}{c}\alpha^i \\ b(i)\end{array}\right];\; i\in I\right\}\bigg)$. Condition \eqref{second-part} in the statement of the theorem is  satisfied.
\medskip 

\underline{\emph{Case 1b:}  $ \sup_{h \in H_2} \tilde{b}(h) / \sum_{k=\ell}^{n} |\tilde{a}^{k}(h)| = \infty$.} This implies that there is a sequence $h_m \in H_2$, $m = 1, 2, \ldots$ such that $\tilde b(h_m) / \sum_{k=\ell}^{n} |\tilde{a}^{k}(h_m)| > m$. This implies $\tilde b(h_m) > 0$ for all $m$. Rearranging the terms,  gives $\lim_{m\to \infty}\frac{\sum_{k=\ell}^{n} |\tilde{a}^{k}(h_m)|}{\tilde b(h_m)} = 0.$ The above limit implies $\lim_{m\to \infty}\frac{\tilde{a}^k(h_m)}{\tilde b(h_m)} = 0$ for  $k=\ell, \ell+ 1, \ldots, n$.
By  Theorem~\ref{theorem:FM-elim-succ}, there exists $u^{h_m} \in \R^{(I)}_+$ with finite support such that $\langle a^j, u^{h_m} \rangle = 0$ for  $j = 1, \ldots, \ell-1$, $\langle a^j, u^{h_m} \rangle = \tilde{a}^j(h_m)$ for $j=\ell, \ldots, n$ and $\langle b, u^{h_m} \rangle = \tilde b(h_m)$. Since  $\tilde b(h_m) > 0$,  $\langle a^j, \frac{u^{h_m}}{\tilde b(h_m)} \rangle = 0$ for all $j = 1, \ldots, \ell-1$, $\langle a^j, \frac{u^{h_m}}{\tilde b(h_m)} \rangle = \frac{\tilde a^j(h_m)}{\tilde b(h_m)}$ for $j=\ell, \ldots, n$ and $\langle b, \frac{u^{h_m}}{\tilde b(h_m)} \rangle = 1$.  Since $\lim_{m\to \infty} \frac{\tilde a^j(h_m)}{\tilde b(h_m)} = 0$ for  $j = 1, \ldots, n$, this gives a sequence of points in $\cone\bigg(\left\{\left[\begin{array}{c}\alpha^i \\ b(i)\end{array}\right];\; i\in I\right\}\bigg)$ that converges to $\left[\begin{array}{c}0_n \\ 1\end{array}\right]$ and condition \eqref{second-part} holds.

\vskip 5pt

\noindent\underline{\emph{Case 2: $P\neq \emptyset$.}} Consider the  semi-infinite linear program

\begin{equation}\label{eq:semi-infinite2}
\begin{array}{rl}
\inf_{x\in \R^n} & c^\top x \\
 \textrm{s.t.} & a^1(i)x_1 + a^2(i)x_2 + \cdots + a^n(i)x_n \geq b(i), \quad \text{ for } i\in I.
\end{array}
\end{equation}
If $P \neq \emptyset$, the semi-infinite linear program defined by~\eqref{eq:semi-infinite2} is feasible, i.e., $z^*  <\infty$. Since $c^\top x \geq d$ is a consequence, \eqref{eq:semi-infinite2} is bounded, i.e., $z^* \geq d > -\infty$. Reformulate as in~\eqref{eq:initial-system-obj}-\eqref{eq:initial-system-con} and apply  Fourier-Motzkin elimination and obtain the system~\eqref{eq:J_system} with the corresponding index sets $I_1, I_2, I_3$ and $I_4$. Then by Lemma~\ref{lemma:primal-optimal-value}  the primal optimal value is
\begin{eqnarray*}
z^{*} = \max\{  \sup_{h \in I_3} \tilde{b}(h),   \lim_{\delta \rightarrow \infty} \omega(\delta)    \}.
\end{eqnarray*}

Again consider two cases :

\underline{\emph{Case 2a:} $z^{*} = \sup_{h \in I_3} \tilde{b}(h)$.} This implies that for any fixed $\epsilon > 0$ there is an $h^* \in I_3$ such that $\tilde b(h^*) \geq z^* - \epsilon \geq d - \epsilon$. Since $h^* \in I_3$, Lemma~\ref{lemma:feasible-FDSILP}\eqref{item:I3-property} implies  that there exists $v^{h^*} \in \R^{(I)}$ such that $\langle a^j, v^{h^*} \rangle = c_j$ and $\tilde b(h^*) = \langle b, v^{h^*}\rangle \geq d-\epsilon$. Thus, $\left[\begin{array}{c} c \\ d -\epsilon \end{array}\right]$ is in $\cone\bigg(\left\{\left[\begin{array}{c}0_n \\ -1 \end{array}\right], \left[\begin{array}{c}\alpha^i \\ b(i)\end{array}\right];\; i\in I\right\}\bigg)$ where the multiplier for $\left[\begin{array}{c}0_n \\ -1 \end{array}\right]$ is $\tilde b(h^*) - (d -\epsilon)$. Since this is true for any  $\epsilon >0$,   $\left[\begin{array}{c}c \\ d\end{array}\right] \in \cl\bigg(\cone\bigg(\left\{\left[\begin{array}{c}0_n \\ -1\end{array}\right], \left[\begin{array}{c}\alpha^i \\ b(i)\end{array}\right];\; i\in I\right\}\bigg)\bigg)$ and condition \eqref{first-part} of the theorem holds.

\underline{\emph{Case 2b:} $z^{*} = \lim_{\delta \rightarrow \infty} \omega(\delta)$. } 
Since $-\infty < z^* < \infty$, by Lemma~\ref{lemma:convergence_lim_case}, there exists a subsequence of indices $h_m, m= 1, 2 , \ldots$ such that $h_m \in I_4$, $\tilde a^k(h_m) \to 0$ for all $k = \ell, \ldots, n$ and $\tilde b(h_m) \to z^*$. Let $\tilde \alpha^m \in \R^n$ be defined by $(\tilde{\alpha}^m)_k = 0$ for $k =1, \ldots, \ell-1$ and $(\tilde{\alpha}^m)_k = \tilde a^k(h_m)$ for $k=\ell, \ldots, n $. By Lemma~\ref{lemma:feasible-FDSILP}\eqref{item:I4-property}, for each $m\in \N$, $\tilde \alpha^m = \alpha^m - c$, for some $\alpha^m \in \cone(\{\alpha^i\}_{i\in I})$. Renaming $\tilde b(h_m) = b_m$, gives $\left[\begin{array}{c}\alpha^m \\ b_m \end{array}\right] \in \cone\bigg(\left\{\left[\begin{array}{c}\alpha^i \\ b(i)\end{array}\right];\; i\in I\right\}\bigg)$
\old{and
$$\left[\begin{array}{c}\tilde \alpha^m \\ 0 \end{array}\right] = \left[\begin{array}{c}\alpha^m - c \\ b_m - d\end{array}\right] + (b_m - d)\left[\begin{array}{c}0_n \\ -1 \end{array}\right].$$
Since $\tilde \alpha^m\to 0$ as $m \to \infty$,

$$\begin{array}{rl}& \left[\begin{array}{c}\alpha^m - c \\ b_m - d\end{array}\right] + (b_m - d)\left[\begin{array}{c}0_n \\ -1 \end{array}\right]\to 0 \\ \\
\Rightarrow & \left[\begin{array}{c}\alpha^m \\ b_m\end{array}\right] + (b_m - d)\left[\begin{array}{c}0_n \\ -1 \end{array}\right] - \left[\begin{array}{c}c  \\ d\end{array}\right]\to 0 \\ \\
\Rightarrow & \left[\begin{array}{c}\alpha^m \\ b_m\end{array}\right] + (b_m - d)\left[\begin{array}{c}0_n \\ -1 \end{array}\right]\to \left[\begin{array}{c}c  \\ d\end{array}\right].
\end{array}
$$}
and $\left[\begin{array}{c}\tilde \alpha^m \\ b_m - z^* \end{array}\right] = \left[\begin{array}{c}\alpha^m - c \\ b_m - d\end{array}\right] + (z^* - d)\left[\begin{array}{c}0_n \\ -1 \end{array}\right].$
Since $\tilde \alpha^m\to 0$ and $b_m \to z^*$ as $m \to \infty$,

$$\begin{array}{rl}& \left[\begin{array}{c}\alpha^m - c \\ b_m - d\end{array}\right] + (z^* - d)\left[\begin{array}{c}0_n \\ -1 \end{array}\right]\to 0 \\ \\
\Rightarrow & \left[\begin{array}{c}\alpha^m \\ b_m\end{array}\right] + (z^* - d)\left[\begin{array}{c}0_n \\ -1 \end{array}\right] - \left[\begin{array}{c}c  \\ d\end{array}\right]\to 0 \\ \\
\Rightarrow & \left[\begin{array}{c}\alpha^m \\ b_m\end{array}\right] + (z^* - d)\left[\begin{array}{c}0_n \\ -1 \end{array}\right]\to \left[\begin{array}{c}c  \\ d\end{array}\right].
\end{array}
$$

Now $z^* \geq d$.   Therefore
 $\left[\begin{array}{c}c \\ d\end{array}\right] \in \cl\bigg(\cone\bigg(\left\{\left[\begin{array}{c}0_n \\ -1\end{array}\right], \left[\begin{array}{c}\alpha^i \\ b(i)\end{array}\right];\; i\in I\right\}\bigg)\bigg)$ and   condition \eqref{first-part}  of the theorem holds. 
\end{proof}
\section{Conclusion}\label{s:conclusion}

This paper explores two related themes. The first is how the powerful extension of Fourier-Motzkin elimination to semi-infinite systems of linear inequalities is used to prove and provide insights about duality theory for semi-infinite linear programs. The second theme is that semi-infinite linear programming has implications for finite dimensional convex optimization. 

The connection between semi-infinite linear programming and convex optimization is made clear by the method of projection.
Fourier-Motzkin elimination is purely algebraic. It is simply the aggregation of pairs of linear inequalities using nonnegative multipliers. The key insight is that topological conditions  common in the duality theory of finite-dimensional convex and conic programming imply simple conditions that ensure duality results. There is no need to appeal to advanced convex analysis or results from the theory of topological vector spaces. 

%

Both themes, and the connections between them, deserve further exploration. Regarding the first, it might be fruitful to further explore the connections between our characterization of zero duality gap and the characterization presented in Theorem~8.2 of Goberna and L\'opez~\cite{goberna1998linear}. Goberna and L\'opez's approach is topological and based on separating hyperplane theory, whereas our approach is based on the purely algebraic Fourier-Motzkin elimination procedure. Our proof of the generalized Farkas'  theorem (see our Theorem~\ref{thm:gen-farkas-minkowski} and Theorem 3.1 in Goberna and L\'opez \cite{goberna1998linear}) provides a useful starting point for further exploration.

Regarding the second theme,  there are at least two avenues for further research. First, all the duality results for finite-dimensional convex optimization considered here were derived by showing the associated semi-infinite linear program was tidy. Recall that when \eqref{eq:SILP} is tidy, $\lim_{\delta \to \infty}\omega(\delta) = - \infty$. This condition (along with primal feasibility) suffices to establish primal solvability (Theorem~\ref{theorem:primal-solvability}) and zero duality gap (Theorem~\ref{theorem:zero-duality-gap}). However, tidiness is far from necessary, as demonstrated in Examples~\ref{example:primal-solvable}~and~\ref{ex:no-slater-duality}. Exploring how to translate more subtle sufficient conditions for zero duality gap arising from finite values for $\lim_{\delta \to \infty}\omega(\delta)$ into the language of finite dimensional convex optimization could prove fruitful.

This paper has not addressed the algorithmic aspects of Fourier-Motzkin elimination applied to semi-infinite linear programs. There is considerable work on computational approaches to solving semi-infinite linear programs, see for instance Glashoff and Gustavson \cite{glashoff1983linear} and Stein and Still \cite{stein2003solving}.
Obviously, when applied to semi-infinite linear programs, Fourier-Motzkin elimination is not a finite process, so a direct comparison with existing computational methods will certainly prove unfavorable for our approach. However, if the functions $b, a^k \in \R^I$ for $k = 1, \ldots, n$ could be characterized in a reasonably simple format, then symbolic elimination might be possible. 

\endgroup

\section*{Acknowledgements}

The authors thank the reviewers and associate editor for generous and highly insightful comments that led to a much better presentation.  

\bibliographystyle{plain}
\bibliography{references}

\begin{thebibliography}{10}

\bibitem{hitchhiker}
C.D. Aliprantis and K.C. Border.
\newblock {\em Infinite Dimensional Analysis: A Hitchhiker's Guide}.
\newblock Springer Verlag, second edition, 2006.

\bibitem{anderson-nash}
E.J. Anderson and P.~Nash.
\newblock {\em Linear Programming in Infinite-Dimensional Spaces: {T}heory and
  Applications}.
\newblock Wiley, 1987.

\bibitem{blair1974extension}
C.E. Blair.
\newblock An extension of a theorem of {J}eroslow and {K}ortanek.
\newblock {\em Israel Journal of Mathematics}, 17(1):111--115, 1974.

\bibitem{charnes1963duality}
A.~Charnes, W.W. Cooper, and K.~Kortanek.
\newblock Duality in semi-infinite programs and some works of {H}aar and
  {C}arath\'{e}odory.
\newblock {\em Management Science}, 9(2):209--228, 1963.

\bibitem{duffin56}
R.J. Duffin.
\newblock Infinite programs.
\newblock In H.~W. Kuhn and A.~W. Tucker, editors, {\em Linear Inequalities and
  Related Systems}, pages 157--170. Princeton University Press, Princeton, NJ,
  1956.

\bibitem{duffin-karlovitz65}
R.J. Duffin and L.A. Karlovitz.
\newblock An infinite linear program with a duality gap.
\newblock {\em Management Science}, 12:122--134, 1965.

\bibitem{fourier26}
J.~B.~J. Fourier.
\newblock Solution d'une question particuli\`{e}re du calcul des
  in\'{e}galit\'{e}s.
\newblock {\em Oeuvres II Paris}, pages 317--328, 1826.

\bibitem{gartner-matousek}
B.~G\"artner and J.~Matous\'ek.
\newblock {\em Approximation Algorithms and Semi-Definite Programming}.
\newblock Springer-Verlag, 2012.

\bibitem{glashoff1983linear}
K.~Glashoff and S.~Gustafson.
\newblock {\em Linear Optimization and Approximation: An Introduction to the
  Theoretical Analysis and Numerical Treatment of Semi-infinite Programs}.
\newblock Springer-verlag New York, 1983.

\bibitem{goberna1998linear}
M.A. Goberna and M.A. L{\'o}pez.
\newblock {\em Linear semi-infinite optimization}.
\newblock John Wiley \& Sons, Chichester, 1998.

\bibitem{goberna2002linear}
M.A. Goberna and M.A. L{\'o}pez.
\newblock Linear semi-infinite programming theory: an updated survey.
\newblock {\em European Journal of Operational Research}, 143(2):390--405,
  2002.

\bibitem{hettich1993semi}
R.~Hettich and K.O. Kortanek.
\newblock Semi-infinite programming: theory, methods, and applications.
\newblock {\em SIAM review}, 35(3):380--429, 1993.

\bibitem{jeroslow1971semi}
R.G. Jeroslow and K.O. Kortanek.
\newblock On semi-infinite systems of linear inequalities.
\newblock {\em Israel Journal of Mathematics}, 10(2):252--258, 1971.

\bibitem{karney81}
D.F. Karney.
\newblock Duality gaps in semi-infinite linear programming -- an approximation
  problem.
\newblock {\em Mathematical Programming}, 20:129--143, 1981.

\bibitem{kortanek1974classifying}
K.O. Kortanek.
\newblock Classifying convex extremum problems over linear topologies having
  separation properties.
\newblock {\em Journal of Mathematical Analysis and Applications},
  46(3):725--755, 1974.

\bibitem{lopez2012stability}
M.A. L{\'o}pez.
\newblock Stability in linear optimization and related topics. {A} personal
  tour.
\newblock {\em Top}, 20(2):217--244, 2012.

\bibitem{luenberger}
D.G. Luenberger.
\newblock {\em Optimization by Vector Space Methods}.
\newblock Wiley-Interscience, 1996.

\bibitem{motzkin36}
T.~S. Motzkin.
\newblock {\em Beitrage zur Theorie der Linearen Ungleichungen}.
\newblock PhD thesis, University of Besel, Jerusalem, 1936.

\bibitem{shapiro2005duality}
A.~Shapiro.
\newblock On duality theory of convex semi-infinite programming.
\newblock {\em Optimization}, 54(6):535--543, 2005.

\bibitem{shapiro2009semi}
A.~Shapiro.
\newblock Semi-infinite programming, duality, discretization and optimality
  conditions†.
\newblock {\em Optimization}, 58(2):133--161, 2009.

\bibitem{stein2003solving}
O.~Stein and G.~Still.
\newblock Solving semi-infinite optimization problems with interior point
  techniques.
\newblock {\em SIAM Journal on Control and Optimization}, 42(3):769--788, 2003.

\bibitem{williams86}
H.~P. Williams.
\newblock Fourier's method of linear programming and its dual.
\newblock {\em The American Mathematical Monthly}, 93:681--695, 1986.

\bibitem{zhang2010understanding}
Q.~Zhang.
\newblock Understanding linear semi-infinite programming via linear programming
  over cones.
\newblock {\em Optimization}, 59(8):1247--1258, 2010.

\end{thebibliography}

\appendix 
\section{Electronic Companion}

\subsection{Invariance of cleanliness under permutations}\label{ss:clean-invariance}

In this section of the Electronic Companion we provide a geometric interpretation of a {\it clean system}.  Recall a clean system is one where all of the variables are projected out, that is, there are no dirty variables.   The key results  are Theorem~\ref{theorem:invariance-of-cleanliness},   Theorem~\ref{theorem:clean-lineality-equivalence}, and Theorem~\ref{theorem:bounded-characterization}.  By Theorem~\ref{theorem:invariance-of-cleanliness},  if there is a variable permutation  that results in a clean system, then every variable permutation results in a clean system.  This is a very useful result.  It tells us that if  Fourier-Motzkin elimination  applied to (SILP)  results in a dirty variable, then there is no permutation that could ever make the elimination process find a clean system.  Hence there is no need to ever search for such a permutation, it does not exist.    Furthermore,  by Theorem~\ref{theorem:clean-lineality-equivalence}, if Fourier-Motzkin elimination does result in a clean system, under any permutation,  then we know the recession cone of the (SILP) feasible region is equal to the lineality space of the (SILP).   Hence dirty variables are always the result of the geometric property that the recession cone is not equal to the lineality space.   Finally, in Theorem~\ref{theorem:bounded-characterization} we give a necessary and sufficient condition for the Fourier-Motzkin elimination procedure to conclude that the feasible region of (SILP) is bounded.

\begin{theorem}\label{theorem:invariance-of-cleanliness}
 If there exists a permutation of the variables that results in a clean system using Fourier-Motzkin elimination, then every variable permutation results in a clean system.
\end{theorem}

\begin{proof}   By Proposition~\ref{prop:clean-permutation-invariant},
if there exists a permutation of the variables that results in a clean system when the Fourier-Motzkin procedure is applied, then every permutation of the variables results in a  clean system.
\end{proof}

Recall $\Gamma$ is the feasible region of the semi-infinite linear system  of (SILP).  The \emph{recession cone} of $\Gamma$ is denoted by $\rec(\Gamma)$  and \emph{lineality space} of $\Gamma$ is denoted by  $\lin(\Gamma) $,  respectively.

\begin{theorem}\label{theorem:clean-lineality-equivalence}
Every  permutation of the variables  results in a clean system using Fourier-Motzkin elimination if and only if  $\rec(\Gamma) = \lin(\Gamma)$.
\end{theorem}

\begin{proof} The logic is as follows.
\begin{itemize}

\item[1.]  By Definition~\ref{definition:conic-index-set},  there exists a conic index set for (SILP) if and only if $\rec(\Gamma) \neq \lin(\Gamma)$.

\item[2.]  By Proposition~\ref{prop:dirty-variable},  if  (SILP) contains a conic index set, then the Fourier Motzkin elimination procedure will terminate with at least one dirty variable regardless of the variable permutation used in the elimination procedure.  By  Corollary~\ref{corollary:dirty-is-conic} if there is a permutation of the variables that results in a dirty variable then (SILP) has a conic index set.  Hence (SILP) has a conic index set if and only if there is permutation of the variables that results in a dirty variable using Fourier-Motzkin elimination.

\item[3.]  By Theorem~\ref{theorem:invariance-of-cleanliness}   there is permutation of the variables that results in a dirty variable using Fourier-Motzkin elimination if and only if there is no permutation of the variable that results in a clean system.  Then by item 2.,  (SILP) contains a conic index set if and only if there is no permutation of the variable that results in a clean system.

\item[4.]  Items 1. and 3.  imply  there is no permutation of the variables that results in a clean system if and only if $\rec(\Gamma) \neq \lin(\Gamma)$.

\end{itemize}

The contrapositive of item 4. gives our result.
\end{proof}

\begin{theorem}\label{theorem:bounded-characterization}
If (SILP) is feasible, then the feasible region of (SILP) is bounded if and only if,   for every variable permutation, application of the Fourier-Motzkin elimination procedure (see Section~\ref{s:fm-elim}) to (SILP) results in both ${\cal H}_{+}(j)$ and ${\cal H}_{-}(j)$ nonempty at each iteration of Step  2b.   
\end{theorem}

\begin{proof}
Assume without loss the variable permutation is $\{ 1, 2, \ldots, n \}$ and that  at each iteration of step 2b of the Fourier-Motzkin elimination procedure, both ${\cal H}_{+}(j)$ and ${\cal H}_{-}(j)$ are not empty.  Show that this implies (SILP) is bounded.    Since both ${\cal H}_{+}(j)$ and ${\cal H}_{-}(j)$ are not empty
\begin{eqnarray*}
x_{j} &\ge&   \frac{\tilde b(p)}{\tilde a^{j}(p)} -\sum_{k = j +1}^{n} \frac{\tilde a^{k}(p)}{\tilde a^{j}(p)} x_{k},  \quad \forall p \in {\cal H}_{+}(j)   \\
x_{j}  &\le&  \frac{\tilde b(q)}{\tilde a^{j}(q)} -\sum_{k = j +1}^{n} \frac{\tilde a^{k}(q)}{\tilde a^{j}(q)} x_{k}, \quad \forall q \in {\cal H}_{-}(j).  
\end{eqnarray*}
Therefore $x_{j}$ has an upper bound  and a lower bound if the variables  $x_{j+1},  \ldots, x_{n}$ are bounded.  When $j = n$,
\begin{eqnarray*}
 x_{n} &\ge&  \sup \{  \frac{\tilde b(p)}{\tilde a^{j}(p)} \, : \, p \in {\cal H}_{+}(n) \} \\
  x_{n} &\le&  \inf \{  \frac{\tilde b(q)}{\tilde a^{j}(q)} \, : \, q \in {\cal H}_{-}(n) \}. 
\end{eqnarray*}
Therefore variable $x_{n}$ has a lower bound and an upper bound.  Then it follows from a simple recursive argument that variables $x_{n-1}, \ldots, x_{1}$ are bounded and the feasible region of (SILP) is bounded. 

\vskip 7pt

Now assume the feasible region of (SILP) is bounded. Then there cannot exists a conic index set nor a lineality index set.  Then by  Corollary~\ref{corollary:dirty-is-conic} there cannot be a dirty variable, i.e.  the case where  ${\cal H}_{+}(j)$ or ${\cal H}_{-}(j)$  is empty, but not both empty. By Corollary~\ref{corollary:zero-is-lineality}  there is never a variable  $j$ with both ${\cal H}_{+}(j)$ and ${\cal H}_{-}(j)$  empty. Then  at each iteration of step 2b of the Fourier-Motzkin elimination procedure, both ${\cal H}_{+}(j)$ and ${\cal H}_{-}(j)$ are not empty.
\end{proof}

The results used in the proofs of Theorem~\ref{theorem:invariance-of-cleanliness},   Theorem~\ref{theorem:clean-lineality-equivalence}, and  Theorem~\ref{theorem:bounded-characterization}   are in Section~\ref{s:appendix-permutation-independent}.  Basic definitions used in these theorems  are in  Section~\ref{s:appendix-basic-defintions}.

\subsubsection{Basic Definitions}\label{s:appendix-basic-defintions}

 \begin{definition}\label{definition:conic-index-set}
An index set $J_{C} =  \{ k_{1}, k_{2}, \ldots, k_{m} \}  \subseteq \{ 1, \ldots, n \}$ is a \emph{conic index set} if and only if there exist nonzero $\alpha_{k_{1}},  \alpha_{k_{2}},  \cdots \alpha_{k_{m}}$  such that for every feasible solution $(\overline{x}_{1}, \ldots, \overline{x}_{n})$  to (SILP),  the vector  $(\hat{x}_{1}, \ldots, \hat{x}_{n})$ defined by
\begin{eqnarray*}
 \hat{x}_{k} = \overline{x}_{k}, \quad \forall k \notin J_{C}, && \hat{x}_{k} = \overline{x}_{k} + r \alpha_{k}, \quad \forall k \in J_{C} 
\end{eqnarray*}
is feasible for all $r > 0$, but    the vector  $(\tilde{x}_{1}, \ldots, \tilde{x}_{n})$    defined by
\begin{eqnarray*}
 \tilde{x}_{k} = \overline{x}_{k}, \quad \forall k \notin J_{C}, && \tilde{x}_{k} = \overline{x}_{k} -   r \alpha_{k}, \quad \forall k \in J_{C}.
\end{eqnarray*}
is infeasible for a sufficiently large $r > 0$.
\end{definition}

\begin{remark}
If $J_{C} =  \{ k_{1}, k_{2}, \ldots, k_{m} \}  \subseteq \{ 1, \ldots, n \}$ is a \emph{conic index set} in Definition~\ref{definition:conic-index-set} then $y = (y_{1}, \ldots, y_{n})$ defined by
\begin{eqnarray*}
y_{k}  = 0, \quad \forall k \notin J_{C}, &&  y_{k} =  \alpha_{k}, \quad \forall k \in J_{C} 
\end{eqnarray*}
is an element of $ \rec(\Gamma)$ since for any feasible  $\overline{x}$,  $\overline{x} + r y \in \Gamma$ for all   $r > 0.$   However, for sufficiently large $r$,  $\overline{x} - r y \notin \Gamma$  so  $y \notin   \lin(\Gamma).$  Likewise each element in $ \rec(\Gamma) \backslash \lin(\Gamma)$ corresponds to a conic index set. \hfill $\triangleleft$
\end{remark}

\begin{definition}\label{definition:lineality-index-set}
An index set $J_{L} =  \{ k_{1}, k_{2}, \ldots, k_{m} \}  \subseteq \{ 1, \ldots, n \}$ is a \emph{lineality index set} if and only if there exist nonzero $\alpha_{k_{1}},  \alpha_{k_{2}},  \cdots \alpha_{k_{m}}$  such that for every feasible solution $(\overline{x}_{1}, \ldots, \overline{x}_{n})$  to (SILP),  the vector  $(\hat{x}_{1}, \ldots, \hat{x}_{n})$  defined by 
\begin{eqnarray*}
 \hat{x}_{k} = \overline{x}_{k}, \quad \forall k \notin J_{L}, && \hat{x}_{k} = \overline{x}_{k} + r \alpha_{k}, \quad \forall k \in J_{L}
 \end{eqnarray*}
 is feasible for all $r > 0$, and the vector $(\tilde{x}_{1}, \ldots, \tilde{x}_{n})$    
defined by
 \begin{eqnarray*}
  \tilde{x}_{k} = \overline{x}_{k}, \quad \forall  k \notin J_{L}, && \tilde{x}_{k} = \overline{x}_{k} -   r \alpha_{k}, \quad \forall k \in J_{L}.
\end{eqnarray*}
is also  feasible for all $r > 0$.
\end{definition}

\subsubsection{Clean Systems are Permutation Independent}\label{s:appendix-permutation-independent}

In this Section we assume  that $(SILP)$ is feasible.   Also assume  that  the FM procedure has eliminated variables $x_{1}, \ldots, x_{\ell}$ and the system of inequalities describing $P(\Gamma; x_{1}, \ldots, x_{\ell})$ is
\begin{eqnarray}
\tilde{a}^{\ell+1}(i) x_{\ell  + 1}+ \tilde{a}^{\ell + 2}(i) x_{\ell + 2} + \cdots + \tilde{a}^{n} x_{n} \ge \tilde{b}(i), \quad i \in \tilde{I}.  \label{eq:conic-variable-generation-1}
\end{eqnarray}
We use the notation  $J_{C}(\ell+1) \subseteq  \{ \ell  + 1, \ldots,   n \}$  to denote a conic index set with respect to the system~\eqref{eq:conic-variable-generation-1}.

\begin{remark}\label{remark:coefficient-nonnegativity}
 If $(\overline{x}_{1}, \ldots,  \overline{x}_{n})$ is a feasible solution to (SILP), then by Theorem 2.2, $(\overline{x}_{\ell + 1}, \ldots,  \overline{x}_{n})$ is a feasible solution to~\eqref{eq:conic-variable-generation-1}.  Then by definition of conic index set, for all $r > 0$, $(\hat{x}_{\ell + 1}, \ldots, \hat{x}_{n})$ is   also feasible to~\eqref{eq:conic-variable-generation-1}  where
\begin{eqnarray*}
 \hat{x}_{k} = \overline{x}_{k}, \quad \forall  k \notin J_{C}(\ell+1) \text{  and  } k > \ell,  \qquad \hat{x}_{k} = \overline{x}_{k} + r \alpha_{k}, \quad \forall k \in J_{C}(\ell+1).
\end{eqnarray*}
Since $(\hat{x}_{\ell + 1}, \ldots, \hat{x}_{n})$ is   also feasible to~\eqref{eq:conic-variable-generation-1}   for all $r > 0$  it follows that
\begin{eqnarray*}
\sum_{k \in J_{C}(\ell+1)}r \alpha_{k} \tilde{a}^{k}(i)   =   r \sum_{k \in J_{C}(\ell+1)} \alpha_{k} \tilde{a}^{k}(i)  \ge 0, \quad i \in \tilde{I}.
\end{eqnarray*}
  This implies 
\begin{eqnarray}
\sum_{k \in J_{C}(\ell+1)} \alpha_{k} \tilde{a}^{k}(i) \ge 0, \quad i \in \tilde{I}.   \label{eq:fm-consitency}
\end{eqnarray}  
\end{remark}

\begin{lemma}(Conic Index Set Extension)\label{lemma:conic-index-extension}
Assume  that   the Fourier-Motzkin  procedure has eliminated variables $x_{1}, \ldots, x_{\ell}$ producing the system~\eqref{eq:conic-variable-generation-1}  that describes $P(\Gamma; x_{1}, \ldots, x_{\ell})$.
If  $J_{C}(\ell+1) \subseteq  \{ \ell  + 1, \ldots,   n \}$  is a conic index set of $P(\Gamma; x_{1}, \ldots, x_{\ell})$, then there is a  conic index set $J_{C}(\ell)$ of $P(\Gamma; x_{1}, \ldots, x_{\ell -1})$ such that $J_{C}(\ell) = J_{C}(\ell+1) \cup \{ \ell \}$ or $J_{C}(\ell) = J_{C}(\ell+1) $.
\end{lemma}

\begin{proof} 
  By hypothesis, variable $\ell$ can be eliminated so ${\cal H}_{+}(\ell)$ is not empty and ${\cal H}_{-}(\ell)$ is not empty (if both ${\cal H}_{+}(\ell)$  and ${\cal H}_{-}(\ell)$ are empty we have a zero column and it follows immediately that $J_{C}(\ell) = J_{C}(\ell+1) $ is a conic index set for $P(\Gamma; x_{1}, \ldots, x_{\ell -1})$).  Assume prior to elimination variable $x_{\ell}$ the system is
\begin{eqnarray}
a^{\ell}(i) x_{\ell}+ a^{\ell + 1}(i) x_{\ell + 1} + \cdots + a^{n} x_{n} \ge b(i), \quad i \in I.  \label{eq:conic-variable-generation-2}
\end{eqnarray}
Now show there is a well-defined $\alpha_{\ell}$ so we extend the conic index set $J_{C}(\ell+1)$ to include variable $x_{\ell}$.  When projecting  out variable $x_{\ell}$ the $\tilde{a}(i)$ and $\tilde{b}(i)$  in~\eqref{eq:conic-variable-generation-1} are generated from the $a(i)$ and $b(i)$ in~\eqref{eq:conic-variable-generation-2}.  For the feasible $(\overline{x}_{1}, \ldots,  \overline{x}_{n})$,
\begin{eqnarray}
\sum_{k=\ell + 1}^{n} a^{\ell + 1}(i) \overline{x}_{k} &\ge& b(i), \quad i \in  {\cal H}_{0}(\ell).  \label{eq:conic-lemma-2-b} \\
\frac{b(p)}{a^{\ell}(p)} -\sum_{k = \ell +1}^{n} \frac{a^{k}(p)}{a^{\ell}(p)} \overline{x}_{k} &\le& \frac{\tilde{b}(q)}{a^{\ell}(q)} -\sum_{k = \ell +1}^{n} \frac{a^{k}(q)}{a^{\ell}(q)} \overline{x}_{k}, \nonumber\\ 
&& \forall p \in {\cal H}_{+}(\ell), \,\,  \forall q \in {\cal H}_{-}(\ell)
\end{eqnarray}
and
\begin{eqnarray}
\overline{x}_{\ell} &\ge&   \frac{b(p)}{a^{\ell}(p)} -\sum_{k = \ell +1}^{n} \frac{a^{k}(p)}{a^{\ell}(p)} \overline{x}_{k},  \quad \forall p \in {\cal H}_{+}(\ell)  \label{eq:project-ell-p-b} \\
\overline{x}_{\ell}  &\le&  \frac{b(q)}{a^{\ell}(q)} -\sum_{k = \ell +1}^{n} \frac{a^{k}(q)}{a^{\ell}(q)} \overline{x}_{k}, \quad \forall q \in {\cal H}_{-}(\ell).  \label{eq:project-ell-q-b}
\end{eqnarray}

If there exists an $\alpha_{\ell}$ that satisfies
\begin{eqnarray}
\alpha_{\ell} &\ge&  -\sum_{k \in J_{C}(\ell+1)} \alpha_{k} \frac{a^{k}(p)}{a^{\ell}(p)}, \quad \forall p \in {\cal H}_{+}(\ell)  \label{eq:def-alpha-ell-p-b} \\
\alpha_{\ell} &\le&   - \sum_{k \in J_{C}(\ell+1)} \alpha_{k} \frac{a^{k}(q)}{a^{\ell}(q)}, \quad \forall q \in {\cal H}_{-}(\ell)  \label{eq:def-alpha-ell-q-b}
\end{eqnarray}
then $r > 0$ gives 
\begin{eqnarray}
 r\alpha_{\ell} &\ge&  -r \sum_{k \in J_{C}(\ell+1)} \alpha_{k} \frac{a^{k}(p)}{a^{\ell}(p)}, \quad \forall p \in {\cal H}_{+}(\ell)   \label{eq:def-alpha-ell-p-r-b} \\
r\alpha_{\ell} &\le&   - r\sum_{k \in J_{C}(\ell+1)} \alpha_{k} \frac{a^{k}(q)}{a^{\ell}(q)}, \quad \forall q \in {\cal H}_{-}(\ell)  \label{eq:def-alpha-ell-q-r-b}.
\end{eqnarray}

\vskip 7pt

\noindent {\bf Claim:}  The  system is~\eqref{eq:def-alpha-ell-p-b}-\eqref{eq:def-alpha-ell-q-b} is consistent.  Multiply~\eqref{eq:def-alpha-ell-q-b} by -1 and apply Fourier-Motzkin elimination.   This yields~\eqref{eq:fm-consitency} and the fact that~\eqref{eq:fm-consitency} is nonnegative  implies~\eqref{eq:def-alpha-ell-p-b}-\eqref{eq:def-alpha-ell-q-b} is consistent.  $\dagger$
\vskip 7pt

Combining~\eqref{eq:def-alpha-ell-p-r-b}-\eqref{eq:def-alpha-ell-q-r-b} with~\eqref{eq:project-ell-p-b}-\eqref{eq:project-ell-q-b}
{\small
\begin{eqnarray}
\overline{x}_{\ell} +r \alpha_{\ell} &\ge&   \frac{b(p)}{a^{\ell}(p)} -\sum_{k = \ell +1}^{n} \frac{a^{k}(p)}{a^{\ell}(p)} \overline{x}_{k}  -r \sum_{k \in J_{C}(\ell+1)} \alpha_{k} \frac{a^{k}(p)}{a^{\ell}(p)}, \, \forall p \in {\cal H}_{+}(\ell)  \label{eq:p-constraint} \\
\overline{x}_{\ell} + r \alpha_{\ell} &\le&  \frac{b(q)}{a^{\ell}(q)} -\sum_{k = \ell +1}^{n} \frac{a^{k}(q)}{a^{\ell}(q)} \overline{x}_{k}  - r \sum_{k \in J_{C}(\ell+1)} \alpha_{k} \frac{a^{t}(q)}{a^{\ell}(q)}, \, \forall q \in {\cal H}_{-}(\ell)  \label{eq:q-constraint}
\end{eqnarray}
}

If there is an $\alpha_{\ell} = 0$ that is a solution to~\eqref{eq:def-alpha-ell-p-b}-\eqref{eq:def-alpha-ell-q-b} set $J_{C}(\ell) = J_{C}(\ell+1).$ Otherwise, if all solutions  to~\eqref{eq:def-alpha-ell-p-b}-\eqref{eq:def-alpha-ell-q-b} are nonzero, pick a nonzero $\alpha_{\ell}$  and  set $J_{C}(\ell) =  \{ \ell \} \cup J_{C}(\ell+1)$. In either case, $J_{C}(\ell)$ is a conic index set  for $P(\Gamma; x_{1}, \ldots, x_{\ell -1})$.

\end{proof}

\begin{example}[Example Illustrating Lemma~\ref{lemma:conic-index-extension}]
Consider the system
\begin{eqnarray*}
-\frac{2}{3} x_{1} - x_{2} &\ge& b_{1} \\
-\frac{1}{2} x_{1} - x_{2} &\ge& b_{2} \\
-x_{1} - x_{2} &\ge& b_{3} \\
x_{1} + 3 x_{2} &\ge& b_{4} 
\end{eqnarray*}
In reference back to Lemma~\ref{lemma:conic-index-extension},  $\ell = 1$.  Project out $x_{1}$ and get
\begin{eqnarray*}
\frac{3}{2}x_{2} &\ge& \frac{3}{2}b_{1}  + b_{4}\\
 x_{2} &\ge& 2 b_{2} + b_{4} \\
2 x_{2} &\ge& b_{3}  + b_{4}.
\end{eqnarray*}
Observe $J_{C}(2) = \{ 2 \}$ and the inequalities corresponding to~\eqref{eq:def-alpha-ell-p-b}-\eqref{eq:def-alpha-ell-q-b}  are
\begin{eqnarray*}
\alpha_{1} &\le& -\frac{3}{2} \\
\alpha_{1} &\le& -2 \\
\alpha_{1} &\le&  -1 \\
\alpha_{1} &\ge&  -3.
\end{eqnarray*}
Then $\{1, 2\}$ is a conic index set  and if  $(\overline{x}_{1}, \overline{x}_{2})$ is feasible then  $(\overline{x}_{1} + \alpha_{1}r, \overline{x}_{2} + r)$ is feasible for all $r > 0$ when $-3 \le \alpha_{1}  \le -2$ and   $\alpha_{2} = 1$.\hfill $\triangleleft$
\end{example}

\begin{prop}[Conic Index Set Extension]\label{prop:conic-index-extension}
Assume  that $(SILP)$ is feasible and  the Fourier-Motzkin procedure has eliminated variables $x_{1}, \ldots, x_{\ell}$ and the system of inequalities describing $P(\Gamma; x_{1}, \ldots, x_{\ell})$ is~\eqref{eq:conic-variable-generation-1}.
If  $J_{C}(\ell+1) \subseteq  \{ \ell  + 1, \ldots,   n \}$  is a conic index set of $P(\Gamma; x_{1}, \ldots, x_{\ell})$, then there is a  conic index set $J_{C}$ of (SILP) such that $ J_{C}(\ell+1)  \subseteq J_{C}$. 
\end{prop}

\begin{proof}
Use Lemma~\ref{prop:conic-index-extension} and generate conic index set $J_{C}(\ell)$ on $P(\Gamma; x_{1}, \ldots, x_{\ell-1})$ from  conic index set $J_{C}(\ell+1)$ on $P(\Gamma; x_{1}, \ldots, x_{\ell})$.  Repeat this process using Lemma~\ref{prop:conic-index-extension} and successively generate conic index set $J_{C}(k)$ from $J_{C}(k +1)$  for $k = 1, \ldots, \ell$, terminating with $J_{C} = J_{C}(1)$ a conic index set of (SILP).
\end{proof}

\begin{corollary}(Dirty Variable in Conic Index Set)\label{corollary:dirty-is-conic}
Assume  that   the Fourier-Motzkin  procedure has eliminated variables $x_{1}, \ldots, x_{\ell}$ producing the system~\eqref{eq:conic-variable-generation-1}  that describes $P(\Gamma; x_{1}, \ldots, x_{\ell})$.
If variable $x_{t}$ for $t > \ell$ is dirty, then  there is a conic index set $J_{C}$ for (SILP) and $t \in J_{C}.$
\end{corollary}

\begin{proof}
If variable $x_{t}$ is dirty, set $J_{C}(\ell + 1)= \{ t \}$ and observe that $J_{C}(\ell + 1)$ is a conic  index set  for the projected space  $P(\Gamma; x_{1}, \ldots, x_{\ell})$.   Then by Proposition~\ref{prop:conic-index-extension}  there is a conic index set $J_{C}$ for (SILP) with $t \in J_{C}.$
\end{proof}

In what follows Lemma~\ref{lemma:lineality-index-extension} replicates Lemma~\ref{lemma:conic-index-extension} for lineality index sets instead of conic index sets.   Proposition~\ref{prop:lineality-index-extension} replicates Proposition~\ref{prop:conic-index-extension} for lineality index sets instead of conic index sets.    Corollary~\ref{corollary:zero-is-lineality} replicates Corollary~\ref{corollary:dirty-is-conic} for a  variable in the projected system  with all zero coefficients instead of all nonnegative or all  nonpositive coefficients.

\begin{lemma}(Lineality Index Set Extension)\label{lemma:lineality-index-extension}
Assume  that   the Fourier-Motzkin  procedure has eliminated variables $x_{1}, \ldots, x_{\ell}$ producing the system~\eqref{eq:conic-variable-generation-1}  that describes $P(\Gamma; x_{1}, \ldots, x_{\ell})$.
If  $J_{L}(\ell+1) \subseteq  \{ \ell  + 1, \ldots,   n \}$  is a lineality index set of $P(\Gamma; x_{1}, \ldots, x_{\ell})$, then there is a  lineality index set $J_{L}(\ell)$ of $P(\Gamma; x_{1}, \ldots, x_{\ell -1})$ such that $J_{L}(\ell) = J_{L}(\ell+1) \cup \{ \ell \}$ or $J_{L}(\ell) = J_{L}(\ell+1) $.
\end{lemma}

\begin{proof}
Observe that in case of a lineality variable, instead of a conic variable, the system~\eqref{eq:fm-consitency}
\begin{eqnarray*}
\sum_{k \in J_{C}(\ell + 1)} \alpha_{k} \tilde{a}^{k}(i) \ge 0, \quad i \in \tilde{I}
\end{eqnarray*}
used in the proof of Lemma~\ref{lemma:conic-index-extension} becomes
\begin{eqnarray*}
\sum_{k \in J_{L}(\ell + 1)} \alpha_{k} \tilde{a}^{k}(i) = 0, \quad i \in \tilde{I}
\end{eqnarray*}
since for a lineality variable $r$ is both positive and negative.  The implication of equality is that when 
 replicating  the proof of Lemma~\ref{lemma:conic-index-extension} we can multiply~\eqref{eq:def-alpha-ell-p-b}-\eqref{eq:def-alpha-ell-q-b}  by positive and negative $r$ and still guarantee the existence of an $\alpha_{\ell}$ solution.  Since multiplying by both $r$ and $-r$ is valid when calculating $\alpha_{\ell}$ it follows that $J_{L}(\ell) = J_{L}(\ell +1)$ is a lineality index set.
\end{proof}

\begin{prop}(Lineality Index Set Extension)\label{prop:lineality-index-extension}
Assume  that   the Fourier-Motzkin  procedure has eliminated variables $x_{1}, \ldots, x_{\ell}$ producing the system~\eqref{eq:conic-variable-generation-1}  that describes $P(\Gamma; x_{1}, \ldots, x_{\ell})$.
If  $J_{L}(\ell+1) \subseteq  \{ \ell  + 1, \ldots,   n \}$  is a lineality  index set of $P(\Gamma; x_{1}, \ldots, x_{\ell})$, then there is a  lineality index set $J_{L}$ of (SILP) such that $ J_{L}(\ell+1)  \subseteq J_{L}$. 
\end{prop}

\begin{proof}
Replicate the proof of Proposition~\ref{prop:conic-index-extension}. 
\end{proof}

\begin{corollary}(A Zero Variable is in a Lineality Index Set)\label{corollary:zero-is-lineality}
Assume  that   the Fourier-Motzkin  procedure has eliminated variables $x_{1}, \ldots, x_{\ell}$ producing the system~\eqref{eq:conic-variable-generation-1}  that describes $P(\Gamma; x_{1}, \ldots, x_{\ell})$.
If $x_{t}$ for $t > \ell$ is a zero variable, i.e.  $\tilde{a}^{t}(i) = 0$ for all $i \in \tilde{I}$, then  there is a lineality index set $J_{L}$ for (SILP)  and $t \in J_{L}.$
\end{corollary}

\begin{proof}
Replicate the proof of Corollary~\ref{corollary:dirty-is-conic}. 
\end{proof}

The following lemma is critical in proving our main result.  The basic idea is that if  there is an index set $J$ that ``behaves'' like a conic index set, but  projection of the variables $x_{1}, \ldots, x_{\ell}$ results in the remaining variables in the index set having zero coefficients in the projected system, then $J$ must actually be a lineality index set. The proof relies heavily on the ideas used in the proof of Lemma~\ref{lemma:lineality-index-extension}.

\begin{lemma}(Lineality Implication)\label{lemma:lineality-implication}
Assume  that   the Fourier-Motzkin  procedure has eliminated variables $x_{1}, \ldots, x_{\ell}$ producing the system~\eqref{eq:conic-variable-generation-1}  that describes $P(\Gamma; x_{1}, \ldots, x_{\ell})$.
Further assume $J =  \{ k_{1}, \ldots,  k_{m} \} \subseteq \{1, \ldots, n\}$ is an index set with associated nonzero $\alpha_{k_{1}}, \ldots, \alpha_{k_{m}}$  such that 
 \begin{itemize}

\item[1.]  the set $J \cap \{ \ell+1, \ldots, n \}$ is not empty and $J \cap \{ \ell+1, \ldots, n \}$ is a lineality index set for $P(\Gamma; x_{1}, \ldots, x_{\ell})$ based on the nonzero $\alpha_{k_{m}}$ where $k_{m} > \ell$,  and

\item[2.] for every feasible solution $(\overline{x}_{1}, \ldots, \overline{x}_{n})$  to (SILP), $(\hat{x}_{1}, \ldots, \hat{x}_{n})$   where
\begin{eqnarray*}
 \hat{x}_{k} = \overline{x}_{k}, \quad \forall k \notin J, && \hat{x}_{k} = \overline{x}_{k} + r \alpha_{k}, \quad \forall k \in J 
\end{eqnarray*}
is feasible for all $r > 0$.

\end{itemize}
Then $J$  is a lineality index set  with associated nonzero $\alpha_{k_{1}}, \ldots, \alpha_{k_{m}}$.

\end{lemma}

\begin{proof}
The essence of the proof is to follow the proof of Proposition~\ref{prop:lineality-index-extension} and start with the lineality index set for $P(\Gamma; x_{1}, \ldots, x_{\ell})$ and recurse back and construct the entire set $J$.  However, for this to work, it is necessary to construct the given $\alpha_{k_{1}}, \ldots, \alpha_{k_{m}}$ in the hypothesis.   By hypothesis we know $J \cap \{ \ell+1, \ldots, n \}$ is a lineality index set for $P(\Gamma; x_{1}, \ldots, x_{\ell})$ based on the nonzero $\alpha_{k_{m}}$ where $k_{m} > \ell$. Hence we need to match the $\alpha_{k_{m}}$ for all $k_{m} \le \ell$. Consider an arbitrary $k_{h} \in J$ with $k_{h} \le \ell$. Assume at this point we have a match for the lineality set $J(k_{h} + 1)$.  It suffices to show
\begin{eqnarray}
\alpha_{k_{h}} &\ge&  -\sum_{k \in J(k_{h}+1)} \alpha_{k} \frac{a^{k}(p)}{a^{k_{h}}(p)}, \quad \forall p \in {\cal H}_{+}(k_{h})  \label{eq:def-alpha-ell-p-lin} \\
\alpha_{k_{h}} &\le&   - \sum_{k \in J(k_{h}+1)} \alpha_{k} \frac{a^{k}(q)}{a^{k_{h}}(q)}, \quad \forall q \in {\cal H}_{-}(k_{h})  \label{eq:def-alpha-ell-q-lin}
\end{eqnarray}
Assume without loss~\eqref{eq:def-alpha-ell-p-lin} is violated (a similar argument is valid if~\eqref{eq:def-alpha-ell-q-lin} is violated).  Then there is a $\hat{p} \in {\cal H}_{+}(k_{h})$ such that
\begin{eqnarray}
\alpha_{k_{h}} &<&  -\sum_{k \in J(k_{h}+1)} \alpha_{k} \frac{a^{k}(\hat{p})}{a^{k_{h}}(\hat{p})}.  \label{eq:phat-violation}
\end{eqnarray}
But 
{\small
\begin{eqnarray}
\overline{x}_{k_{h}} +r \alpha_{k_{h}} &\ge&   \frac{b(p)}{a^{k_{h}}(p)} -\sum_{k = k_{h} +1}^{n} \frac{a^{k}(p)}{a^{k_{h}}(p)} \overline{x}_{k}  -r \sum_{k \in J(k_{h}+1)} \alpha_{k} \frac{a^{k}(p)}{a^{k_{h}}(p)} \label{eq:phat-violation-b}
\end{eqnarray}
}
must hold for all $p \in {\cal H}_{+}(\ell)$ and feasible $\hat{x}$ by part 2. of the hypothesis for this lemma.  But~\eqref{eq:phat-violation}  implies for sufficiently large $r$ that~\eqref{eq:phat-violation-b} will be violated.  Hence it is possible in the backward recursion to generate the exact $\alpha_{k}$.  By a similar argument, $j \notin J$ implies that $\alpha_{j} = 0$ must be in the interval defined by~\eqref{eq:def-alpha-ell-p-lin}-\eqref{eq:def-alpha-ell-q-lin}.  

Also, as in the proof of Lemma~\ref{lemma:lineality-index-extension}, if  $\alpha_{k_{h}}$ satisfies~\eqref{eq:def-alpha-ell-p-lin}-\eqref{eq:def-alpha-ell-q-lin}, then $\alpha_{k_{h}}$ satisfies
\begin{eqnarray*}
\alpha_{k_{h}} &\le&  -\sum_{k \in J(k_{h}+1)} \alpha_{k} \frac{a^{k}(p)}{a^{k_{h}}(p)}, \quad \forall p \in {\cal H}_{+}(k_{h})   \\
\alpha_{k_{h}} &\ge&   - \sum_{k \in J(k_{h}+1)} \alpha_{k} \frac{a^{k}(q)}{a^{k_{h}}(q)}, \quad \forall q \in {\cal H}_{-}(k_{h})  
\end{eqnarray*}
since $J(k_{h}+1)$ is a lineality set and this implies
\begin{eqnarray*}
\sum_{k \in J(k_{h} + 1)} \alpha_{k} \tilde{a}^{k}(i) = 0, \quad i \in \tilde{I}
\end{eqnarray*}
and adding variable $k_{h}$ to $J(k_{h} + 1)$ results in a new lineality index set. 

Finally, given that the recursion began with an index set of lineality variables, this is maintained at each step and the lemma is proved.
\end{proof}

\begin{prop}\label{prop:dirty-variable}
If (SILP) contains a conic index set, then the Fourier-Motzkin elimination procedure will terminate with at least one dirty variable regardless of the variable permutation used in the elimination procedure. 
\end{prop}
\begin{proof}
Assume an arbitrary variable permutation.  By hypothesis there is a conic   index set $J_{C} = \{ j_{1}, \ldots,   j_{m} \}$.  The order of the index set is irrelevant, so  assume without loss  that for the given variable  permutation    $j_{1} < j_{2} < \cdots < j_{m}$.   Apply Fourier-Motzkin elimination  and attempt to eliminate the variables (reindexed to reflect the selected variable permutation)  $1, \ldots, \ell$  where $\ell =  j_{m-1} $.   If a dirty variable is discovered prior to eliminating variable $x_{\ell}$, we are done, there exists a dirty variable. Therefore,  we can assume variables    $1, \ldots, \ell$  where $\ell =  j_{m-1} $ are eliminated and the projected system is
\begin{eqnarray*}
\tilde{a}^{\ell+1}(i) x_{\ell  + 1}+ \tilde{a}^{\ell + 2}(i) x_{\ell + 2} + \cdots + \tilde{a}^{n} x_{n} \ge \tilde{b}(i), \quad i \in \tilde{I}
\end{eqnarray*}
where $j_{m} \ge \ell + 1$. There are two cases to consider. 

\vskip 7pt

\noindent {\bf Case 1:} Eliminating a variable in the index set $\{1, \ldots, \ell\}$  does not result in $j_{m}$ indexing a zero column in the projected system.    We show  variable  $j_{m}$ is dirty in the projected system.   This follows because  the projected system does not include variables $j_{1}, \ldots, j_{m-1}$.   Variable $j_{m}$ indexes the only variable in the conic index set $J_{C}$ that remains in the projected system   and the not all $\tilde{a}^{j_{m}}(i)$ are zero.   Remark~\ref{remark:coefficient-nonnegativity} implies 
$r \alpha_{ j_{m}} \tilde{a}^{ j_{m}}(i)     \ge 0$ for all  $i \in \tilde{I}.$
If $\alpha_{ j_{m}} > 0$ then $\tilde{a}^{ j_{m}}(i) \ge 0$   for all  $i \in \tilde{I}$.    If $\alpha_{ j_{m}} < 0$ then $\tilde{a}^{ j_{m}}(i) \le 0$  for all  $i \in \tilde{I}$.  In either case, column $j_{m}$ is dirty since not all coefficients are equal to zero.

\vskip 7pt

\noindent {\bf Case 2:}  Eliminating a variable in the index set $\{1, \ldots, \ell\}$  does  result in $j_{m}$ indexing a zero column in the projected system.     Then set    $  J \cap \{ \ell + 1, \ldots, n \}  =\{ j_{m} \}$ is a lineality index set in $P(\Gamma; x_{1}, \ldots, x_{\ell})$. Then by Lemma~\ref{lemma:lineality-implication}  the index set $J_{C}$ must  be lineality index set which contradicts the hypotheses.  Therefore Case 2 cannot occur.
\end{proof}

\begin{prop}\label{prop:clean-permutation-invariant}
If there exists a permutation of the variables that results in a clean system when the Fourier-Motzkin procedure is applied to the variables, then every permutation of the variables results in a  clean system.
\end{prop}

\begin{proof}
If the Fourier-Motzkin  procedure for some permutation results in a clean  system then there are no conic index sets by the contrapositive of Proposition~\ref{prop:dirty-variable}.  If there are no conic index sets, then applying the Fourier-Motzkin  procedure to any  permutation of the variables  cannot result in dirty variables since a dirty variable  implies the existence of a conic index set  by Corollary~\ref{corollary:dirty-is-conic}.
\end{proof}

\subsection{Differences between semi-infinite linear programming and finite linear programming}\label{ss:silp-finite-lp}

In this section we show how the Fourier-Motzkin elimination procedure can be used to reveal important differences between a finite linear programs and a semi-infinite linear program. Consider the following well-known facts about finite linear programs:
\begin{enumerate}[(i)]
\item if the primal is infeasible then the dual must be either infeasible or unbounded, and
\item if the primal has a finite optimal objective value, then the dual must be feasible and bounded with the same objective value (that is, strong duality always holds).
\end{enumerate}
The following two examples demonstrate that (i) and (ii) need not hold for general semi-infinite linear programs. 
\begin{example}\label{example:primal-infeasible-dual-solvable}
Consider the following problem
\begin{eqnarray*}
\inf x_1 \phantom{+ x_2} && \\
\phantom{\inf x_1 + } \frac{1}{i}x_2&\ge& 1  \quad \text{ for $i = 1, 2, \dots$} \\
              x_1 \phantom{+x_2} &\ge& 0.
\end{eqnarray*}
This problem is infeasible since for any  $x_2$, there exists a sufficiently large $i$ such that $\frac{1}{i}x_2 < 1$. 

Add the constraint $ - x_1 + z \ge 0,$ apply Fourier-Motzkin elimination and project out the clean variable $x_1$ to get the  following system.
\begin{equation*} 
\begin{array}{rcl}
\dfrac{1}{i}x_2 &\ge& 1  \quad \text{ for $i = 1, 2, \dots$} \\
\phantom{\frac{1}{i}x_1}  z &\ge&  0,
\end{array}
\end{equation*}
In this system $I_1 = I_4 = \emptyset$ and $I_3$ is a singleton. This implies that $\sup_{h \in I_3} \tilde b(h)$ is attained and  the dual is solvable. \hfill $\triangleleft$
\end{example}

\vskip 5pt
\begin{example}[Example~\ref{example:not-primal-optimal} revisited]\label{example:not-primal-optimal-revisited}

In Example~\ref{example:not-primal-optimal},  the primal problem has a finite optimal value of $0$. This optimal value remains greater than or equal to zero even without the non-negativity constraint on $x_1$ in \eqref{eq:not-primal-optimal}. This is because  $\omega(\delta)$ still equals  $\frac{1}{4(\delta - 1)}$  and   $\lim_{\delta \to \infty} \omega(\delta) = 0$.  Then by Lemma~\ref{lemma:primal-optimal-value}, the optimal primal value is greater than or equal to zero. However,  the finite support dual of this semi-infinite  linear program is infeasible. The objective coefficient of $x_2$ in the primal is $0$ and the coefficient of $x_2$ is strictly positive in the constraints.   This implies that  the only possible dual element satisfying the dual constraint corresponding to $x_2$ is $u = 0$; however, the objective coefficient of $x_1$ is $1$ and this dual vector does not satisfy the dual constraint corresponding to $x_1$. Alternatively, the infeasibility of the dual follows from Theorem~\ref{theorem:dual-infeasibility} because in this case $I_3 = \emptyset$. \hfill $\triangleleft$
\end{example}

\subsection{Application: Conic linear programs}\label{s:application-conic}

We consider conic programming with the following primal problem:
\begin{align*}\label{eq:conlp-primal}
\begin{array}{rl}
\quad \inf_{x\in X} & \langle x, \phi  \rangle \\
\textrm{s.t.} & A(x) \succeq_P d
\end{array}\tag{\text{ConLP}}
\end{align*}
where $X$ is a finite dimensional vector space and $Y$ an arbitrary vector space, $A : X \rightarrow Y$ is a linear mapping, $d \in Y$, $P$ is a pointed convex cone in $Y$ and $\phi$ is a linear functional on $X$. The standard dual (also a conic program) is
\begin{align*}\label{eq:conlp-dual}
\begin{array}{rl}
\sup_{\psi \in Y'} & \langle d, \psi  \rangle \\
\textrm{s.t.} & A'(\psi) = \phi \\
& \psi \in P'.
\end{array}\tag{\text{ConLPD}}
\end{align*}
In this section of the electronic companion, we study a semi-infinite linear program that is equivalent to (ConLP) and use the method of projection to give a new proof of the following well-known duality result for conic programs. 
\theoremstyle{definition}
\newtheorem*{thm:slater}{Theorem~\ref{thm:slater}}
\begin{thm:slater}[Zero duality gap via an interior point]
Let $Y$ be finite-dimensional, and let $P$ be reflexive. Assume the primal conic program~\eqref{eq:conlp-primal} is feasible. Suppose there exists a $\psi^* \in \intr(P')$ with $A'(\psi^*) = \phi$. Then the primal-dual pair~\eqref{eq:conlp-primal}-\eqref{eq:conlp-dual} has a zero duality gap. Moreover, the primal is solvable.
\end{thm:slater}
Our proof uses the interior point $\psi^*$ to construct a set of  constraints that show the associated semi-infinite linear program is tidy. Thus,  zero duality gap and primal solvability are established in a transparent ``algebraic" manner. We believe our results add fresh insight to the literature on connections between conic programming and semi-infinite linear programming (see, for instance Zhang \cite{zhang2010understanding}).

\subsubsection{Preliminaries}

For the linear map $A$ defined on $X$, let $\ker(A) = \{x \in X : A(x) = 0\}$ denote the kernel of $T$. The \emph{algebraic adjoint} $A' : Y' \to X'$ of $A$, where $Y'$ and $X'$ are the algebraic dual vector spaces of $X$ and $Y$ respectively, is the mapping $A'(\psi) = \psi \circ A$ (the algebraic adjoint has been defined elsewhere, see for instance Chapter 6 of \cite{hitchhiker}).

The following results demonstrate how, without loss of generality, we may assume $A'$ is surjective when \eqref{eq:conlp-primal} is feasible and bounded. 

\begin{lemma}\label{lem:A'-surjective}
Given a linear mapping  $A:X \rightarrow Y$,  $\ker(A) = \{0\}$ if and only if $A'$ is surjective.
\end{lemma}
\begin{proof}
($\Longrightarrow$)  If  $\ker(A) = \{0\}$, then $A$ is one-to-one and there is  a linear map $A^{-1} : Im(A) \to X$.  Let $\phi$ be an arbitrary linear functional in  $X^{\prime}.$  We show there exists a linear functional $\psi \in Y^{\prime}$  such that   $\phi = A^{\prime}(\psi).$    Define the linear functional $\phi \circ A^{-1}$  on $Im(A)$ and let  $\psi$ be any extension of this linear functional  from  $Im(A)$ to $Y$. Thus $\psi \in Y'$. We now show $\phi = A'(\psi)$. For any $x \in X$, $\langle x, A'(\psi)\rangle = \langle A(x), \psi \rangle = (\phi \circ A^{-1} )(A(x)) = \phi(x) = \langle x, \phi \rangle$. The second equality follows since $A(x) \in Im(A)$.

($\Longleftarrow$) Consider $x \in X$ such that $A(x) = 0$. Note that for every $\phi \in X'$, $\langle x, \phi \rangle = 0$. This would imply that $x = 0$. Since $A'$ is surjective, for every $\phi \in X'$ there exists $\psi \in Y'$ such that $A'(\psi) = \phi$. Thus, $\langle x, \phi \rangle = \langle x, A'(\psi) \rangle = \langle A(x), \psi \rangle = \langle 0, \psi \rangle = 0$. \end{proof}

\begin{remark}
Note that a related result to Lemma~\ref{lem:A'-surjective} for topological adjoints is well-known in the functional analysis literature (see for instance Theorem 2 on page 156 of \cite{luenberger}). This more familiar result requires that $Im(A)$ be a closed set in $Y$. This requirement does not fit our setting since we assume no topology on $Y$. In contrast, for algebraic adjoints, no assumption on $Im(A)$ is necessary. Indeed, \emph{any} extension of the linear functional in the forward direction of the above proof suffices, it need not be continuous in a given topology.
\hfill $\triangleleft$
\end{remark}

\begin{lemma}\label{lem:kernel} If \eqref{eq:conlp-primal} is feasible and bounded, then $\ker(A) \subseteq \ker(\phi)$. 
\end{lemma}
\begin{proof}
Prove the contrapositive and assume  that there is an  $r \in \ker(A) \setminus \ker(\phi)$. Without loss of generality assume $\langle r, \phi \rangle < 0$ (otherwise  make the argument with $-r$). Let $\bar x$ be a feasible solution to \eqref{eq:conlp-primal}, i.e., $A(\bar x) \succeq_P d$. Since $r \in \ker(A)$,  $A(\bar x + \lambda r) \succeq_P d$ for all $\lambda \geq 0$. But since $\langle r, \phi \rangle < 0$, $\langle \bar x + \lambda r, \phi \rangle \to -\infty$ as $\lambda \to \infty$, contradicting the boundedness of \eqref{eq:conlp-primal}.\end{proof}

\begin{lemma}\label{lem:restriction-problem}
Let $X$ be a finite-dimensional space, so that orthogonal complements of subspaces are well-defined. Let $\bar\phi = \phi|_{\ker(A)^\perp}$ be the linear functional  on $\ker(A)^\perp$ defined by the restriction of $\phi$ to $\ker(A)^\perp$. Similarly, let $\bar A = A|_{\ker(A)^\perp}$ denote the restriction of the linear map $A$. Consider the optimization problem 

\begin{equation}\label{eq:conlp-primal-lineality}
\begin{array}{rl}
\quad \inf_{x\in \ker(A)^\perp} & \langle x, \bar\phi  \rangle \\
\textrm{s.t.} & \bar A(x) \succeq_P d.
\end{array}
\end{equation}
If \eqref{eq:conlp-primal} is feasible and bounded, the optimal value of \eqref{eq:conlp-primal} equals the optimal value of~\eqref{eq:conlp-primal-lineality}. Moreover, if $O$ is the set of optimal primal solutions for \eqref{eq:conlp-primal}, and $\overline O$ is the set of optimal primal solutions for~\eqref{eq:conlp-primal-lineality}, then $O = \overline O + \ker(A)$. 
\end{lemma}

\begin{proof}
Since \eqref{eq:conlp-primal} is feasible and bounded,  $\ker(A) \subseteq \ker(\phi)$ by Lemma~\ref{lem:kernel}. For any $x$ feasible to \eqref{eq:conlp-primal}, let $r \in \ker(A)$ and $\bar x \in \ker(A)^\perp$ such that $x = \bar x + r$. Since $\ker(A) \subseteq \ker(\phi)$, $\langle r, \phi \rangle = 0$. Thus, $\langle x, \phi\rangle = \langle \bar x + r, \phi \rangle = \langle \bar x, \phi \rangle = \langle \bar x, \bar \phi \rangle$, the last equality follows since $\bar x \in \ker(A)^\perp$. Similarly, $\bar A(\bar x) = A(\bar x) = A(\bar x + r) = A(x) \succeq_P d$. Thus, $\bar x$ is a feasible solution to~\eqref{eq:conlp-primal-lineality} with the same objective value as $\langle x, \phi \rangle$. 
\end{proof}

\begin{remark}\label{rem:surjective-assum}
By Lemma~\ref{lem:restriction-problem}, when \eqref{eq:conlp-primal} is feasible and bounded, it suffices to consider a restricted optimization problem like~\eqref{eq:conlp-primal-lineality}. Note that $\ker(\bar A) = \{0\}$. Thus, without loss of generality,  it is valid to assume that for an instance of a feasible and bounded \eqref{eq:conlp-primal} in a finite-dimensional space $X$, the linear map $A$ has zero kernel, i.e., it is one-to-one. This implies that $A'$ is surjective by Lemma~\ref{lem:A'-surjective}. \hfill$\triangleleft$
\end{remark}

Let $F = \{x \in X\st A(x) \succeq_P d\}$ denote the feasible region of \eqref{eq:conlp-primal}. In our development, it is convenient to assume that the algebraic adjoint $A'$ of the linear map $A$ is surjective. As discussed above (Lemmas~\ref{lem:A'-surjective}--\ref{lem:restriction-problem} and Remark~\ref{rem:surjective-assum}) this can be assumed without loss of generality when \eqref{eq:conlp-primal} is feasible and bounded. 

Construct the following primal-dual pair of semi-infinite linear programs in the case where $X$ is finite-dimensional and the cone $P$ is reflexive.  Recall that a cone $P$ is  {\em reflexive} if $P'' = P$ under the natural embedding of $Y \hookrightarrow Y''$. The condition that $P$ is reflexive naturally holds in many important special cases of conic programming. Once such case is when $Y$ is finite dimensional and $P$ is a closed, pointed cone in $Y$. Then $P$ is easily seen to be reflexive. This case includes   linear programming, semi-definite programming (SDPs) and copositive programming. The above reformulation as a semi-infinite linear program works for any such instance.

The primal semi-infinite linear program is
\begin{align*}\label{eq:con-silp}
\begin{array}{rl}
\quad \inf_{x \in \R^n}  & c^\top x  \\
\textrm{s.t.} & a^1(\psi)x_1 + a^2(\psi)x_2 + \cdots + a^n(\psi)x_n \geq   b(\psi) \quad \text{ for all } \psi \in P'
\end{array}\tag{\text{ConSILP}} 
\end{align*}
where $n=\dim(X)$, and we choose a  basis $e^1, \ldots, e^n \in X$ to view $X$ as isomorphic to $\R^n$, and  $c \in \R^n$ represents the linear functional $\phi \in X'$ (also using the isomorphism of $X'$ and $\R^n$).  In \eqref{eq:con-silp}, the elements $a^{j} \in \R^{P^{\prime}}$  $j = 1, \ldots, n$ and $b \in \R^{P^{\prime}},$ are defined by $a^j(\psi) := \langle A(e^j), \psi \rangle$ and $b(\psi) := \langle b, \psi \rangle$. The finite support dual of \eqref{eq:con-silp} is
\begin{align*}\label{eq:con-silp-lagrangian-dual-finite}
\begin{array}{rcl}
\sup \sum_{\psi \in P'} b(\psi) v(\psi) && \\
 \quad  {\rm s.t.} \quad \sum_{\psi \in P'} a^{k}(\psi) v(\psi) &=& c_k\quad \text{ for } k = 1, \ldots, n \\
v &\in& \R^{(P')}_+.
\end{array}\tag{\text{ConFDSILP}}
\end{align*}
The close connection of this primal-dual pair to the conic pair  \eqref{eq:conlp-primal}--\eqref{eq:conlp-dual} is shown in Theorem~\ref{theorem:new_express} and Theorem~\ref{theorem:new_express_dual} below. Theorem~\ref{theorem:new_express} shows that \eqref{eq:conlp-primal} and 
\eqref{eq:con-silp} are \emph{equivalent}, meaning their respective feasible sets are isomorphic under an isomorphism which preserves objective values. In particular, this means $v(\ref{eq:conlp-primal}) = v(\ref{eq:con-silp})$. Similarly, Theorem~\ref{theorem:new_express_dual} shows that \eqref{eq:conlp-dual} and \eqref{eq:con-silp-lagrangian-dual-finite} are equivalent. In particular this means, $v(\ref{eq:conlp-dual})=v(\ref{eq:con-silp-lagrangian-dual-finite})$.

\subsubsection{Equivalent semi-infinite linear programming formulations of conic programs}\label{sss:appendix-conic-programs}

%

The following two theorems show the equivalence of \eqref{eq:conlp-primal} and \eqref{eq:conlp-dual} with their semi-infinite programming formulations given in \eqref{eq:con-silp} and \eqref{eq:con-silp-lagrangian-dual-finite}.

\begin{theorem}[Primal correspondence]\label{theorem:new_express} 
Assume $P$ is reflexive and $X$ is finite-dimensional. Let $e^1, \ldots, e^n$ be the basis of $X$ used to define~\eqref{eq:con-silp} and ~\eqref{eq:con-silp-lagrangian-dual-finite}. Then, $v(\ref{eq:conlp-primal}) = v(\ref{eq:con-silp})$. Moreover, the set of feasible solutions to \eqref{eq:conlp-primal} is isomorphic to the set of feasible solutions to \eqref{eq:con-silp} under this basis.
\end{theorem}
\begin{proof}
Since $X$ is isomorphic to $\R^n$ with respect to the basis $e^1, \ldots, e^n$  and  $c \in \R^n$  represents the linear functional $\phi \in X'$ the objective functions of both problems are identical (under this isomorphism). The result  follows if the feasible regions of both problems are isomorphic under this same mapping. 
%

Let $F$ denote the feasible region of (ConLP) and $\hat F$ denote the feasible region of (ConSILP). We show $F$ is isomorphic to $\hat F$ under the basis $e^1, \ldots, e^n$. 
%
First we show that if $x = x_1e^1 + \ldots + x_ne^n \in F$ then $(x_1, \ldots, x_n) \in \hat F$.   If $x\in F$, then  $A(x) \succeq_P d$. Therefore, $A(x) - d \in P$ and so for all $\psi \in P'$, $ \langle (A(x) - d), \psi \rangle \geq 0$. Writing $A(x) = \sum_{j=1}^n x_j A(e^j)$ and using the linearity of $\psi$, it follows  that $(x_1, \ldots, x_n)\in \hat F$.

Next we show that if $(x_1, \ldots, x_n) \in \hat F$, then $x = x_1e^1 + \ldots + x^ne^n \in F$.  We establish the contrapositive, i.e. if $x \not\in F$ then $(x_1, \ldots, x_n) \not\in \hat F$. If $x \not\in F$,  then $A(x) - d \not\in P$ and since $P$ is reflexive, $A(x)-d \not\in P''$ (under the natural embedding of $Y \hookrightarrow Y''$). Therefore, there exists $\psi  \in P'$ such that $\langle (A(x)-d), \psi \rangle < 0$. Again, using the linearity of $\psi$  it follows  that $(x_1, \ldots, x_n) \not\in \hat F$.
\end{proof}

\begin{theorem}[Dual Correspondence]\label{theorem:new_express_dual}
 Assume $P$ is reflexive and $X$ is finite-dimensional. Let $e^1, \ldots, e^n$ be the basis of $X$ used to define~\eqref{eq:con-silp} and ~\eqref{eq:con-silp-lagrangian-dual-finite}. Then, $v(\ref{eq:conlp-dual}) = v(\ref{eq:con-silp-lagrangian-dual-finite})$. Moreover, there exists maps $T : P' \to \R^{(P')}_+$ and $\hat T : \R^{(P')}_+ \to P'$ such that if $\psi \in P'$ is a feasible solution to \eqref{eq:conlp-dual} then $T(\psi)$ is a feasible solution to \eqref{eq:con-silp-lagrangian-dual-finite}.  Conversely, if $v \in \R^{(P')}$ is a feasible solution to \eqref{eq:con-silp-lagrangian-dual-finite} then $\hat T(v)$ is a feasible solution to \eqref{eq:conlp-dual}.
\end{theorem}
\begin{proof}
It suffices to construct maps $T$ and $\hat T$ which satisfy the following properties.
\begin{enumerate}[(i)]
\item $\langle e^k, A'(\psi^*)\rangle = \sum_{\psi \in P'} a^{k}(\psi) T(\psi^*)(\psi)$, for every $\psi^* \in P'$ and all $k=1, \ldots, n$.
\item $\langle d, \psi^* \rangle = \sum_{\psi \in P'}b(\psi)T(\psi^*)(\psi)$, for every $\psi^* \in P'$.
\item $\sum_{\psi \in P'} a^{k}(\psi) v(\psi) = \langle e^k, A'(\hat T(v))\rangle$, for every $v \in \R^{(P')}_+$ and all $k=1, \ldots, n$.
\item $\sum_{\psi \in P'}b(\psi)v(\psi) = \langle d, \hat T(v) \rangle$, for every $v \in \R^{(P')}_+$.
\end{enumerate}
The map $T$ is defined as follows. For any $\psi^* \in P'$, $T(\psi^*)$ is the finite support element $v^* \in \R^{(P')}$ where the only non-zero component of $v^*$ is 1 and corresponds to $\psi^*$. For any $k \in \{1, \ldots, n\}$, $\sum_{\psi \in P'} a^{k}(\psi) v^*(\psi) = a^k(\psi^*) = \langle A(e^k), \psi^* \rangle = \langle e^k, A'(\psi^*) \rangle$ and  (i) is satisfied.  Also,  $\sum_{\psi \in P'}b(\psi)v^*(\psi) = b(\psi^*)=\langle d, \psi^* \rangle$ and  (ii) is satisfied.

The map $\hat T$ is defined as follows. For any $v^* \in \R^{(P')}$, $\hat T(v^*) = \sum_{\psi\in P'}v^*(\psi)\psi$ where the sum is well-defined because $v^*$ has finite support. Since $v^*$ has nonnegative entries, $\hat T(v^*) \in P'$. Now, $\sum_{\psi \in P'} a^k(\psi)v^*(\psi) = \sum_{\psi \in P'}\langle A(e^k), \psi \rangle v^*(\psi) = \langle A(e^k), \sum_{\psi \in P'} v^*(\psi)\psi\rangle = \langle A(e^k), \hat T(v^*) \rangle = \langle e^k, A'(\hat T(v^*)\rangle$ and (iii) is satisfied.  Also, $\sum_{\psi \in P'}b(\psi)v^*(\psi) = \sum_{\psi \in P'} \langle d, \psi\rangle v^*(\psi) = \langle d, \sum_{\psi \in P'}v^*(\psi)\psi \rangle = \langle d, \hat T(v^*)\rangle$ and (iv) is satisfied.
\end{proof}

\subsubsection{Zero duality gap via boundedness}

This result is known in the literature (see for instance Shapiro~\cite{shapiro2005duality}), but we show it as an immediate consequence of Theorem~\ref{theorem:bounded-zero-duality-gap} based on Fourier-Motzkin elimination techniques.

\begin{theorem}[Zero duality gap via boundedness]\label{thm:tbounded}

If $P$ is reflexive and there exists a scalar $\gamma$ such the set $\left\{x : A(x) \succeq_P d \text{ and } \langle x, \phi  \rangle \le \gamma\right\}$ is nonempty and bounded, then there is no duality gap for the primal-dual pair~\eqref{eq:conlp-primal}-\eqref{eq:conlp-dual}.
\end{theorem}

\begin{remark}
The above result shows that semi-definite programs (SDPs) and copositive programs with nonempty, bounded feasible regions have zero duality gap.
\end{remark}

\subsubsection{Regular duality for conic programs}

We now prove a central result of conic programming, known as {\em regular duality}, using the machinery of FM elimination. First, some notions from conic programming (see Chapter 4 of Gartner and Matous\'ek~\cite{gartner-matousek} for more details). A sequence $(\psi^m)_{m\in \N}$, is called a feasible sequence for the dual program~\eqref{eq:conlp-dual} if $\psi^m \in P'$ for all $m \in \N$ and $\lim_{m \to \infty} A'(\psi^m) = \phi.$

The {\em value of a feasible sequence} $(\psi^m)_{m\in \N}$ is  $\langle d, (\psi^m)_{m\in \N} \rangle = \lim\sup_{m \to \infty}\langle d, \psi^m\rangle$. The {\em limit value (a.k.a. subvalue)} of the dual program~\eqref{eq:conlp-dual} is    $$\sup\{\langle d, (\psi^m)_{m \in \N}\rangle \st (\psi^m)_{m \in \N} \textrm{ is a feasible sequence for }\eqref{eq:conlp-dual} \}.$$

A simple  proof of regular duality for conic programs is easily obtained using projection (see Theorem 4.7.3 in~Gartner and Matousek~\cite{gartner-matousek} for the more standard proof technique).
\begin{theorem}[Regular duality for conic programs]\label{thm:regular-duality}

If the primal conic program~\eqref{eq:conlp-primal} is feasible and has a finite optimal value $z^*$,  then the dual program~\eqref{eq:conlp-dual} has a finite limit value $\hat d$ and $z^* = \hat d$.
\end{theorem}

\begin{proof}
By Theorem~\ref{theorem:new_express}, the optimal value of~\eqref{eq:con-silp} is equal to  $z^*$ and $z^*$ is finite since the optimal value of~\eqref{eq:conlp-primal} is finite.   By Theorem~\ref{theorem:regular-duality-silp}, the limit value of~\eqref{eq:con-silp-lagrangian-dual-finite} equals the optimal value of~\eqref{eq:con-silp}.   By Theorem~\ref{theorem:new_express_dual},  every feasible sequence 
for~\eqref{eq:conlp-dual} maps  to a feasible sequence 
for~\eqref{eq:con-silp-lagrangian-dual-finite}. Similarly, every feasible sequence 
for~\eqref{eq:con-silp-lagrangian-dual-finite} 
maps to a feasible sequence 
for~\eqref{eq:conlp-dual}. Thus, the limit value $\hat d$ of~\eqref{eq:conlp-dual} is equal  $z^{*}$,   the limit value of~\eqref{eq:con-silp-lagrangian-dual-finite}.
\end{proof}

\subsubsection{Zero duality gap via an interior point condition}

The main result of this section demonstrates how the Fourier-Motzkin elimination procedure can be used to establish a ``Slater-like" theorem for conic programs. The result is well known. Alternate proofs can be found in the conic programming literature (see for instance Chapter~4 of \cite{gartner-matousek}). The novelty here is the new proof using projection techniques.

For this section, we impose the condition that $Y$ is also finite-dimensional (along with $X$). As in the discussion after the definition of \eqref{eq:con-silp}, we identify $X$ and $X'$ with $\R^n$. Let $B(x, \epsilon) \subseteq \R^n$ denote the open ball of radius $\epsilon$ with center $x \in \R^n$.  Identify the objective linear functional $\phi \in X^{\prime}$ with the vector $c\in \R^n$.

\begin{lemma}\label{lem:open-mapping} Let $Y$ be finite-dimensional, and let $P$ be reflexive. Assume  $A' : Y' \to X'$ is surjective and there exists $\psi^* \in \intr(P')$ with $A'(\psi^*) = c$. Then there exists $\epsilon > 0$ and such that for all $\bar c \in B(c, \epsilon),$ there exists a  $\bar \psi \in P'$ such that  $\bar{c}^\top x \geq \langle d, \bar \psi \rangle$ is a constraint in~\eqref{eq:con-silp}.
\end{lemma}

\begin{proof}
For each $\psi \in P'$, the constraint in~\eqref{eq:con-silp} corresponding to $\psi$ is   $\sum_{j=1}^n x_j    \langle A(e^{j}), \psi  \rangle  \geq \langle d, \psi \rangle$. The left hand side of the inequality is the same as $\sum_{j=1}^n x_j    \langle e^{j}, A'(\psi)  \rangle = \langle x, A'(\psi)\rangle$. Since $A'$ is a linear map between finite-dimensional spaces, it is continuous and by assumption, surjective. By the Open Mapping theorem, $A'$ maps open sets to open sets. Since $\psi^* \in \intr(P')$ there exists an open ball $B^* \subseteq P'$ containing $\psi^*$. Thus, $A'(B^*)$ is an open set containing $c$. Therefore, there exists an $\epsilon > 0$ such that $B(c, \epsilon) \subseteq A'(B^*)$. Thus, for every $\bar c \in B(c, \epsilon)$, there exists $\bar\psi \in B^*$ such that $A'(\bar \psi) =\bar c$. Since all $\psi \in B^*\subseteq P'$ give constraints $\langle x,  A'(\psi)\rangle \geq \langle d, \psi\rangle$ in~\eqref{eq:con-silp}, for every $\bar c \in B(c, \epsilon)$ there is the  constraint $\bar c ^\top x = \langle x, A'(\bar\psi)\rangle \geq \langle d, \bar\psi\rangle$ in~\eqref{eq:con-silp}.
\end{proof}

\begin{theorem}[Zero duality gap via an interior point]\label{thm:slater} Let $Y$ be finite-dimensional, and let $P$ be reflexive. If  the primal conic program~\eqref{eq:conlp-primal} is feasible and  there exists $\psi^* \in \intr(P')$ with $A'(\psi^*) = \phi$, then there  is a  zero duality gap for the primal dual pair~\eqref{eq:conlp-primal}-\eqref{eq:conlp-dual}. Moreover, the primal is solvable.
\end{theorem}

\begin{proof}
By hypothesis,  there exists $\psi^* \in \intr(P')$ with $A'(\psi^*) = c$ so the dual conic program~\eqref{eq:conlp-dual} is feasible.  Since~\eqref{eq:conlp-primal} is also feasible by hypothesis, feasibility  of~\eqref{eq:conlp-dual} implies  \eqref{eq:conlp-primal}  is both feasible and  bounded. Then by Remark~\ref{rem:surjective-assum}, it is valid to assume $A'$ is surjective. 

\begin{claim} The variables $x_1, \ldots, x_n$ remain clean when  Fourier-Motzkin elimination is applied  to~\eqref{eq:con-silp}.
\end{claim}
\begin{proof}[Proof of Claim]\renewcommand{\qedsymbol}{} 
Since $A'(\psi^*) = c$, there is a constraint $c^\top x \geq \langle d, \psi^* \rangle$ in the system~\eqref{eq:con-silp}.   The constraint $-c^\top x + z \geq 0$ is also present when Fourier-Motzkin elimination is performed on a semi-infinite linear program. 
By Lemma~\ref{lem:open-mapping}, there exists $\epsilon > 0$ such that every $\bar c \in B(c, \epsilon)$ gives a constraint $\bar c^\top x \geq b'$ in~\eqref{eq:con-silp} where $b' = \langle d,\bar \psi\rangle$ with $A'(\bar\psi) = \bar c$. Thus, for any $\delta < \epsilon$, both $(c + \delta e^j)^\top x \geq  b^j_+$ and $(c - \delta e^j)^\top x \geq  b^j_-$ are constraints for every $j = 1, \ldots, n$, (where $b^j_+$ and $b^j_-$ are $\langle d,\psi^j_+\rangle$ and $\langle d,\psi^j_-\rangle$ respectively with $A'(\psi^j_+) = c + \delta e^j$ and $A'(\psi^j_-) = c - \delta e^j)$). \medskip

\noindent\underline{\em Case 1: $c_j = 0$ for all $j=1, \ldots, n.$}  In this case the constraints are  $\frac{\epsilon}{2}x_j \geq  b^j_+$ and $-\frac{\epsilon}{2}x_j \geq  b^j_-$ in the system. During Fourier-Motzkin, for each $j=1, \ldots, n$, the constraints $\frac{\epsilon}{2}x_j \geq  b^j_+$ and $-\frac{\epsilon}{2}x_j \geq  b^j_-$ remain in the system until variable $x_j$ is reached. This makes all variables $x_1, \ldots, x_n$ clean throughout the Fourier-Motzkin procedure.

\medskip

\noindent\underline{\em Case 2: $c_j \neq 0$ for some $j \in \{1, \ldots, n\}.$}  Relabel the variables such that $j = 1$ and  $c_1 \neq 0$. Note that coefficient of $x_1$ in $-c^\top x + z \geq 0$,  has opposite sign to the coefficient of $x_1$ in each pair of  constraints $(c + \frac{\epsilon}{2} e^k)^\top x \geq b^kj_+$ and $(c - \frac{\epsilon}{2} e^k)^\top x \geq b^k_-$ for $j =2, \ldots, n$. Clearly $x_1$ is clean, and when  $x_1$ is eliminated the constraint $-c^\top x + z \geq 0$  is aggregated with the constraints $(c + \frac{\epsilon}{2} e^k)^\top x \geq b^k_+$ and $(c - \frac{\epsilon}{2} e^k)^\top x \geq b^k_-$, for each $k = 2, \ldots, n$. This leaves the constraints $\frac{\epsilon}{2} x_k + z \geq  b^k_+$ and $-\frac{\epsilon}{2} x_j + z\geq  b^k_-$ in the system for $k =2, \ldots, n$, after $x_1$ is eliminated.  As in Case 1, these constraints remain in the system variable until $x_k$ is reached. This makes all variables $x_1, \ldots, x_n$ clean throughout the Fourier-Motzkin procedure.\quad $\dagger$
\end{proof}

Since variables $x_1, \dots, x_n$ are clean throughout the Fourier-Motzkin procedure, and~\eqref{eq:con-silp} is feasible (since \eqref{eq:conlp-primal} is feasible), the problem is feasible and tidy and  by Theorem~\ref{theorem:all-clean-system}, there is a zero duality gap between the pair~\eqref{eq:con-silp}-\eqref{eq:con-silp-lagrangian-dual-finite}, and~\eqref{eq:con-silp} is solvable. 
This implies that there is zero duality gap for the pair~\eqref{eq:conlp-primal}-\eqref{eq:conlp-dual}, and the primal \eqref{eq:conlp-primal} is solvable.\end{proof}


\begin{remark}
Since the dual conic program~\eqref{eq:conlp-dual} is also a conic program, one can consider~\eqref{eq:conlp-dual} as a primal conic program.  In this case  the dual is~\eqref{eq:conlp-primal}.  By Theorem~\ref{thm:slater}, there is a zero duality gap between this primal-dual pair if there is  a point $x^*$ such that $A(x^*) - d \in \intr(P)$. Moreover, the dual is solvable in this case. \hfill $\triangleleft$
\end{remark}

\subsection{Convex programs}\label{sss:appendix-convex}

\begin{theorem}\label{theorem:LD-SILP}
$v(\ref{eq:LD}) = v(\ref{eq:convex-silp})$. Moreover, \eqref{eq:convex-silp} is solvable if and only if there exists $\lambda^* \geq 0$ such that $L(\lambda^*) = \inf_{\lambda\geq 0}L(\lambda)$.
\end{theorem}
\begin{proof}
First we show $v(\ref{eq:LD}) \geq v(\ref{eq:convex-silp})$. If, for every $\lambda \geq 0$, $L(\lambda) = \infty$ then $v(\ref{eq:LD}) = \infty$ and the result is immediate. Else, consider any $\bar\lambda \geq 0$ such that $L(\bar\lambda) < \infty$. Set $\bar\sigma = L(\bar\lambda)$. Then $(\bar\sigma, \bar\lambda)$ is a feasible solution to \eqref{eq:convex-silp} with the same objective value as $L(\bar\lambda)$. Thus, $L(\bar\lambda)\geq v(\ref{eq:convex-silp})$. Since $\bar\lambda\geq 0$ was chosen arbitrarily,  $\inf_{\lambda\geq 0}L(\lambda) \geq v(\ref{eq:convex-silp})$. 

Now we show $v(\ref{eq:convex-silp}) \geq v(\ref{eq:LD})$. If~\eqref{eq:convex-silp} is infeasible then $v(\ref{eq:convex-silp}) = \infty$ and the result is immediate. Otherwise, consider any feasible solution $(\bar\sigma, \bar\lambda)$ to \eqref{eq:convex-silp}. Then $\bar\sigma \geq L(\bar\lambda)$ and thus $\bar\sigma \geq \inf_{\lambda \geq 0}L(\lambda)$. Since $\bar\sigma$ is the objective value of this feasible solution to \eqref{eq:convex-silp}, the optimal value of \eqref{eq:convex-silp} is greater than or equal to $\inf_{\lambda \geq 0}L(\lambda)$.

The second part follows from very similar arguments.
\end{proof}
\begin{theorem}\label{theorem:CP-FDSILP}
$v(\ref{eq:CP}) = v(\text{CP-FDSILP})$.
\end{theorem}
\begin{proof}
First we show $v(\ref{eq:CP}) \geq v(\text{CP-FDSILP})$. If~\eqref{eq:convex}-\eqref{eq:nonnegative} is infeasible, then $v(\text{CP-FDSILP}) = -\infty$ and the result is immediate. Assume~\eqref{eq:convex}-\eqref{eq:nonnegative}    has feasible solution $(\bar u, \bar v)$. Let $\bar x = \sum_{x\in \Omega} x\bar u(x)$.  This sum is well-defined because $\bar u$ has finite support. Note that $\bar x$ is feasible to \eqref{eq:CP}. First, since $\Omega$ is convex, by~\eqref{eq:convex} $\bar x \in \Omega$.   By~\eqref{eq:g},    $-\sum_{x\in \Omega}\bar u(x)g_i(x)  + \overline{v}_i =  0$ for all $i = 1, \ldots, p$. Since $\overline{v}_i \geq 0$,  $\sum_{x\in \Omega}\bar u(x)g_i(x) \geq 0$. By~\eqref{eq:convex} and concavity of $g_i$,   $g_i(\bar x) = g_i(\sum_{x\in \Omega} x\bar u(x)) \geq \sum_{x\in \Omega}\bar u(x)g_i(x) \geq 0$ for all $i=1, \ldots, p$. Thus the constraints of \eqref{eq:CP} are satisfied. Since $f$ is concave it follows that $f(\bar x) = f(\sum_{x\in \Omega} x\bar u(x)) \geq \sum_{x\in \Omega}\bar u(x)f(x)$ and $\sum_{x\in \Omega}\bar u(x)f(x)$ is the objective value of $(\bar u, \bar v)$ in \eqref{eq:convex-fdsilp}. This implies $v(\ref{eq:CP}) \geq v(\text{CP-FDSILP})$.

Now we show that $v(\text{CP-FDSILP}) \geq v(\ref{eq:CP})$. If~\eqref{eq:CP} is infeasible, then $v(\ref{eq:CP}) = -\infty$ and the result is immediate. Otherwise, consider any feasible solution $\bar x$ to~\eqref{eq:CP}. Let $\bar u \in \R_{+}^{(\Omega)}$ be defined by $\bar u(\bar x) = 1$ and $\bar u(x) = 0$ for all $x \neq \bar x$. Define $\bar v \in \R^p$ by $\bar v_i = g_i(\bar x)$. Since $\bar x$ is feasible to~\eqref{eq:CP}, $\bar v \in \R^p_+$. Thus, $(\bar u, \bar v)$ is a feasible solution to~\eqref{eq:convex-fdsilp}. The objective value of $(\bar u, \bar v)$ in \eqref{eq:convex-fdsilp} is $f(\bar x)$ which is the objective value $\bar x$ in \eqref{eq:CP}. 
\end{proof}

\subsection{Additional sufficient conditions for zero duality gap}\label{s:app-countable}

By looking at the recession cone of~\eqref{eq:bounded-zero-duality-gap} it is possible gain further insights and discover useful sufficient conditions for zero duality gaps in general semi-infinite linear programs.   We show  results first discovered by Karney~\cite{karney81}  follow directly and easily from our methods.   The recession  cone of~\eqref{eq:bounded-zero-duality-gap} is  defined by the system
\begin{eqnarray}
-  c_{1} x_{1} -  c_{2} x_{2} - \cdots - c_{n} x_{n}    &\ge&   0  \label{eq:bounded-recession-cone-1}  \\
  a^{1}(i) x_{1} + a^{2}(i) x_{2} + \cdots + a^{n}(i) x_{n} \phantom{ + z} &\ge& 0 \quad  \text{ for } i \in I.   \label{eq:bounded-recession-cone-2}
\end{eqnarray}
Applying Fourier-Motzkin elimination to~\eqref{eq:bounded-recession-cone-1}-\eqref{eq:bounded-recession-cone-2} gives
\begin{eqnarray}
0 &\ge&  0 \quad \text{ for } h\in H_1\label{eq:defineH1-recession} \\
\tilde{a}^{\ell}(h) x_{\ell} + \tilde{a}^{\ell+1}(h) x_{\ell+1} + \cdots + \tilde{a}^{n}(h) x_{n} &\ge& 0 \quad \text{ for } h \in H_{2}.  \label{eq:defineH2-recession}
\end{eqnarray}
Following the notation of Karney~\cite{karney81},  $K$ denotes the recession cone of \eqref{eq:SILP},  given by the inequalities~\eqref{eq:bounded-recession-cone-2} and $N$ denotes the null space of the objective function vector $c.$  
\old{
\begin{remark}\label{remark:nullspace_K_N}
 If  $K \cap N$ is a subspace,   then for any  $r \in K \cap N$,   $-r  \in (K \cap N)$.  Then both $r$ and $-r$ are in $K$.  This implies $r \in M.$ Therefore,  $K \cap N \subseteq M.$  This observation is used later in Theorem \ref{theorem:unbounded-recession-cone}. \hfill $\triangleleft$
 \end{remark} }

\begin{lemma}\label{lemma:strict-inequality}
If $H_2$ is nonempty in \eqref{eq:defineH2-recession}, then there exists a ray $r \in \R^n$ satisfying~\eqref{eq:bounded-recession-cone-1}-\eqref{eq:bounded-recession-cone-2} with at least one of the inequalities in~\eqref{eq:bounded-recession-cone-1}-\eqref{eq:bounded-recession-cone-2} satisfied strictly.
\end{lemma}
\begin{proof}
If  $H_2$ is nonempty, there is a  $k \ge \ell$ such that   $\tilde a^k(h)$ is nonzero for at least  one $h \in H_{2}$.   Since $x_{k}$ is a dirty variable,  the nonzero $\tilde a^k(h)$ are of the same sign for all $h \in H_{2}$. If the $\tilde a^k(h)$ are all nonnegative, then set $x_k = 1$ and $x_i = 0$ for $i \neq k$; if  the $\tilde a^\ell(h)$ are all nonpositive, then set $x_k = -1$ and $x_i = 0$ for $i \ne k$. This  solution to \eqref{eq:defineH1-recession}-\eqref{eq:defineH2-recession} satisfies at least one of the inequalities in \eqref{eq:defineH1-recession}-\eqref{eq:defineH2-recession} strictly. Since this is the projection of some $r$ satisfying~\eqref{eq:bounded-recession-cone-1}-\eqref{eq:bounded-recession-cone-2}, this $r$ must satisfy at least one inequality in~\eqref{eq:bounded-recession-cone-1}-\eqref{eq:bounded-recession-cone-2} strictly, since all inequalities in \eqref{eq:defineH1-recession}-\eqref{eq:defineH2-recession} are conic combinations of inequalities in~\eqref{eq:bounded-recession-cone-1}-\eqref{eq:bounded-recession-cone-2}.
\end{proof}
 \old{
 Next, for each $k = \ell, \ldots, n$, construct a nonzero solution of~\eqref{eq:bounded-recession-cone-1}-\eqref{eq:bounded-recession-cone-2} as follows. For a given $\hat{k},$ set  $x_{\hat{k}} = \delta > 0$ and $x_{k} = 0$ for $k \neq \hat{k}$  if the nonzero elements $\tilde{a}^{\hat{k}}(h)$  are positive and  $x_{\hat{k}} = \delta < 0$ and $x_{k} = 0$ for $k \neq \hat{k}$  if the nonzero elements $\tilde{a}^{\hat{k}}(h)$  are positive.      Complete the solution for $x_{1}, \ldots, x_{\ell -1}$ as follows. When Fourier-Motzkin elimination is applied to~\eqref{eq:bounded-recession-cone-1}-\eqref{eq:bounded-recession-cone-2} in matrix order, at step $\ell - 1$, there is a system of inequalities
\begin{eqnarray*}
\hat{a}^{\ell - 1}(i) x_{\ell-1}+ \sum_{k=\ell}^{n} \hat{a}^{k}(i) x_{k} \ge 0, \quad i \in I(\ell - 1),
\end{eqnarray*}
 where $I(\ell - 1)$ is an index set for all the constraints in the system after projecting out variables $x_{1}$ through $x_{\ell - 2}.$ Fixing  the $x_{k},$ $k = \ell, \ldots, n$ as described above gives
 \begin{eqnarray*}
\hat{a}^{\ell - 1}(i) x_{\ell-1}+ \hat{a}^{\hat{k}}(i)  \delta  \ge 0, \quad i \in I(\ell - 1).
\end{eqnarray*}
At this step, both ${\cal H}_{+}(\ell - 1)$ and ${\cal H}_{-}(\ell - 1)$ are nonempty. Then the system of inequalities at this step with the fixed $x_{\ell}, \ldots, x_{n}$ is
 \begin{eqnarray*}
 x_{\ell-1} &\ge& -  \hat{a}^{\hat{k}}(i) \delta / \hat{a}^{\ell - 1}(i) , \quad i \in {\cal H}_{+}(\ell - 1)  \\
 x_{\ell-1} &\le&  \hat{a}^{\hat{k}}(i) \delta / \hat{a}^{\ell - 1}(i) , \quad i \in {\cal H}_{-}(\ell - 1)
\end{eqnarray*}
Then all feasible values of  $x_{\ell-1}$ are given by (assuming  without loss $\delta > 0$)
\begin{eqnarray*}
\delta  \sup \{ -\hat{a}^{\hat{k}}(i)  / \hat{a}^{\ell - 1}(i)  \, : \, i \in {\cal H}_{+}(\ell - 1) \}    \le  x_{\ell-1} \le  \delta  \inf \{ \hat{a}^{\hat{k}}(i)  / \hat{a}^{\ell - 1}(i)  \, : \, i \in {\cal H}_{-}(\ell - 1) \},
\end{eqnarray*}
and a feasible value is, without loss,  $x_{\ell - 1} = d_{\ell - 1} \delta$ where 
\begin{eqnarray*}
d_{\ell - 1} =  \inf \{ \hat{a}^{\hat{k}}(i)  / \hat{a}^{\ell - 1}(i)  \, : \, i \in {\cal H}_{-}(\ell - 1) \}.
\end{eqnarray*}
Recursing back in this fashion to variable $x_{1}$ gives a nonzero solution vector $r$
\begin{eqnarray}
r = \delta (d_{1}, \ldots, d_{\ell - 1}, 0, \ldots, 1, \dots, 0).  \label{eq:define-extreme-ray}
\end{eqnarray}
where the 1 is in component $\hat{k}.$
 The rays defined  in~\eqref{eq:define-extreme-ray}   for  $k = \ell, \ldots, n$ are obviously nonzero  and feasible to~\eqref{eq:bounded-recession-cone-1}-\eqref{eq:bounded-recession-cone-2}    for all $\delta > 0$.  Indeed,  by construction,  the rays $r,$ are extreme rays of~\eqref{eq:bounded-recession-cone-1}-\eqref{eq:bounded-recession-cone-2}, but this fact is not needed in the results that follow.}

\old{\begin{corollary}\label{cor:bounded-recession-cone}
Assume that   $\gamma \in \R$ and there is a feasible solution value to \eqref{eq:SILP} that is less than or equal to $\gamma.$  If the zero vector  is the unique solution to the system~\eqref{eq:bounded-recession-cone-1}-\eqref{eq:bounded-recession-cone-2}, then \eqref{eq:SILP} is tidy, the primal is solvable, and  there is a zero duality gap for the primal-dual pair \eqref{eq:SILP} and \eqref{eq:FDSILP}.
\end{corollary}

\begin{corollary}\label{cor:bounded-recession-cone}
\eqref{eq:SILP} is feasible and tidy if and only if there exists a $\gamma \in \R$ such that \eqref{eq:bounded-zero-duality-gap} is feasible and the zero vector  is the unique solution to the system~\eqref{eq:bounded-recession-cone-1}-\eqref{eq:bounded-recession-cone-2}
In particular, the ``if" direction implies \eqref{eq:SILP} is solvable and there is zero duality gap.
\end{corollary}

\begin{proof}
We prove this by showing that the feasible region defined by  system~\eqref{eq:bounded-zero-duality-gap} is bounded if and only if the zero vector  is the unique solution to the system~\eqref{eq:bounded-recession-cone-1}-\eqref{eq:bounded-recession-cone-2}.  The result then follows from Theorem~\ref{theorem:bounded-zero-duality-gap}.

($\Longrightarrow$)  We show  that if  the feasible region defined by~\eqref{eq:bounded-zero-duality-gap} is bounded, then~\eqref{eq:bounded-recession-cone-1}-\eqref{eq:bounded-recession-cone-2}  has the zero vector as the unique solution.  Prove the contrapositive and assume $\overline{x}$ is a nonzero solution to~\eqref{eq:bounded-recession-cone-1}-\eqref{eq:bounded-recession-cone-2}.  By hypothesis,~\eqref{eq:bounded-zero-duality-gap} has a feasible solution $\hat{x}.$  Then $\hat{x} + \lambda \overline{x}$ is feasible to~\eqref{eq:bounded-zero-duality-gap} for all positive $\lambda$, so $\overline{x}$ nonzero implies   the feasible region defined by~\eqref{eq:bounded-zero-duality-gap} is unbounded.  

($\Longleftarrow$)  We show that if the zero vector  is the unique solution to the system~\eqref{eq:bounded-recession-cone-1}-\eqref{eq:bounded-recession-cone-2}, then the feasible region defined by system~\eqref{eq:bounded-zero-duality-gap} is bounded.  Prove the contrapositive and assume the feasible region defined by~\eqref{eq:bounded-zero-duality-gap} is not bounded. Note that this implies there is a nonzero solution to~\eqref{eq:bounded-recession-cone-1}-\eqref{eq:bounded-recession-cone-2}. 
If~\eqref{eq:bounded-zero-duality-gap} is  not bounded, then  by Theorem \ref{theorem:region-boundedness},  Fourier-Motzkin elimination applied to~\eqref{eq:bounded-zero-duality-gap}   gives system~\eqref{eq:defineI1}-\eqref{eq:defineI2}  with   $H_{2}$  not empty. The values of the right hand side  have no effect on the elimination process. Therefore, Fourier-Motzkin elimination applied to~\eqref{eq:bounded-recession-cone-1}-\eqref{eq:bounded-recession-cone-2}  also results in a nonempty $H_{2}$ in  system~\eqref{eq:defineH2-recession}.   Then  the extreme rays defined in~\eqref{eq:define-extreme-ray} are nonzero solutions to~\eqref{eq:bounded-recession-cone-1}-\eqref{eq:bounded-recession-cone-2}. 
\end{proof}

\begin{theorem}\label{theorem:unbounded-recession-cone}
Assume  \eqref{eq:SILP} is feasible and that applying Fourier-Motzkin elimination to~\eqref{eq:bounded-recession-cone-1}-\eqref{eq:bounded-recession-cone-2}  gives~\eqref{eq:defineH1-recession}-\eqref{eq:defineH2-recession}.  If $r$ is a ray defined by~\eqref{eq:define-extreme-ray}, and $r$ is not an element of  the null space $N,$ then $v(\ref{eq:SILP}) = v(\ref{eq:FDSILP})$.
\end{theorem}

\begin{proof}
Assume $r$ is a ray defined in~\eqref{eq:define-extreme-ray} and $r$ is not in $N,$  the null space of $c.$ Then $\langle r,  c\rangle  \neq 0.$ By construction, $r$ is a ray of~\eqref{eq:bounded-recession-cone-1}-\eqref{eq:bounded-recession-cone-2}  so $\langle r,  c\rangle  \neq 0$ and~\eqref{eq:bounded-recession-cone-1}  together imply  $\langle r,  c\rangle < 0.$  By hypothesis, \eqref{eq:SILP} has a feasible solution $\overline{x}.$ Then $\overline{x} + r$ is feasible for all $\delta > 0.$ This implies \eqref{eq:SILP} is unbounded so $v(\ref{eq:SILP}) = v(\ref{eq:FDSILP})=-\infty$ and there is a zero duality gap.
\end{proof}
}
\begin{theorem}\label{theorem:unbounded-recession-cone}
If \eqref{eq:SILP} is feasible and $K \cap N$ is a subspace, then  $v(\ref{eq:SILP}) = v(\ref{eq:FDSILP})$.
\end{theorem}

\begin{proof}
\underline{{\em Case 1: $H_{2}$ in~\eqref{eq:defineH2-recession} is empty.}} Observe that the columns in  systems~\eqref{eq:bounded-recession-cone-1}-\eqref{eq:bounded-recession-cone-2} and~\eqref{eq:initial-system-obj-con}-\eqref{eq:initial-system-con} are identical for variables $x_{1}, \ldots, x_{n}$. This means if $x_k$ is eliminated when  Fourier-Motzkin elimination  is applied to one system, it  is eliminated in the other system. Since $H_{2}$ in~\eqref{eq:defineH2-recession} is empty,  \eqref{eq:SILP} is tidy. Then by Theorem~\ref{theorem:all-clean-system},  $v(\ref{eq:SILP}) = v(\ref{eq:FDSILP})$.

\underline{{\em Case 2:  $H_{2}$ in~\eqref{eq:defineH2-recession} is not empty.}} If $H_{2}$ is not empty, by Lemma~\ref{lemma:strict-inequality}, there exists a $r$ satisfying~\eqref{eq:bounded-recession-cone-1}-\eqref{eq:bounded-recession-cone-2} such that at least one of the inequalities in~\eqref{eq:bounded-recession-cone-1}-\eqref{eq:bounded-recession-cone-2} is satisfied strictly. If $c^T r < 0$  and  $r \in K$, then $v(\ref{eq:SILP}) = -\infty$.   Therefore \eqref{eq:FDSILP} is infeasible by weak duality and $v(\ref{eq:SILP}) = v(\ref{eq:FDSILP}) = -\infty$. If  $c^T r = 0$ then the constraint~\eqref{eq:bounded-recession-cone-1} is tight at $r$ and so
$r\in N$. Then $r \in K \cap N$ which is a subspace by hypothesis.   Then  $-r \in K \cap N$.  This implies  $r \in K \cap -K$. But this means that $r$ satisfies all inequalities in~\eqref{eq:bounded-recession-cone-2} at equality and this contradicts the fact established for this case that at least one inequality in \eqref{eq:bounded-recession-cone-1}-\eqref{eq:bounded-recession-cone-2} is strict.
\old{
 it either contains an index $h$ such that $\supp(\overline{u}^{h})$ contains the index for row~\eqref{eq:bounded-recession-cone-1} or it does not contain such an $h$.  Assume no such $h$ index exits.  Fourier Motzkin   elimination applied to (SILP) generates exactly the same set of multiplier vectors as Fourier-Motzkin elimination applied to~\eqref{eq:bounded-recession-cone-1}-\eqref{eq:bounded-recession-cone-2}.  Therefore, if $H_{2}$  in~\eqref{eq:defineH2-recession} does not contain an index for a multiplier vector with support containing the objective function row,  $J_{4}$ must be empty  when Fourier-Motzkin elimination is applied to (SILP).   This  implies $\lim_{\delta \rightarrow \infty} \omega(\delta) = - \infty.$ Since (SILP) is feasible, by Theorem \ref{theorem:zero-duality-gap} there is a zero duality gap.

Now    assume that there is an element $h \in H_{2}$ such that $\overline{u}^{h}$ has a positive multiplier for the row~\eqref{eq:bounded-recession-cone-1}.     In the row of~\eqref{eq:defineH2-recession}  indexed by $h,$ at least one $\tilde{a}^{k}(h)$   for  $k = \ell, \ldots, n$ is nonzero.  Assume without loss $\tilde{a}^{n}(h)$ is nonzero.  Consider the extreme ray   $\overline{r} = \delta(d_{1}, \ldots, d_{\ell - 1}, 0, \ldots, 1)$ of~\eqref{eq:bounded-recession-cone-1}-\eqref{eq:bounded-recession-cone-2} defined in~\eqref{eq:define-extreme-ray} .  There are two cases to consider.  The ray $\overline{r}$ is in the lineality space of $K$ or it is not. First assume $\overline{r}$ is in the lineality space of $K$.  For all $\delta > 0$ the solution $(d_{1} \delta, \ldots, d_{\ell - 1} \delta, 0, \ldots, \delta)$ has positive slack on aggregate constraint $h$. This means at least one constraint indexed by $\supp(\overline{u}^{h})$ must have positive slack. But if $\overline{r}$ is in the lineality space of $K$, the constraint with positive slack must be $-  c_{1} x_{1} -  c_{2} x_{2} - \cdots - c_{n} x_{n}   \ge   0$ and $\overline{r}$ is not in $N.$  Then by Theorem \ref{theorem:unbounded-recession-cone}, $v(\ref{eq:SILP}) = v(\ref{eq:FDSILP})$.

Now assume $\overline{r}$ is not in the lineality space of $K,$ that is $\overline{r} \notin M.$  There are two subcases. Either $\overline{r}$ is in $N$ or $\overline{r}$ is not in $N$.  If $\overline{r}$ is not in $N,$ then again by by Theorem \ref{theorem:unbounded-recession-cone}, $v(\ref{eq:SILP}) = v(\ref{eq:FDSILP})$.  So assume $\overline{r}$ is in $N.$  Then  $\overline{r} \in K$ and $\overline{r} \in N$.  Therefore $\overline{r} \in K \cap N.$ By hypothesis $K \cap N$ is a subspace.  Then by  Remark  \ref{remark:nullspace_K_N},  $\overline{r} \in K \cap N \subseteq M.$  But $\overline{r}$ cannot be $M$  since we have assumed $\overline{r} \notin M.$  This is a contradiction, so this case cannot arise when $K \cap N$ is a subspace.  

Thus, when $K \cap N$ is a subspace, either $v(\ref{eq:SILP}) = v(\ref{eq:FDSILP}) = -\infty$ and there is no duality gap, or  $H_{2}$ must be empty and there is no duality gap by Corollary \ref{cor:bounded-recession-cone}. }\end{proof}

\subsubsection{Finite approximation results}

Consider an instance of \eqref{eq:SILP} and the corresponding finite support dual~\eqref{eq:FDSILP}. 
For any subset $J \subseteq I$, define $\text{SILP}(J)$ as the semi-infinite linear program with only the constraints indexed by $J$ and the same objective function, and $v(J)$  the optimal value of $\text{SILP}(J)$. For example, if $J$ is a finite subset of $I$, $\text{SILP}(J)$ is a finite linear program.

\begin{theorem}\label{theorem:finite-approx} If \eqref{eq:SILP} is feasible, then $v(\text{FDSILP}) = \sup\{v(J) : J \textrm{ is a finite subset of } I\}.$
\end{theorem}

\begin{proof}
We first show that $v(\text{FDSILP}) \leq \sup\{v(J) : J \textrm{ is a finite subset of } I\}$. \old{By hypothesis, \eqref{eq:FDSILP} is feasible. }By hypothesis,  \eqref{eq:SILP} is feasible and this  implies  by Corollary~\ref{cor:dual-optimal-value} that  $v(\text{FDSILP}) = \sup_{h \in I_3} \tilde{b}(h)$. 
Therefore, for every $\epsilon > 0$, there exists a $h^* \in I_3$ such that $v(\text{FDSILP}) - \epsilon \leq \tilde{b}(h^*)$. 
By Lemma~\ref{lemma:feasible-FDSILP}\eqref{item:I3-property}, there exists a $v^{h^*}\in \R^{(I)}$ with support $J^*$ such that $\tilde{b}(h^*) = \langle b, v^{h^*} \rangle = \sum_{i \in J^*}b(i) v^{h^*}(i)$, and $\sum_{i \in J^*}a^k(i)v^{h^*}(i) = c_k$. Since~\eqref{eq:SILP} is feasible, $\text{SILP}(J^*)$ is feasible; let $\bar x$ be {\em any} feasible solution to this finite LP. Thus, 
$$
\begin{array}{rcl}
c^T\bar x & = &\sum_{k=1}^n c_k \bar{x}_k \\
&=& \sum_{k=1}^n(\sum_{i \in J^*}a^k(i)v^{h^*}(i))\bar{x}_k \\
&=& \sum_{i \in J^*}(\sum_{k=1}^n a^k(i)\bar{x}_k) v^{h^*}(i) \\
&\geq &\sum_{i\in J^*} b(i)v^{h^*}(i) \\
& = & \tilde{b}(h^*).
\end{array}
$$

Since this holds for any feasible solution to $\text{SILP}(J^*)$, $v(J^*) \geq \tilde{b}(h^*) \geq v(\text{FDSILP}) - \epsilon$. Thus, for every $\epsilon > 0$, there exists a finite $J^* \subseteq I$ such that $v(J^*) \geq v(\text{FDSILP}) - \epsilon$. Hence,  $v(\text{FDSILP}) \leq \sup\{v(J) : J \textrm{ is a finite subset of } I\}$.

Next we show that $v(\text{FDSILP}) \geq \sup\{v(J) : J \textrm{ is a finite subset of } I\}$. Consider any finite $J^*\subseteq I$. It suffices to show that $v(\text{FDSILP}) \geq v(J^*)$.     If $v(J^*) = -\infty$, then the result is immediate. So assume $v(J^*) > -\infty$. Then $\text{SILP}(J^*)$  is bounded.  Since  \eqref{eq:SILP} is feasible by hypothesis, $\text{SILP}(J^*)$  is also feasible. Then   by Theorem~\ref{theorem:finite-strong-duality}, there exists a $v^* \in \R^{J^*}$ such that $\sum_{i\in J^*}b(i)v^*(i) = v(J^*)$ and $\sum_{i\in J^*}a^k(i)v^*(i) = c_k$. Define $\bar v \in \R^{(I)}$ by $\bar v(i) = v^*(i)$ for $i \in J^*$ and $\bar v(i) = 0$ for $i \not\in J^*$. Thus, $\bar v$ is a feasible solution to~\eqref{eq:FDSILP} with objective value $v(J^*)$. Therefore, $v(\text{FDSILP}) \geq v(J^*)$.
\end{proof}

Theorem \ref{theorem:finite-approx} is used to prove a series of results by Karney~\cite{karney81}. Consider a semi-infinite linear program with countably many constraints, i.e., $I = \N$. For every $n \in \N$, let $P_n$ denote the finite linear program formed using the constraints indexed by $\{1, \ldots, n\}$ and the same objective function. Let $v(P_n)$ denote its optimal value.

\begin{corollary}\label{cor:karney-duffin-karlowitz}
If~\eqref{eq:SILP} is feasible with $I = \N$, then $\lim_{n\to \infty} v(P_n) = v(\text{FDSILP}).$
\end{corollary}

\begin{proof}
Since $\{1, \ldots, n\}$ is a finite subset of $I$, $v(P_n) \leq \sup\{v(J) : J \textrm{ is a finite subset of } I\} = v(\text{FDSILP}) <\infty$ where the equality follows from Theorem~\ref{theorem:finite-approx} and the ``$<$" follows from weak duality since \eqref{eq:SILP} is feasible. Since $v(P_n)$ is a nondecreasing sequence of real numbers bounded above, $\lim_{n\to \infty} v(P_n)$ exists and $\lim_{n\to\infty}v(P_n) \leq v(\text{FDSILP}).$
Next prove that $\lim_{n\to \infty} v(P_n) \geq v(\text{FDSILP})$. Observe that for any finite subset $J^* \subseteq I$ there exists an $n^* \in \N$ such that $J^* \subseteq \{1, \ldots, n^*\}$ and this  implies  $v(P_{n^*}) \geq v(J^*)$. Thus, $\lim_{n\to \infty} v(P_n) \geq \sup\{v(J) : J \textrm{ is a finite subset of } I\}$ = v(\text{FDSILP}) where the equality follows from Theorem~\ref{theorem:finite-approx}.
\end{proof} 

\begin{corollary}[Karney~\cite{karney81} Theorem 2.1]\label{cor:karney81} If the feasible region of~\eqref{eq:SILP} with $I = \N$ is nonempty and bounded, then $\lim_{n\to \infty} v(P_n) = v(\ref{eq:SILP}).$\end{corollary}
\begin{proof} This follows from Theorem~\ref{theorem:bounded-zero-duality-gap} and Corollary~\ref{cor:karney-duffin-karlowitz}.
\end{proof}

\begin{corollary}[Karney~\cite{karney81} Theorem 2.4]\label{cor:karney-theorem-2-4} If $\eqref{eq:SILP}$ with $I = \N$ is feasible and the zero vector  is the unique solution to the system~\eqref{eq:bounded-recession-cone-1}-\eqref{eq:bounded-recession-cone-2}, then $\lim_{n\to \infty} v(P_n) = v(\ref{eq:SILP}).$
\end{corollary}

\begin{proof} 
If the zero vector  is the unique solution to the system~\eqref{eq:bounded-recession-cone-1}-\eqref{eq:bounded-recession-cone-2}, then the recession cone of \eqref{eq:bounded-zero-duality-gap} is $\{0\}$ and  \eqref{eq:bounded-zero-duality-gap} is bounded for any value of $\gamma$ such that \eqref{eq:bounded-zero-duality-gap} is feasible (such a $\gamma$ exists because \eqref{eq:SILP} is feasible). The result then follows from Theorem~\ref{theorem:bounded-zero-duality-gap} and Corollary~\ref{cor:karney-duffin-karlowitz}.\end{proof}

\begin{corollary}[Karney~\cite{karney81} Theorem 2.5] \label{cor:karney-theorem-2-5} Assume \eqref{eq:SILP} with $I = \N$ is feasible and let $r$ be a ray satisfying~\eqref{eq:bounded-recession-cone-1}-\eqref{eq:bounded-recession-cone-2}. If $r$ is not an element of  the null space $N,$  then $\lim_{n\to \infty} v(P_n) = v(\ref{eq:SILP}) = -\infty.$
\end{corollary}

\begin{proof} If  $r \in K$ and $r \notin N,$  then $c^T r <0$.   This implies $v(\ref{eq:SILP}) = -\infty$ and \eqref{eq:FDSILP} is infeasible by weak duality.   Then  $v(\ref{eq:SILP}) = v(\ref{eq:FDSILP}) = -\infty$ and the result  follows from Corollary~\ref{cor:karney-duffin-karlowitz}.\end{proof}

\old{
\begin{corollary}\label{cor:karney-theorem-2-5}[Karney~\cite{karney81} Theorem 2.5]  Assume \eqref{eq:SILP} is feasible and that applying Fourier-Motzkin elimination to~\eqref{eq:bounded-recession-cone-1}-\eqref{eq:bounded-recession-cone-2}  gives~\eqref{eq:defineH1-recession}-\eqref{eq:defineH2-recession}.  If $r$ is a ray defined by~\eqref{eq:define-extreme-ray}, and $r$ is not an element of  the null space $N,$  then $\lim_{n\to \infty} v(P_n) = v(\ref{eq:SILP}) = -\infty.$
\end{corollary}

\begin{proof}  This follows from Theorem~\ref{theorem:unbounded-recession-cone}  and Corollary~\ref{cor:karney-duffin-karlowitz}.\end{proof}
}

\begin{corollary}[Karney~\cite{karney81} Theorem 2.6]\label{cor:karney-theorem-2-6}  If \eqref{eq:SILP} is feasible and   $K \cap N$ is a linear subspace, then $\lim_{n\to \infty} v(P_n) = v(\ref{eq:SILP}).$
\end{corollary}

\begin{proof} 
This follows from Theorem~\ref{theorem:unbounded-recession-cone}  and Corollary~\ref{cor:karney-duffin-karlowitz}.\end{proof}

\end{document}